%% file: ms.tex
\documentclass{article}

\PassOptionsToPackage{numbers, compress}{natbib}
\usepackage[preprint]{neurips_2023}

\usepackage{amsmath}
\usepackage{amssymb}
\usepackage{mathtools}
\usepackage{amsthm}

\usepackage{xcolor}         % colors 
\definecolor{mydarkblue}{rgb}{0,0.08,0.45} 
\definecolor{mydarkgreen}{rgb}{0,0.45,0.08} 
\usepackage[utf8]{inputenc} % allow utf-8 input
\usepackage[T1]{fontenc}    % use 8-bit T1 fonts
\usepackage{hyperref}       % hyperlinks 
\hypersetup{
    colorlinks=true,
    citecolor=mydarkgreen, 
    linkcolor=mydarkblue,
}
\usepackage{url}            % simple URL typesetting
\usepackage{booktabs}       % professional-quality tables
\usepackage{amsfonts}       % blackboard math symbols
\usepackage{nicefrac}       % compact symbols for 1/2, etc.
\usepackage{microtype}      % microtypography 

\usepackage{float}

\usepackage{threeparttable}
\usepackage[ruled]{algorithm2e}

\usepackage[toc,page,header]{appendix}
\usepackage{minitoc} 

\usepackage{chrkroer}
\usepackage[capitalize,noabbrev]{cleveref}

\newcommand{\EE}{\mathbb{E}} 
\newcommand{\gap}{\textnormal{Gap}}
\newcommand{\distt}{\textnormal{dist}} 
\newcommand{\interior}{\textnormal{int}}
\DeclarePairedDelimiterX{\inp}[2]{\langle}{\rangle}{#1, #2}
\DeclarePairedDelimiter{\En}{\lVert}{\rVert} 

\title{Convergence of SVRG-Extragradient for Variational Inequalities: 
Error Bounds and Increasing Iterate Averaging}

\author{%
  Tianlong Nan, Yuan Gao, Christian Kroer
  \\
  Department of Industrial Engineering and Operations Research\\
  Columbia University\\
  New York, NY 10027 \\
  \texttt{\{tn2447,yg2541,ck2945\}@columbia.edu} \\
}

\begin{document}

\graphicspath{ {./plots/} }

\maketitle

\input{abstract.tex}

\doparttoc % Tell to minitoc to generate a toc for the parts
\faketableofcontents % Run a fake tableofcontents command for the partocs 

\input{intro.tex} 
\input{related_work.tex}

% Submissions are limited to \textbf{8 pages} excluding references. 
\input{pre.tex}
\input{algorithm.tex}
\input{error_bounds.tex}
\input{weak_sharpness.tex}

\input{iias.tex} 
\input{experiments.tex}

\bibliographystyle{plainnat}
\bibliography{refs}

%%%%%%%%%%%%%%%%%%%%%%%%%%%%%%%%%%%%%%%%%%%%%%%%%%%%%%%%%%%%%%%%%%%%%%%%%%%%%%%
%%%%%%%%%%%%%%%%%%%%%%%%%%%%%%%%%%%%%%%%%%%%%%%%%%%%%%%%%%%%%%%%%%%%%%%%%%%%%%%
% APPENDIX
%%%%%%%%%%%%%%%%%%%%%%%%%%%%%%%%%%%%%%%%%%%%%%%%%%%%%%%%%%%%%%%%%%%%%%%%%%%%%%%
%%%%%%%%%%%%%%%%%%%%%%%%%%%%%%%%%%%%%%%%%%%%%%%%%%%%%%%%%%%%%%%%%%%%%%%%%%%%%%%
\newpage
\appendix 

% create title for supplementary materials 
\hrule height 4pt 
\vskip 0.25in
\vskip -\parskip%

{\centering \LARGE\bf Supplementary Materials for \\ 
Convergence of SVRG-Extragradient for Variational Inequalities: Error Bounds and Increasing Iterate Averaging \par} 

\vskip 0.29in
\vskip -\parskip
\hrule height 1pt
\vskip 0.09in% 

\addcontentsline{toc}{section}{Appendix} % Add the appendix text to the document TOC
\part{} % Start the appendix part
\parttoc % Insert the appendix TOC 

\input{appendix.tex}

%%%%%%%%%%%%%%%%%%%%%%%%%%%%%%%%%%%%%%%%%%%%%%%%%%%%%%%%%%%%

\end{document}

%% file: abstract.tex
\begin{abstract}
We study the last-iterate convergence of variance reduction methods for extragradient (EG) algorithms for a class of variational inequalities satisfying error-bound conditions. 
Previously, last-iterate linear convergence was only known under strong monotonicity. 
We show that EG algorithms with SVRG-style variance reduction, denoted SVRG-EG, attain last-iterate linear convergence under a general error-bound condition much weaker than strong monotonicity. This condition captures a broad class of non-strongly monotone problems, such as bilinear saddle-point problems commonly encountered in two-player zero-sum Nash equilibrium computation.
Next, we establish linear last-iterate convergence of SVRG-EG with an improved guarantee under the weak sharpness assumption.
Furthermore, motivated by the empirical efficiency of increasing iterate averaging techniques in solving saddle-point problems,
we also establish new convergence results for SVRG-EG with such techniques. 
\end{abstract}

%% file: intro.tex
\section{INTRODUCTION}

% We study the variational inequality (VI) problem: 
% find $z^* \in V$ such that 
% \begin{equation}
%     \inp{ F(z^*)}{ z - z^* } + g(z) - g(z^*) \ge 0, \; \forall\; z \in V, 
%     \tag{VI}
%     \label{prob:vi}
% \end{equation}
% where 
% % $\cZ$ is a convex subset of 
% $V$ is some Euclidean space, $F: V \rightarrow V$ is a monotone operator, and $g$ is a proper convex lower semicontinuous function. %, meaning that $\langle F(z) - F(z'), z-z' \rangle \geq 0\ \forall z,z'\in\cZ$. 

We study the variational inequality (VI) problem: 
find $z^* \in V$ such that 
\begin{equation}
    \inp{ F(z^*)}{ z - z^* } \ge 0, \; \forall\; z \in \cZ, 
    \tag{VI}
    \label{prob:vi}
\end{equation}
where 
% $\cZ$ is a convex subset of 
$V$ is some Euclidean space, $F: V \rightarrow V$ is a monotone operator, and $\cZ$ is a compact convex set in $V$. %, meaning that $\langle F(z) - F(z'), z-z' \rangle \geq 0\ \forall z,z'\in\cZ$. 
% Moreover, we assume that the set of optimal solutions $\cZ^*$ is nonempty. 
A prominent 
use case % application 
of VIs is 
in addressing % the solution of 
convex-concave saddle-point problems (SPPs):
\begin{equation}
    \min_{x \in \mathcal{X}} \max_{y \in \mathcal{Y}} f(x, y),  \tag{SP}
    \label{prob:sp}
\end{equation}
where $\mathcal{X}$ and $\mathcal{Y}$ are compact convex sets, 
and $f$ is a differentiable function with Lipschitz gradients. 
To transform \eqref{prob:sp} into a VI, 
we let $z$ be the concatenation of $x$ and $y$, construct the monotone operator $F(z) = (\nabla f_x(x, y), - \nabla f_y(x, y))$, and 
% $g(z) = 0$ if $x \in \mathcal{X}, y \in \mathcal{Y}$ and $+\infty$ otherwise. 
$\cZ = \mathcal{X} \times \mathcal{Y}$. 
Then, the set of optimal solutions $\cZ^*$ corresponds to the set of saddle-points of \eqref{prob:sp}. 

In this paper, we are particularly motivated by the SPPs that arise from two-player zero-sum games. 
In such games, the set of Nash equilibria is the set of solutions to an SPP, 
% , where the set of Nash equilibria can be described as the set of solutions to an SPP
where $\cX,\cY$ represent the sets of possible mixed strategies for the respective players, 
and $f$ encodes the expected value attained by the second player given a strategy pair $(x,y)$.
Notably, for both \emph{normal-form games} (i.e. matrix games) and \emph{extensive-form games} (EFGs), 
% the sets 
$\cX,\cY$ are compact convex polyhedral sets, and $f$ is a biaffine function. 
This configuration results in a VI characterized by
% Consequently, the resulting VI has 
% % a convex polyhedral $\cZ$, and 
an affine operator $F$ and 
% % characteristic function $g$ of 
a compact convex polyhedral set $\cZ$. 
% % , and we will derive results that leverage this structure. 

In many real-world applications of VIs, 
the operator $F$ often exhibits a finite-sum structure, 
denoted as 
% A common occurrence in real-world problems for VIs is that the operator $F$ has a \emph{finite-sum structure}, meaning that  
$F = F_1 + \cdots + F_N$.
Originally, the finite-sum structure was motivated by machine-learning problems, 
where it reflects % corresponds to 
the empirical risk across $N$ data points. 
Nonetheless, it can also arise from the decomposition of the payoff matrix in a normal-form game, 
% or arise naturally due to the payoff matrix consisting of an expectation, for example due to cards being dealt out by chance in the EFG representation of a poker game~\citep{lanctot2009monte}. 
or from the inherent structure of the payoff matrix based on expected values—as seen when cards are randomly dealt in the EFG representation of a poker game~\citep{lanctot2009monte}. 
Stochastic first-order methods (FOMs) can leverage finite-sum structures to achieve convergence rates comparable to those of deterministic methods, while at % achieving 
a much lower cost per iteration. 
% \tianlong{These methods allow us to solve large-scale VIs and SPPs with limited computing power.}
% Claim that one advantage of stochastic methods such as SVRG is computing large scale problems with limited computing power. 
One such approach is known as stochastic variance reduced gradient (SVRG), which leverages the finite-sum structure when doing variance reduction.

We consider the SVRG extragradient method (SVRG-EG) for VIs, as proposed % which was introduced 
by \citet{alacaoglu2022stochastic}. 
In particular, our focus is primarily on % we mainly study a 
the \emph{loopless} version of SVRG-EG. 
For this algorithm, \citet{alacaoglu2022stochastic} showed that for a monotone VI, the uniform average of the iterates converges to a solution at a rate of $O(1/T)$, where $T$ is the number of iterations. 
Moreover, they showed that, when $\cZ$ is a strongly convex set (or $F$ is strongly monotone, 
% , e.g., saddle-point problems with strongly convex-concave $f$
though they omitted the proof), 
the loopless SVRG-EG achieves last-iterate linear convergence. % in mean square error.
% \tianlong{(strongly convex)} \ck{What do you mean?  A VI being strongly monotone corresponds to strong convexity for functions?}

We study convergence properties of SVRG-EG under more general assumptions. 
% Motivated by problems such as two-player zero-sum games, we study the performance of these algorithms under \emph{error-bound} conditions. 
Firstly, we consider a classical error-bound condition (or simply error bounds). 
This condition generalizes the strong monotonicity assumption and captures a range of practical problems of interest where strong monotonicity does not hold, 
such as two-player zero-sum games and 
% perhaps most prominently in two-player zero-sum games, 
image segmentation~\citep{chambolle2011first}. 
% and more generally for affine VIs with polyhedral decision sets~\citep{tseng1995linear}.
We show that both the loopless SVRG-EG and its \emph{double-loop} variant achieve linear convergence in their last iterates under the error-bound condition. 
% a general error-bound condition that is known to capture strong monotonicity and affine VIs with polyhedral decision sets. 
To the best of our knowledge, our result is the first constant-stepsize stochastic FOM that solves two-player zero-sum games at a linear rate. 
Furthermore, we consider a class of VIs whose solution sets are subject to a \emph{weak sharpness} condition. 
Under this condition, we show that SVRG-EG achieves a last-iterate linear convergence rate which parallels that achieved under the strong monotonicity of $F$. 

In addition to evaluating last iterates, we study \emph{increasing iterate averaging schemes} (IIAS) for 
% the loopless and double-loop 
SVRG-EG. 
Typically, most first-order methods for convex, but not strongly convex, problems are shown to have 
% what is known as 
an \emph{ergodic} convergence rate, meaning that the uniform average of the iterates attains some guarantee of convergence, whereas the individual iterates may not have such a guarantee.
However, the uniform averaging of iterates is very conservative: 
in practice, the early iterates are of very poor quality, but there is rapid improvement in the iterates. 
By averaging uniformly, this improvement is only slowly integrated to the running average of iterates. 
It was shown by \citet{gao2021increasing} that IIAS 
significantly bolsters % leads to much faster 
the practical performance in various FOMs, including the mirror prox method~\citep{nemirovski2004prox}, and the primal-dual method of \citet{chambolle2011first} and its variants~\citep{malitsky2018first,chambolle2016ergodic}, while retaining the same theoretical guarantee of convergence. 
They pinpointed two particularly effective strategies: \emph{linear averaging}, where the weight of iterate $t$ is proportional to $t$, and \emph{quadratic averaging}, where the weight $t$ is proportional to $t^2$. 
% Specifically, they found that \emph{linear averaging}, where iterate $t$ is weighted proportional to $t$, or \emph{quadratic averaging}, where iterate $t$ is weighted proportional to $t^2$, leads to the best performance.
We show that IIAS 
% , more specifically polynomial averaging schemes, 
maintains the same $O(1/T)$ rate of convergence as that of the ergodic scheme presented by \citet{alacaoglu2022stochastic} for SVRG-EG. 

Finally, we study the numerical performance of SVRG-EG algorithms, under uniform averaging, linear averaging, 
% quadratic averaging (can add this later)
and last iterate. 
We contrast that with state-of-the-art methods for solving SPPs, also instantiated with each of these averaging schemes. 
First, we demonstrate that the last iterates of SVRG-EG achieve fast convergence for both matrix games and EFGs. 
Second, we find that it is important to take into account IIAS: the deterministic and SVRG-EG methods have orders-of-magnitude performance differences depending on the averaging scheme. 
% and the right IIAS depends on the game. 
% In general we find that linear averaging always performs vastly better than uniform averaging, and the last iterate sometimes performs even better than linear averaging, but sometimes worse. 
On matrix games we find that SVRG-EG methods with IIAS, or last-iterate, perform substantially better than other methods. On EFGs we find that SVRG-EG performs better than EG in one game and on par in another.
For image segmentation, SVRG-EG with linear and last-iterate again perform substantially better than other methods.

%% file: related_work.tex
\section{RELATED WORK}
\label{sec:related-work}

For VIs and SPPs, EG-type methods are renowned for their strong convergence guarantees in the deterministic full-gradient setting. Convergence guarantees have been established for uniform iterate averaging (e.g.,~\citet{nemirovski2004prox}) and the last iterate (e.g.,~\citet{tseng1995linear}). 
There is also extensive literature on variations of EG that avoid taking two steps or avoid two gradient computations per iteration, and so on (e.g.~\citet{hsieh2019convergence,malitsky2020forward}). 
That said, all of these full-gradient methods require one or more full gradient computations per iteration. 

The rising significance of stochastic methods in large-scale computations has driven interest in stochastic EG methods. 
\citet{juditsky2011solving} studied stochastic Mirror-Prox (which generalizes EG) methods for VIs with compact convex feasible sets, and \citet{mishchenko2020revisiting} improved the algorithm by using a single sample per iteration. 
\citet{hsieh2019convergence} showed stochastic variants of the single-call type EG methods. 
Yet, with a constant stepsize, the iterates of these methods only converge to a neighborhood of the solution set. 
Diminishing stepsizes do ensure convergence, but they empirically slow down the performance. 
% \citet{hsieh2020explore} showed convergence to the solution set by employing double stepsizes, but they only considered the unconstrained case. 
% In all of these stochastic variations of EG, the authors do not apply variance reduction, and this generally forces them to accept sublinear rates of convergence, especially when applied to non-strongly monotone VIs. 

% Moreover, they are generally not able to exploit finite-sum structures in order to speed up convergence. 
% SVRG-type methods alleviate this to some extent by enabling the exploitation of finite-sum structure, thereby achieving a linear rate for strongly monotone VIs in Alacaoglu \& Malitsky. 
% Yet even with SVRG-EG, in the case of non-strongly-monotone VIs, it was not known to attain a linear rate. 
% We show that, indeed it IS possible for SVRG-EG to attain a linear rate, specifically when the VI satisfies a suitable error-bound condition. 

Variance reduction is a powerful mechanism to enhance the convergence rates of stochastic FOMs. A detailed discussion of SVRG and other variance reduction techniques in the context of convex minimization can be found in Appendix~\ref{app:related-work}. 
For solving saddle-point problems, SVRG-style variance reduction was first studied under the strongly-convex-concave setting by \citet{palaniappan2016stochastic}.
Since then, there has been extensive work on similar variance reduction techniques for the more general strongly monotone VIs \citep{beznosikov2022sarah,alkousa2020accelerated,jin2022sharper}. 
However, these works do not easily extend to the setting we focus on, that is, non-strongly convex-concave---especially bilinear---saddle-point problems, a special form of non-strongly monotone VIs. 
It is worth noting, though, that one could apply Nesterov smoothing in order to reduce the monotone setting to the strongly monotone setting, at a small cost in theory~\citep{nesterov2005smooth}. 
An adaptive variant of that approach is given by \citet{jin2022sharper}. 
However, here we are interested in methods such as EG, which avoid smoothing. 

\citet{carmon2019variance} gave, to our knowledge, the first 
theoretical results 
% yielding a theoretical speedup 
for variance reduction in matrix games, a special case of monotone VIs.
\citet{huang2022accelerated} considered monotone and strongly monotone VIs, and also developed SVRG-style methods. 
\citet{han2021lower} gave lower bounds for solving finite-sum saddle-point problems via FOMs. Note that their lower bounds do not contradict our linear-rate results: their analysis relies on non-polyhedral feasible sets such as $\ell_2$-norm balls. 
None of the works mentioned in this paragraph consider the error-bound condition for VIs and SPPs. %, and thus do not attain linear convergence. 
Also, they do not incorporate non-uniform iterate averaging schemes. 
% See~\cref{table:linear-convergence} for a summary of linear convergence results for SVRG-type methods.

% \begin{table}[tb]
%     \caption{Linear convergence rates of SVRG-type methods}
%     \label{table:linear-convergence}
%     \centering
%     \begin{tabular}{c c c c}
%         \toprule
%         % \multicolumn{2}{c}{Part}                   \\
%         % \cmidrule(r){1-2}
%         \shortstack{Strongly convex \\ minimization} & \shortstack{Convex \\ minimization} & \shortstack{Strongly convex-concave \\ saddle-point problems} & \shortstack{Convex-concave \\ saddle-point problems} \\
%         \midrule
%         \citep{xiao2014proximal,kovalev2020don}  & \citep{zhang2022linear} & \citep{alacaoglu2022stochastic} & This work \\
%         $\mathcal{O}( \small( N \hspace{-2pt} + \hspace{-2pt} \frac{L}{\mu} \small) \log{\frac{1}{\epsilon}} )$ &  $\mathcal{O}( ( N \hspace{-2pt} + \hspace{-2pt} \frac{L}{\kappa} ) \log{\frac{1}{\epsilon}} )$ & $\mathcal{O}( ( N \hspace{-2pt} + \hspace{-2pt} \frac{\sqrt{N}L}{\mu} ) \log{\frac{1}{\epsilon}} )$ & $\mathcal{O}\left( (N \hspace{-2pt} + \hspace{-1pt} \bar{C}^2 N L^2) \log{\frac{1}{\epsilon}} \right)$ \\  
%         \bottomrule
%         % \multicolumn{4}{l}{ \; In this table, }
%         \vspace{-8pt}
%     \end{tabular}
%     \begin{tablenotes}
%         \item In this table, $\mu$ is the strong convexity(-concavity) constant and $\kappa$ is the bounded metric subregularity constant. $\bar{C}$ is defined in this paper.
%     \end{tablenotes}
%     \vspace{-8pt}
% \end{table}

\citet{tseng1995linear} established a paradigm for deriving linear convergence for iterative methods in VIs, under the projection-type error-bound condition. 
Since then, error bounds have been commonly used in the analysis of FOMs. 
For more about this, see~\citet{drusvyatskiy2018error} and references therein. 
Recently, error bounds (or equivalent conditions) have also been considered to show linear convergence of SVRG methods without strong convexity, in the convex minimization setting, e.g., by \citet{karimi2016linear,xu2017adaptive,zhang2022linear}. 

Proposed by~\citet{polyak1979sharp}, 
the concept of a sharp solution was generalized by~\citet{burke1993weak} to weak sharpness to include the possibility of multiple solutions. 
\citet{patriksson1993unified} generalized weak sharpness to variational inequalities. 
While weak sharpness is often linked to identifying finite convergence~\citep{marcotte1998weak}, 
we aim to leverage this assumption to determine a specific rate at which iterates approach the solution set. 
In recent work, \citet{applegate2023faster} used sharpness to show linear convergence of first-order primal-dual methods for minimax problems. Note that, their assumption is not generally equivalent to ours, and they analyzed a different, and deterministic, algorithm.

%% file: pre.tex
% \section{Preliminaries}
\section{PRELIMINARIES}

Let $V$ be a finite dimensional vector space and $S$ be a nonempty closed convex set in $V$. 
Let $\inp{\cdot}{\cdot}$ denote the Euclidean inner product and $\En{\cdot}$ denote the Euclidean norm. 
Denote the Euclidean projection of $z$ onto set $S$ as $\Pi_S(z) = \arg\min_{z' \in S} \En{ z' - z }$. 
% More generally, we define proximal operator $\prox_g(z) = \arg\min_{z' \in V}\{ \frac{1}{2}\En{z' - z}^2 + g(z') \}$. 
% The characteristic function of $S$ is denoted as $\delta_S$ such that $\delta_S(z) = 0$ if $z \in S$ and $+\infty$ otherwise. 
% : V \rightarrow \mathbb{R} \cup \{ + \infty \}
% If $g = \delta_S$ then $\prox_{\tau g} = \Pi_S$ for any $\tau > 0$. 
For a matrix $\bA$, we denote its $i$'th row as $\bA_i$ and $j$'th column as $\bA_{\cdot j}$. 
$\bA^\top$ is the transpose of $\bA$. 
$\En{\bA}_F$ is the Frobenius norm and $\En{\bA}_2$ is the operator norm of $\bA$ (induced by $\En{\cdot}_2$). 
% , respectively. 
$\EE_\xi$ represents the expectation with respect to a random variable $\xi$. 
$\cF = \sigma(\cdot)$ denotes $\sigma$-field generated by a set of random variables. 

% The finite-dimensional variational inequality (VI) problem is to determine $z^* \in \mathcal{Z}$ such that 
% \begin{equation}
%      \langle F(z^*),\, z - z^* \rangle + g(z) - g(z^*) \ge 0, \quad \forall\; z \in \mathcal{Z}. 
%     \label{prob:vi}
% \end{equation}
We assume \eqref{prob:vi} satisfies the following assumptions: 

\begin{assumption} 
    The solution set $\mathcal{Z}^*$ is nonempty. 
    \label{asp:1}
\end{assumption} 
\begin{assumption} 
    The subset $\cZ \subset V$ is a compact convex set. 
    \label{asp:2}
\end{assumption} 
\begin{assumption} 
    The operator $F$ is monotone, i.e., $\inp{F(z_1) - F(z_2)}{z_1-z_2} \ge 0,\; \forall z_1, z_2 \in \cZ$. 
    \label{asp:3}
\end{assumption}
\begin{assumption} 
The operator $F$ has a stochastic oracle $F_{\xi}$ that is unbiased, i.e., $\mathbb{E}_\xi[F_{\xi}(z)] = F(z),\; \forall z \in \cZ$ and $L_f$-Lipschitz in mean, i.e., 
$\bbE_\xi[ \En{ F_\xi(z_1) - F_\xi(z_2) }^2 ] \leq L^2 \En{ z_1 - z_2 }^2,\; \forall z_1, z_2 \in \cZ$. 
    \label{asp:4}
\end{assumption}
% [\textbf{A.1}] The solution set $\mathcal{Z}^*$ is nonempty. 

% [\textbf{A.2}] The function $g$ is a proper convex lower semicontinuous (lsc), whose domain is defined as $\dom g = \{ z: g(z) < +\infty \} \subseteq V$. 

% [\textbf{A.3}] The operator $F$ is monotone, i.e., for all $z_1, z_2 \in \dom g$, 
%     \begin{equation}
%         \inp{F(z_1) - F(z_2)}{z_1-z_2} \ge 0. 
%     \end{equation} 
%     % and $\dom F \subseteq \dom g$. 

% [\textbf{A.4}] 
%     The operator $F$ has a stochastic oracle $F_{\xi}$ that is unbiased $\mathbb{E}[F_{\xi}(z)] = F(z)$ and $L$-Lipschitz in mean, i.e., for all $z_1, z_2 \in \dom g$, 
%     \begin{equation}
%         \mathbb{E} \En{ F_\xi(z_1) - F_\xi(z_2) }^2 \le L^2 \En{ z_1 - z_2 }^2. 
%     \end{equation}
    %\ck{how is $\cZ$ related to \eqref{prob:vi} at this point?}
    % with $\proj(\cdot)$ denoting the element of $\mathcal{Z}$ nearest to $z$ in Euclidean norm, i.e., $\proj(z) = \arg\min_{z'\in \mathcal{Z}}\En{z' - z}$ for all $z \in \mathbb{R}^n$. 
While our results are somewhat more general, as in \citet{alacaoglu2022stochastic}, the primary application is the finite-sum setting where $F = F_1+\cdots+F_N$, in which case a natural stochastic oracle is one that samples an index $\xi \in [N]$, where $[N]=\{1,\ldots,N\}$. 
Throughout the paper, when we refer to $N$ it should be thought of as the number of terms in the finite-sum structure, or more generally the time to compute a full gradient relative to the computation of a gradient for a sample $\xi$.

% The constrained saddle-point (SP) problem is of the form: 
% \begin{equation}
%     \min_{x \in \mathcal{X}} \max_{y \in \mathcal{Y}} f(x, y),  \tag{SP}
%     \label{prob:sp}
% \end{equation}
% where $\mathcal{X}$ and $\mathcal{Y}$ are closed convex sets, and $f$ is continuous differentiable function. The set of minimax optimal solutions is denoted by $\mathcal{X}^*$ and for $x^* \in \mathcal{X}^*$, $x^* = \arg\min_{x \in \mathcal{X}} \max_{y \in \mathcal{Y}} f(x, y)$. Similarly, the set of maximin optimal solutions is denoted by $\mathcal{Y}^*$ and for $y^* \in \mathcal{Y}^*$, $y^* = \arg\max_{y \in \mathcal{Y}} \min_{x \in \mathcal{X}} f(x, y)$. It is well-known that $\mathcal{X}^*$ and $\mathcal{Y}^*$ are convex, and any pair $(x^*, y^*) \in \mathcal{X}^* \times \mathcal{Y}^*$ is a Nash equilibrium satisfying $f(x^*, y) \le f(x^*, y^*) \le f(x, y^*)$ for any $(x, y) \in \mathcal{X} \times \mathcal{Y}$. 

% The optimality conditions of \eqref{prob:sp} is captured by a VI with $z = (x, y)$, $g(z)$ is an indicator function that equals $0$ if $z \in \mathcal{X} \times \mathcal{Y}$ and $+\infty$ otherwise. 
% % similarly $z^* = (x^*, y^*)$, $\mathcal{Z}^* = \mathcal{X}^* \times \mathcal{Y}^*$, and 
% For any $z \in \mathcal{Z}$, $F(z) = (\nabla f_x(x, y), - \nabla f_y(x, y))$. 

%% file: algorithm.tex
% \section{Algorithm} 
\section{ALGORITHM}

This paper considers extragradient algorithms with variance reduction, first proposed by \citet{alacaoglu2022stochastic}. We focus primarily on their \emph{loopless} variant.
Most results can be extended to their double-loop variant as well, but we defer the discussion on this to the appendix.
% This paper considers two Euclidean extragradient algorithms with variance reduction, both proposed by \citet{alacaoglu2022stochastic}. 
% Instead of emphasizing the loopless SVRG-EG method as in \citet{alacaoglu2022stochastic}, we first consider the classic double-loop SVRG~\citep{johnson2013accelerating} and then extend to its loopless variant. 
% This is because the linear convergence analysis is simpler for double-loop SVRG due to the ergodic snapshot points. 
% Throughout this paper, we use notation consistent with the double-loop version unless additional explanation, and will explain how the corresponding loopless results work. 
The basic idea in (loopless) SVRG-EG is to maintain some intermittently-updated snapshot point $w_k$,
% (updated once per outer iteration in double-loop, and stochastically with low probability in the loopless SVRG-EG) 
which is updated stochastically with low probability.
A full gradient is computed for $w_k$, and then the stochastic gradient estimates at each iteration use the gradient at the snapshot point combined with the gradient at the sampled index $\xi_k$ to perform variance reduction.
This keeps the cost of most iterations cheap, while the variance of the gradient estimates reduces to $0$ as $k$ increases.
This ensures that SVRG-EG converges with a constant stepsize. 
% By optimizing the tradeoff between the cost of computing a full gradient and the variance reduction it achieves, it is possible to recover a better overall complexity of computing a solution to \eqref{prob:vi} and \eqref{prob:sp}. 
Beyond that, we can set adaptive (greater than $\tau_0$) stepsizes to exploit the feature of reduced variance to improve the convergence rate.  

\begin{algorithm}[tb]
  \KwIn{$p \in (0, 1]$, probability distribution $\cD$, step size $\tau$, $\alpha \in (0, 1)$, $z_0 = w_0$}
  \For{$k = 0, 1, \cdots$}{
    $\bar{z}_k = \alpha z_k + (1 - \alpha) w_k$\; 
    % Select suitable $\tau_k$\; 
    % $z_{k + 1/2} = \prox_{\tau g}(\bar{z}_k - \tau_k F(w_k))$\; 
    $z_{k + 1/2} = \Pi_\cZ(\bar{z}_k - \tau F(w_k))$\; 
    Draw an index $\xi_k$ according to $\cD$\; 
    $\hat{F}(z_{k+1/2}) = F_{\xi_k}(z_{k+1/2}) - F_{\xi_k}(w_k) + F(w_k)$\; 
    $z_{k+1} = \Pi_\cZ( \bar{z}_k - \tau \hat{F}(z_{k+1/2}) )$\; 
    $w_{k+1} = \begin{cases}
      z_{k+1} & \text{with probability } p \\ 
      w_k & \text{otherwise}.  
    \end{cases}$
  }
  \caption{SVRG-Extragradient (SVRG-EG)}
  \label{alg:eg-svrg}
\end{algorithm}

We state
% the double-loop SVRG-EG algorithm (\citep[Algorithm 2]{alacaoglu2022stochastic} with Euclidean distance) in \cref{alg:eg-svrg-double-loop} and 
the loopless SVRG-EG algorithm \citep[Algorithm 1]{alacaoglu2022stochastic} in \cref{alg:eg-svrg}. 
% In both methods, 
We have several important parameters: the probability $p$ to update $w_k$ at each iteration,
the probability distribution $Q$ on the indices $\xi$ which defines the stochastic oracle, a step size $\tau_0$, and the convex combination parameter $\alpha$ which specifies the extent to which the centering point $\bar z_k$ 
% ($\bar z_k^s$ in double-loop and $\bar z_k$ in loopless) 
moves towards the snapshot point $w_k$.
In practice, we run these algorithms with the suggested parameters from \citet{alacaoglu2022stochastic}. 
% For \cref{alg:eg-svrg-double-loop}, we use $K = \frac{N}{2}$, $\alpha = 1 - \frac{1}{K}$ and $\tau = \frac{0.99\sqrt{1-\alpha}}{L}$. For \cref{alg:eg-svrg}, 
That is, we use $p = \frac{2}{N}$, $\alpha=1 - \frac{2}{N}$ and $\tau = \frac{0.99\sqrt{2}}{\sqrt{N}L}$. 

%% file: error_bounds.tex
\section{LINEAR CONVERGENCE UNDER ERROR BOUNDS}
\label{sec:error bounds}

% In \citep[Theorem 2.2 (a)]{tseng1995linear}, author shows the following projection-type error bound holds when $F$ is affine and $X$ is a polyhedral set (we use $c$ instead of $\tau$ to avoid confusion): 
% . This is very useful in deriving linear convergence rate of Euclidean projection methods since it gives a way to measure decreasing distance between each iteration compared to the distance to the optimum. 
There is a variety of possible error-bound conditions in the literature. 
At a high level, they usually state that the distance to the solution set from a feasible point near the solution set can be bounded by some simple function of the feasible point, often the norm of an ``iterative residual'' term at that point. 
They are useful in establishing linear convergence guarantees for first-order methods, since they capture the distance to optimality using a distance-like measure for consecutive iterates. 
% (see \citet{tseng1995linear} and \citet{drusvyatskiy2018error} for examples).

We define the \emph{proximal-gradient mapping} $\mathcal{G}_\tau(z) = \tau^{-1} \left( z - \Pi_\mathcal{Z}\left( z - \tau F(z) \right) \right)$. 
Using $\mathcal{G}_\tau(z)$ as the residual function, we assume a projection-type error-bound condition \citep[Equation (5)]{tseng1995linear}: 

% We assume ($\ell_2$ norm) projection-type error-bound condition \citep[Eq. (5)]{tseng1995linear}
% with residual function of \emph{proximal-gradient mapping}: 
% \vspace{-4mm} 
\begin{assumption}[Error-Bound Condition] 
% (Error-bound condition) 
% \label{assumption:error-bound}
% The function $g$ is a characteristic function whose domain is $\mathcal{Z} = \dom g$. 
The operator $F$ satisfies: 
given some $\tau \in (0, +\infty)$, we have 
\begin{equation*}
    \textnormal{dist}(z, \mathcal{Z}^*) \le C_0 \En{\mathcal{G}_\tau(z)}, \ \forall\; z \in \mathcal{Z} \text{ s.t. }\En{\mathcal{G}_\tau(z)} \le \epsilon_0, 
    \label{eq:error-bound-assumption}
\end{equation*}
%with $\En{R_\tau(z)} \le \epsilon_0$ 
for some $C_0, \epsilon_0 > 0$, where we let 
$\textnormal{dist}(z, \mathcal{Z}^*) = \min_{z^* \in \mathcal{Z}^*} \En{z - z^*}.$ 
    % and 
    % \begin{equation}
    %     \mathcal{G}_\tau(z) = \tau^{-1} \left( z - \Pi_\mathcal{Z}\left( z - \tau F(z) \right) \right). 
    % \end{equation}
% Here, we apply notation $\mathcal{G}$ to make $C_0, \epsilon_0$ $\tau$-invariant. 
\label{asp:ebc}
\end{assumption} 

Note that, $C_0, \epsilon_0$ are $\tau$-invariant in this assumption. 
% We first introduce the following two error bounds lemma, which are results derived from error bounds \citep[Theorem 2.2]{tseng1995linear} (combined with the algorithm). 
% \ck{how is $\cZ$ in the above related to \eqref{prob:vi}? Currently a different $\cZ$ from the one below}

In convex optimization and VI literature, the error-bound condition is often stated as a local property within a neighborhood around the optimal set.
However, for problems with a bounded feasible set, we can easily extend it to the entire feasible set, as stated in the following lemma.
% Also, the ``minimum'' in Assumption~\ref{assumption:error-bound} implies that the bound only applies to the closet point to $z$ in $\mathcal{Z}^*$.
% Based on this, We state the following lemma. 
\begin{lemma} 
    If Assumption \textnormal{Error-Bound Condition} holds, 
    % and $\mathcal{Z}$ is bounded, 
    % Let $\cZ$ be a bounded set and $z\in \cZ$.
    % If $F$ satisfies assumption $\textbf{A.5}$, 
    then for any $z \in \cZ$ and $\tau > 0$, we have 
    % there exists a $z^* = \Pi_{\mathcal{Z}^*}(\bar{z}_k) \in \mathcal{Z}^*$ such that 
    \begin{equation}
        \En{ z - \Pi_{\mathcal{Z}^*}(z) } \le C \En{ z - \Pi_\mathcal{Z}(z - \tau F(z)) }, 
        \label{eq:eb}
    \end{equation}
    where $\bar{C} = \max\left\{C_0, {D(\mathcal{Z})}/{\epsilon_0} \right\}$ and $C = {\bar{C}}/{\tau}$. Here, 
    % $C_0$ and $\epsilon_0$ are error bound constants (when iterate changes are small), 
    $D(\mathcal{Z})$ is the maximal diameter of the compact set $\mathcal{Z}$. 
    \label{lemma:eb-1}
\end{lemma}

Note that for each $z$, \eqref{eq:eb} only holds for the solution $\Pi_{\cZ^*}(z)$ (closest to $z$) and, in general, different $z$ yield different solutions $\Pi_{\mathcal{Z}^*}(z)$.
This is one of the most challenging issues in extending the proof of linear convergence under strong monotonicity from \citet{alacaoglu2022stochastic} to the more relaxed error-bound condition \eqref{eq:eb}, because 
% \tianlong{(edited)} 
the strong monotonicity holds for any pair of feasible points. 
% their proof relies heavily on the uniqueness of the optimal solution $z^*$ and the ``prox inequality'' (which holds for any pair of feasible points) for a strongly monotone operator. 
% (They did not assume the uniqueness of the optimal solution $z^*$, instead, they do not need that because their inequalities are valid for all feasible points)

% \subsection{Convergence with Fixed Stepsize}

For our convergence analysis, we define a Lyapunov function combining both iterate point $z_k$ and snapshot point $w_k$ with a weight parameter $\theta \in [0, 1]$: 
$
    \Phi^\theta_{z}(z_k, w_k) = \theta \En{ z_k - z }^2 + (1 - \theta) \En{ w_k - z }^2
    % \label{def:phi}
$
for any $z \in \mathcal{Z}$ and a corresponding projection onto the solution set 
$
    \Pi^\theta_{\mathcal{Z}^*}(z_k, w_k) = \mathrm{argmin}_{z^* \in \mathcal{Z}^*} \Phi^\theta_{z^*}(z_k, w_k) 
    % \label{def:pi}
$.
We define the distance from the iterate $(z_k, w_k)$ to the solution set $\mathcal{Z}^*$ as 
\begin{equation}
    \text{dist}^\theta_{\mathcal{Z}^*}(z_k, w_k) = \min\nolimits_{z^* \in \mathcal{Z}^*} \Phi^\theta_{z^*}(z_k, w_k). 
    \label{eq:theta-Z*-distance}
\end{equation}
% We will defer the specific choice of $\theta$ to later, but note that it will be crucial in completing the proofs.
% To simplify notation, we denote $\Phi^\theta_{\Pi_{\mathcal{Z}^*}}(z_k, w_k)$ as a simplified version of $\Phi^\theta_{\Pi^\theta_{\mathcal{Z}^*}(z_k, w_k)}(z_k, w_k)$. 
% We need the following lemma, which is a result derived from error bounds (combined with the algorithm). 
Using the above notation and a sequence of auxiliary iterates $\tilde{z}_{k+1/2} = \Pi_\cZ(\bar{z}_k - \tau F(\bar{z}_k))$, we can derive another error-bound condition under the weighted distance \eqref{eq:theta-Z*-distance}, formalized in the following lemma. 
% It is the key step in the linear convergence proof.
% It is stated using the double-loop SVRG-EG subscripts $s$ and $k$, but it also holds for the loopless version.
\begin{lemma} 
    Assume Assumptions~\ref{asp:1}-\ref{asp:4} and \textnormal{Error-Bound Condition} hold. 
    Let $\alpha \in [0, 1), \tau = \gamma \frac{\sqrt{1-\alpha}}{L}$, for $\gamma \in (0, 1)$, and $\theta \in [0, 1]$.
    Then, for each $k\geq 1$,~\cref{alg:eg-svrg} ensures
    % \begin{align}
    %     & \frac{1}{12 C^2 (1 - \alpha)} \textnormal{dist}^\theta_{\mathcal{Z}^*}(z_k, w_k) \nonumber \\ 
    %     \leq & \frac{2 - \alpha}{\gamma^2 (1-\alpha)} \En{ z_k - z_{k+1/2} }^2 + \En{ w_k - z_{k+1/2} }^2, 
    %     \label{eq:eb-in-algo}
    % \end{align} 
    \begin{align}
        \tau^2 \frac{\alpha L^2}{6 \bar{C}^2 \alpha L^2 + 8} \textnormal{dist}^\theta_{\mathcal{Z}^*}(z_k, w_k) \leq \alpha \En{ z_k - z_{k+1/2} }^2 + (1 - \alpha) \En{ w_k - z_{k+1/2} }^2, 
        \label{eq:eb-in-algo}
    \end{align} 
    where $\bar{C}$ is defined in \cref{lemma:eb-1}. 
    % The same statement holds for the iterates $z_k,w_k, z_{k+1/2}$ of
    % the loopless \cref{alg:eg-svrg}.
    \label{lemma:eb-in-alg}
\end{lemma}

\cref{lemma:eb-in-alg} paves the way for our main linear convergence theorem. 
We state the result here and give a proof sketch;
% We first state the result for the double-loop version and give a proof sketch; 
the full proof is in \cref{subsec:app-linear-convergence-for-loopless}. 
See similar results for the double-loop version of SVRG-EG in \cref{subsec:app-linear-convergence-for-double-loop}.

\begin{theorem} 
    Assume Assumptions~\ref{asp:1}-\ref{asp:4} and \textnormal{Error-Bound Condition} hold. 
    Let $\{ z_k, w_k \}_{k \in \bbN_+}$ be iterates generated by~\cref{alg:eg-svrg} with $p \in (0, 1]$, $\alpha \in [0, 1), \tau = \gamma \frac{\sqrt{1-\alpha}}{L}$, for $\gamma \in (0, 1)$. 
    Then, let $\theta = \frac{\alpha}{\alpha + \frac{1 - \alpha}{p}}$, we have 
    \begin{equation}
        \bbE\left[ \textnormal{dist}^\theta_{\mathcal{Z}^*}(z_{k+1}, w_{k+1}) \right] \le \rho \bbE\left[ \textnormal{dist}^\theta_{\mathcal{Z}^*}(z_k, w_k) \right], 
        \label{eq:linear-convergence-for-loopless}
    \end{equation}
    where $\rho = 1 - \tau^2 \frac{(1 - \gamma)\alpha L^2}{(4 \bar{C}^2 \alpha L^2 + 8)(\alpha + \frac{1 - \alpha}{p})}$. 
    \label{thm:linear-convergence-for-loopless}
\end{theorem}
% \begin{theorem} 
%     Assume Assumptions $1$-$5$ hold. Let $\alpha$, $\tau$ be defined as in \cref{lemma:eb-in-alg} and $\gamma \in [\frac{4}{5}, 1)$, then in ~\cref{alg:eg-svrg}, 
%     there exists an $\alpha' \in (\alpha, \frac{1 + \alpha}{2})$ 
%     such that 
%     \begin{equation}
%     \alpha' - \alpha = \frac{(1-\gamma)\beta}{12C^2} \left( \alpha' - \frac{p\alpha'}{p\alpha' + 1 - \alpha'} \right), 
%     \end{equation}
%     where $\beta \in (0, 1)$. 
%     Then, 
%     \begin{equation}
%         \mathbb{E}\left[ \textnormal{dist}^\theta_{\mathcal{Z}^*}(z_k, w_k) \right] \le \left( \frac{1}{1 + \frac{pc}{p + 2(1 - \alpha')}} \right)^k \textnormal{dist}^\theta_{\mathcal{Z}^*}(z_0, w_0), 
%     \end{equation}
%     where $\theta = \frac{p\alpha'}{p\alpha' + 1 - \alpha'}$ and $c$ is given by the minimum of $\frac{(1-\beta)(1-\gamma)(p + 2(1-\alpha'))}{4(1-p)}$ and $\frac{(1-\gamma)\beta}{12C^2 - (1-\gamma)\beta}$. 
%     \label{thm:linear-convergence-for-loopless}
% \end{theorem}

% We note that our result for loopless SVRG-EG has a more attractive linear convergence than the double-loop method: the exponent is the iteration $k$, which is, in a sense, equivalent to the \emph{inner} iteration count for double-loop SVRG-EG. Thus we get a much nicer convergence rate dependence on the number of stochastic gradient evaluations.

\paragraph*{Proof sketch of~\cref{thm:linear-convergence-for-loopless}.} 
% The basic idea is to show iteration decrease guarantees some distance measure decreases with a less than $1$ ratio. 
% The basic idea is similar to the proof of \cref{thm:linear-convergence-for-double-loop} 
% except we must handle the random update of snapshots with a pre-specified probability $p$ instead of taking averages. This change leads to a different $\theta(\alpha')$, where the choice of $\alpha'$ is more complicated. 

We start from a crucial iteration decrease inequality in \citet[Eq. 10]{alacaoglu2022stochastic} without dropping $\En{z_k - z_{k+1/2}}^2$ term: 
\begin{align}
    \bbE_k {\En{ z_{k + 1} - z^* }}^2 \leq& \alpha {\En{ z_k - z^* }}^2 + (1 - \alpha) {\En{ w_k - z^* }}^2 - (1 - \alpha)(1 - \gamma) {\En{ z_{k + 1/2} - w_k }}^2 \nonumber \\ 
    & \hspace{80pt} - (1 - \gamma) \bbE_k{\En{ z_{k + 1} - z_{k + 1/2} }}^2 - \alpha {\En{ z_{k+1/2} - z_k }}^2 \nonumber 
\end{align} 
for any $z^* \in \mathcal{Z}^*$. In the proof, $\EE_k[\cdot]$ and $\EE_{k+1/2}[\cdot]$ are the expectations w.r.t. $\mathcal{F}_k = \sigma(z_0, w_0, \ldots, z_k, w_k)$ and $\mathcal{F}_{k+1/2} = \sigma(z_0, w_0, \ldots, z_k, w_k, z_{k+1})$, respectively. 
$\EE[\cdot]$ is the total expectation w.r.t. $\mathcal{F}_0$. 

After applying~\cref{lemma:eb-in-alg} with a specific weight $\theta = {\alpha}/({\alpha + \frac{1 - \alpha}{p}})$ at each iteration we partially offset the terms (on the left-hand side) measuring the distance from the current iterates to the solution set. 
% The term $\alpha\En{z_k - z_{k+1/2}}^2$ is used to ensure negativity of remaining residual terms.
% % Here, we would need to keep the negative terms for later proof. 
% Rearranging and normalizing the inequality, we obtain that
% \begin{equation}
%     (1 + \hat{\delta})\mathbb{E}_k{\En{ z_{k + 1} - z^*_k }}^2 \leq \; \alpha' {\En{ z_k - z^*_k }}^2 + (1 - \alpha') {\En{ w_k - z^*_k }}^2 - e', 
% \end{equation}
% where $\hat{\delta} > 0$, $\alpha'$ is a parameter determining $\theta$ 
% and $e' > 0$ represents residuals. 
% which is similar to the key inequality in the proof of \cite[Theorem 4.9]{alacaoglu2022stochastic}. Using similar techniques, we can finish the remaining proof. 

By using 
$
    \mathbb{E}_{k+1/2} \En{ w_{k+1} - z }^2 = p \En{ z_{k+1} - z }^2 + (1 - p) \En{ w_k - z }^2
$
for any $z \in \mathcal{Z}$ to handle the random update of snapshots, 
we transform the distance measure to the form of the Lyapunov function. 
After taking total expectation and applying the tower property, it yields that 
\begin{align*}
    \mathbb{E}\Big[ \alpha \En{ z_{k+1} - z^*_k }^2 + \frac{1 - \alpha}{p}\En{ w_{k+1} - z^*_k }^2 \Big] \leq \rho \mathbb{E} \Big[\alpha \En{ z_k - z^*_k }^2 + \frac{1 - \alpha}{p}\En{ w_k - z^*_k }^2 \Big]. 
    % \label{eq:iteration-proportional-decrease}
\end{align*}
% This means we need to carefully choose $\theta(\alpha')$ at the beginning. 
Finally, we set $z^*_k = \Pi^{\theta}_{\cZ^*}(z_k, w_k)$ and leverage the definition of projection based on our Lyapunov function to obtain that~\cref{eq:linear-convergence-for-loopless}. 

\paragraph*{Remark.}
It is vital to determine the right $z^*_k$ and $\theta$ because our VI setting allows multiple solutions. 
Note that the error-bound assumption only applies to the specific solution minimizing the distance to some feasible point. 
Unlike \citet{alacaoglu2022stochastic}, we cannot use one fixed solution when we sum over a sequence of indices.
Instead, we need to pick $z^*_k = \Pi_{\mathcal{Z}^*}^\theta(z_k, w_k)$.
% , the solution closest to the iterate $(z_k, w_k)$ for each iteration $k$.
% To derive a recurrence relation~\eqref{eq:recurrence} across iterations, we need to determine $\theta$ such that the right-hand side of~\cref{eq:iteration-proportional-decrease} is exactly our Lyapunov function.
% the specific coefficients as well.
% This requires us to choose a specific $\theta$ value. 

Based on the nature of randomized methods, 
we say a randomized algorithm reaches \emph{$\epsilon$-accuracy} if the expected distance to the solution set is under $\epsilon$. 
We show the time complexity results of the algorithms following this convention. 
% We denote $\tau$-invariant error-bound constant as $\bar{C} = \left\{ C_0, {D(\mathcal{Z})}/{\epsilon_0} \right\}$. 

\begin{corollary}
    Set $p=\frac{2}{N}$, $\alpha = 1 - \frac{2}{N}$ and $\tau = \frac{0.99\sqrt{2}}{\sqrt{N}L}$ in \cref{alg:eg-svrg}. 
    Then, the time complexity to reach $\epsilon$-accuracy is $\mathcal{O}\left( (N + \bar{C}^2 N L^2) \log{\frac{1}{\epsilon}} \right)$. 
    \label{crl:linear-convergence-for-loopless}
\end{corollary} 

Though, with the fixed stepsize, 
we reach the same time complexity as the deterministic extragradient method 
(which needs a $O(N)$ time complexity because of its full gradient computing per iteration), 
SVRG-EG converges much faster numerically. 

Note that the convergence rate per iteration is affected by the small fixed stepsizes we used in analysis, 
while in practice we can usually use a much larger stepsize. 

%% file: weak_sharpness.tex
\section{LINEAR CONVERGENCE UNDER WEAK SHARPNESS} 

% Weak sharpness can be viewed as an error-bound condition. 
% Both sharpness and strong convexity can be viewed as error bound conditions with different parameters [73]. 

% based on the classical notion of sharp solution set proposed by~\citet{polyak1979sharp}, which is a geometric definition for convex minimization problems.
% This notion was later generalized by~\citet{burke1993weak} to weak sharpness to allow multiple optimal solutions, and further generalized by \citet{patriksson1993unified} for analyzing variational inequalities.
% Though the weak sharpness condition was often connected with identify finite convergence~\citep{marcotte1998weak}, it has recently been used to establish linear convergence of first-order primal-dual methods on minimax problem~\citep{applegate2023faster}.

% Weak sharpness can be viewed as an error-bound condition. 
% Both sharpness and strong convexity can be viewed as error bound conditions with different parameters [73]. 

In this section, we consider the weak sharpness condition for (solution sets of) VIs and show that it is sufficient for the linear convergence of SVRG-EG.
As discussed in \cref{sec:related-work}, this condition extends the classical geometric definition for convex minimization problems.
To state these conditions, we need additional notation.
Denote $\cZ^\circ := \left\{ v \in V: \inp{v}{z} \leq 0 \quad \forall\; z \in \cZ \right\}$ as the polar set of $\cZ \subseteq V$. 
If $\cZ$ is a convex set and $z \in \cZ$, the normal cone of $\cZ$ at $z$ is $N_{\cZ}(z) := \left\{ v \in V: \inp{v}{w - z} \leq 0 \quad \forall\; w \in \cZ \right\}$ and the tangent cone of $\cZ$ at $z$ is $T_{\cZ}(z) := [N_{{\cZ}(z)}]^\circ$. 
Denote $\interior(\cZ)$ as the interior of the set $\cZ$.
Let $\mathcal{Z}^*$ be the solution set of~\ref{prob:vi}, 
the weak sharpness condition of $\mathcal{Z}^*$ is defined as follows \citep{patriksson1993unified}: 
% (again, we consider a bounded feasible set) 

\begin{assumption}[Weak Sharpness] 
% The function $g$ is a characteristic function whose domain is a bounded set $\mathcal{Z} = \dom g$. 
The operator $F$ satisfies: 
\begin{equation}
    - F(z^*) \in \textnormal{int}\left( \bigcap_{z' \in \mathcal{Z}^*} \big[\, T_{\mathcal{Z}}(z') \cap N_{\mathcal{Z}^*}(z') \,\big]^\circ \right),  
    \label{eq:weak-sharpness-VI}
\end{equation}
for all $z^* \in \cZ^*$. 
\label{asp:ws}
\end{assumption} 

Here, we give a simple example satisfying weak sharpness. 
\paragraph{Example.} 
Consider $x = (x_1, x_2) \in \mathbb{R}^2$ and a monotone operator $F(x) = ( \frac{x_1 + x_2}{2} -2 ,\, \frac{x_1 + x_2}{2} - 2 )$ and $\mathcal{X} = \{ x \,\vert\, 0 \leq x \leq 1, x_1 + x_2 \leq \frac{3}{2} \}$. 
One can verify that the corresponding VI has a solution set 
$\mathcal{X}^* = \{ x \,\vert\, \frac{1}{2} \leq x_1 \leq 1, x_1 + x_2 = \frac{3}{2} \}$. 
For $x^* \in \mathcal{X}^* \setminus \{ (\frac{1}{2}, 1), (1, \frac{1}{2}) \}$, $T_\mathcal{X}(x^*) = \{ x \,\vert\, x_1 + x_2 \leq 0 \}$ and $N_{\mathcal{X}^*}(x^*) = \{ x \,\vert\, x_1 = x_2 \}$, thus $[T_\mathcal{X}(x^*) \cap N_{\mathcal{X}^*}(x^*)]^\circ = \{ x \,\vert\, x_1 + x_2 \geq 0 \}$. 
For $x^* = (\frac{1}{2}, 1)$, $T_\mathcal{X}(x^*) = \{ x \,\vert\, x_2 \leq 0, x_1 + x_2 \leq 0 \}$ and $N_{\mathcal{X}^*}(x^*) = \{ x \,\vert\, x_2 - x_1 \geq 0 \}$, thus $[T_\mathcal{X}(x^*) \cap N_{\mathcal{X}^*}(x^*)]^\circ = \{ x \,\vert\, x_2 \geq 0, x_1 + x_2 \geq 0 \}$. 
Similarly, for $x^* = (1, \frac{1}{2})$, $[T_\mathcal{X}(x^*) \cap N_{\mathcal{X}^*}(x^*)]^\circ = \{ x \,\vert\, x_1 \geq 0, x_1 + x_2 \geq 0 \}$. 
Hence, $\bigcap_{x^* \in \mathcal{X}^*} [T_\mathcal{X}(x^*) \cap N_{\mathcal{X}^*}(x^*)]^\circ = \mathbb{R}^2_+$. 
Obviously, $-F(x^*) = - (\frac{5}{4}, \frac{5}{4}) \in \textnormal{int}( \bigcap_{x^* \in \mathcal{X}^*} [T_\mathcal{X}(x^*) \cap N_{\mathcal{X}^*}(x^*)]^\circ )$. 
Therefore, the corresponding VI satisfies weak sharpness. 
Note that this VI does not satisfy strong monotonicity since $\langle F(x) - F(y), x - y \rangle = (x_1 - y_1 + x_2 - y_2)^2 / 2$. 

Under this assumption, we have the following property, which is essential in our proof. 
\begin{lemma}
    If the solution set of the VI satisfies \textnormal{Weak Sharpness}, then there exists a $\mu$ such that 
    \begin{equation}
        \langle F(z^*),\, z - \Pi_{\cZ^*}(z) \rangle \geq \mu \En{ z - \Pi_{\cZ^*}(z) }, 
        \label{eq:weak-sharpness-error-bound} 
    \end{equation}
    for any $z \in \mathcal{Z}, z^* \in \mathcal{Z}^*$. 
    \label{lem:rsi}
\end{lemma} 
Note that this can be interpreted as an error-bound condition which says that the distance between $z$ and $\cZ^*$ is bounded by some residual term. 
We then leverage this property of weak sharpness to derive the linear convergence rate and the resulting logarithmic time complexity of SVRG-EG. 
See the proof in~\cref{app:linear-convergence-under-weak-sharpness}. 
\begin{theorem}
    Let Assumptions~\ref{asp:1}-\ref{asp:4} and \textnormal{Weak Sharpness} hold. 
    Let $\{ z_k, w_k \}_{k \in \bbN_+}$ be iterates generated by~\cref{alg:eg-svrg} with $p \in (0, 1]$, $\alpha \in [\frac{1}{2}, 1), \tau = \gamma \frac{\sqrt{1-\alpha}}{L}$, for $\gamma \in (0, 1)$. 
    Then, let $\theta = \frac{\alpha}{\alpha + \frac{1 - \alpha}{p}}$, we have 
    \begin{equation}
        \mathbb{E}_k\left[ \textnormal{dist}^\theta_{\mathcal{Z}^*}(z_{k+1}, w_{k+1}) \right] \leq \rho^{\textnormal{WS}} \textnormal{dist}^\theta_{\mathcal{Z}^*}(z_k, w_k), 
    \end{equation}
    where $\rho^{\textnormal{WS}} = 1 - \frac{c}{2(\alpha + \frac{1 - \alpha}{p})}$ and $c = \min\{ \tau \frac{\mu}{D(\cZ)},\, (1 - \gamma)(1 - \alpha) \}$. 
    \label{thm:linear-convergence-for-loopless-weak-sharpness}
\end{theorem}

\begin{corollary}
    Set $p=\frac{2}{N}$, $\alpha = 1 - \frac{2}{N}$ and $\tau = \frac{0.99\sqrt{2}}{\sqrt{N}L}$ in \cref{alg:eg-svrg}. 
    Then, the time complexity to reach $\epsilon$-accuracy is $\mathcal{O}\left( (N + \frac{\sqrt{N}L}{\mu}) \log{\frac{1}{\epsilon}} \right)$. 
    \label{crl:linear-convergence-for-loopless-weak-sharpness}
\end{corollary}

\paragraph*{Remark.} 
Though weak sharpness results in a better dependence on $N$ in the convergence rate compared to that of the deterministic EG method in~\citet{tseng1995linear}, it is more restrictive than the error-bound condition. In particular, it does \emph{not} hold for bilinear saddle-point problems in general. 
% Specifically, this requires $F$ to be \emph{pseudomonotone+} (for any $z_1, z_2 \in \cZ$, (1) $\inp{F(z_1)}{z_2 - z_1} \geq 0 \implies \inp{F(z_2)}{z_2 - z_1} \geq 0$; (2)$\inp{F(z_1)}{z_2 - z_1} \geq 0$ and $\inp{F(z_2)}{z_2 - z_1} = 0$ $\implies$ $F(z_1) = F(z_2)$). 

% \paragraph*{\tianlong{Tianlong's todo: add some examples for weak sharpness of variational inequalities. }} 

%% file: iias.tex
% \section{Increasing iterate averaging schemes} 
\section{INCREASING ITERATE AVERAGING SCHEMES} 

% As discussed in \cite{gao2021increasing}, increasing iterate averaging schemes (IIAS) often converge much faster than the uniform averages. 
The main convergence guarantees given in \citet{alacaoglu2022stochastic} are for the uniform averages of iterates.
However, empirically the uniform average of iterates rarely performs well.
In contrast, as shown in \citet{gao2021increasing}, when solving saddle-point problems using first-order algorithms such as the deterministic EG method, increasingly-weighted averages often outperform uniform averages numerically while having the same theoretical convergence guarantees.
% for SVRG-EG with IIAS. 

Motivated by such empirical evidence, we show that~\cref{alg:eg-svrg} preserves the $O(1/T)$ rate of convergence when using increasing iterate averaging.
In particular, we establish a $O(1/T)$ convergence rate for SVRG-EG with polynomial (e.g. uniform, linear, quadratic or cubic) weighted iterate averaging.
We also provide the same results for double-loop SVRG-EG, which can be found in Appendix~\ref{subsec:app-iias-double-loop}.

Here, we consider the generalized variational inequality problem: 
to find $z^* \in V$ such that 
\begin{equation}
    \inp{ F(z^*)}{ z - z^* } + g(z) - g(z^*) \geq 0, \; \forall\; z \in V, 
    \tag{GVI} 
    \label{prob:generalized-vi}
\end{equation}
which is subject to the following assumptions: 
% $\cZ$ is a convex subset of 
% $V$ is some Euclidean space, $F$ is a monotone operator, and $g$ is a proper convex lower semicontinuous function. 
% We summarized the rest of assumptions as follows: 
\begin{assumption} 
    The solution set of \ref{prob:generalized-vi} is nonempty. 
    \label{asp:7} 
\end{assumption} 
\begin{assumption} 
    The function $g$ is a proper convex lower semicontinuous function. 
    \label{asp:8} 
\end{assumption} 
\begin{assumption} 
    The operator $F$ is monotone, i.e., $\inp{F(z_1) - F(z_2)}{z_1-z_2} \ge 0,\; \forall z_1, z_2 \in \textnormal{dom}\, g$. 
    \label{asp:9} 
\end{assumption}
\begin{assumption} 
The operator $F$ has a stochastic oracle $F_{\xi}$ that is unbiased, i.e., $\mathbb{E}_\xi[F_{\xi}(z)] = F(z),\; \forall z \in V$ and $L$-Lipschitz in mean, i.e., 
$\bbE_\xi[ \En{ F_\xi(z_1) - F_\xi(z_2) }^2 ] \leq L^2 \En{ z_1 - z_2 }^2,\; \forall z_1, z_2 \in V$. 
    \label{asp:10} 
\end{assumption}

The classical~\ref{prob:vi} can be reduced from~\ref{prob:generalized-vi} by letting $g = 0$ if $z \in \cZ$ and $+\infty$ otherwise. 
Correspondingly, using 
% By replacing $\Pi_\cZ$ with 
$$\textnormal{Prox}_{\tau g}(z) = \arg\min_{z' \in V}\left\{ \frac{1}{2}\En{z' - z}^2 + \tau g(z') \right\}$$ 
to replace $\Pi_\cZ(z)$ in~\cref{alg:eg-svrg}, we can adapt SVRG-EG to solve~\ref{prob:generalized-vi}.
As is typical in the analysis of the convergence of algorithms for VIs, we chose the following gap function as the convergence measure:
$
\gap(w) = \max_{z \in \cC} \left\{ \inp*{F(z)}{w - z} + g(w) - g(z) \right\}
$.
% which is well-defined, convex, and equal to $0$ when $w$ is a solution to \eqref{prob:vi}. 
Here, $\cC \subseteq V$ is a compact subset, which allows possible unboundedness of $\text{dom}\, g$ (see \citet[Lemma 1]{nesterov2007dual}).
The gap function is well-defined, convex, and equal to $0$ when $w$ is a solution to \eqref{prob:vi}.
% We define another Lyapunov function 
% \begin{equation}
%     \tilde{\Phi}_k = \alpha \En{z^s_k - z^*}^2 + \frac{1 - \alpha}{p} \En{w^s - z^*}^2. 
% \end{equation}

To analyze the convergence of SVRG-EG under increasing iterate averaging, we need to extend \citet[Lemmas 2.2 \& 2.4]{alacaoglu2022stochastic} to handle non-uniform iterate weights.
% Here, we show lemmas for the loopless version for simplicity. See \cref{subsec:app-iias-lemmas} for double-loop version lemmas. 
% \ge \max_{z, z' \in \mathcal{Z}} {\En{ z - z' }}^2$
\begin{lemma}
    Assume Assumptions~\ref{asp:7}-\ref{asp:10} hold. 
    Let $z_k, w_k, z_{k+1/2}$ be iterates generated by loopless SVRG-EG with $p \in (0, 1]$, $\alpha \in [0, 1)$, $\tau = \frac{\sqrt{1-\alpha}}{L} \gamma$, for $\gamma \in (0, 1)$. Then, for $q \in \{0, 1, 2, \ldots\}$ and any solution $z^*$ to~\ref{prob:generalized-vi}, we have 
    \begin{align} 
        \sum_{k=0}^{K-1} k^q ( \mathbb{E}_k{\En{ z_{k + 1} - z_{k + 1/2} }}^2 + (1-\alpha){\En{ z_{k + 1/2} - w_k }}^2 ) \leq K^q \frac{ ( \alpha + \frac{1 - \alpha}{p} ) }{1 - \gamma} \En{z_0 - z^*}^2. 
    \end{align}  
    \vspace{-8pt}
    \label{lemma:poly-sum-bound}
\end{lemma}
% \paragraph*{Remark.} Note that $\alpha < 1$ is important in variation reduced SVRG-EG methods, and essentially leads to variation reduction. 
It is worth mentioning that here $\EE\En{ z_{k + 1/2} - w_k }^2 \rightarrow 0$ as $k \rightarrow \infty$, 
which implies a reduction in variance, 
that is, $\EE\En{\hat{F}(z_{k+1/2}) - F(z_{k+1/2})}^2 \rightarrow 0$ as $k \rightarrow \infty$ (by Lipschitz continuity).  
The second lemma is to switch expectation and maximum operators. 
\begin{lemma}
    Let $\mathcal{F} = (\mathcal{F}_k)_{k \ge 0}$ be a filtration and ($u_k$) a stochastic process adapted to $\mathcal{F}$ with $\mathbb{E}[ u_{k+1} \vert \mathcal{F}_k ] = 0$. 
    Then for any $K \in \mathbb{N}, z_0 \in \dom\, g$, $q \in \{0, 1, 2, \ldots\}$, and any compact set $\cC \subset V$, 
    \begin{align} 
        \mathbb{E} \left[ \max_{z \in \cC} \sum_{k=1}^{K-1} k^q \inp{u_{k+1}}{z} \right] \leq \frac{1}{2} K^q \max_{z, z' \in \cC} {\En{ z - z' }}^2 + \frac{1}{2} \sum_{k=0}^{K-1} k^q \EE{\En{ u_{k+1} }}^2. 
        \label{eq:bound-exp-max-inp}
    \end{align} 
    \vspace{-8pt}
    \label{lemma:bound-exp-max-inp}
\end{lemma} 

% For double-loop SVRG-EG, we prove the following sublinear convergence rate. 
% \begin{theorem}
%     Let Assumptions $1$-$4$ hold, $K \in \mathbb{N}$, $\alpha \in [\frac{1}{2}, 1)$, and $\tau = \frac{\sqrt{1-\alpha}}{L}\gamma$, for $\gamma \in (0, 1)$. Then, for $q \in \{0, 1, 2, \ldots\}$ and $z^S = \frac{1}{K Q_S} \sum_{s=0}^{S-1} s^q \sum_{k=0}^{K-1} z^s_{k+1/2}$ where $Q_S = \sum_{s=0}^{S-1} s^q$, it follows that 
%     \begin{equation}
%         \EE\left[ \gap(z^S) \right] \le \frac{(q+1)}{2 \tau K(S-q-1)}D_1 \max_{z, z'} \En{z - z'}^2,  
%     \end{equation} 
%     where $D_1 = \left( 1 + \frac{2\gamma^2}{1 - \gamma} \right) \left( \alpha + K(1 - \alpha) \right) + \frac{1}{2}$. 
%     \label{thm:sublinear-convergence-iias-double-loop}
% \end{theorem}

% \begin{corollary}
%     Set $K=\frac{N}{2}$, $\alpha = 1 - \frac{1}{K}$ and $\gamma = 0.99$ in \cref{alg:eg-svrg-double-loop}. 
%     The time complexity to reach $\epsilon$-accuracy on the gap function is 
%     $\mathcal{O}\left( N + \frac{L\sqrt{N}}{\epsilon} \right)$. 
%     \label{crl:sublinear-convergence-iias-double-loop}
% \end{corollary}

Then, we prove the following $O(1/T)$ convergence rate for the loopless SVRG-EG. 
% We also extend our convergence analysis to loopless SVRG-EG. 
\begin{theorem}
    Assume Assumptions~\ref{asp:7}-\ref{asp:10} hold. 
    Let $z_k, w_k, z_{k+1/2}$ be iterates generated by loopless SVRG-EG with $p \in (0, 1]$, $\alpha \in [\frac{1}{2}, 1)$, 
    and $\tau = \frac{\sqrt{1-\alpha}}{L}\gamma$, for $\gamma \in (0, 1)$. 
    Then, for $q \in \{0, 1, 2, \ldots\}$, $K > q + 1$, and $z^K = \frac{1}{s_k} \sum_{k=0}^{K-1} k^q z_{k+1/2}$ 
    where $s_k = \sum_{k=0}^{K-1} k^q$, 
    % it follows that 
    \begin{equation}
        \EE\left[ \gap(z^K) \right] \leq \frac{(q+1)}{2\tau (K-q-1)} D_1 \max_{z, z' \in \cC} \En{z - z'}^2, 
    \end{equation}
    where $D_1 = 2 \big( 1 + \frac{\gamma^2 + 1}{1-\gamma} \big) \big( \alpha + \frac{1 - \alpha}{p} \big)$. 
    % \vspace{-4pt}
    \label{thm:sublinear-convergence-iias-loopless}
\end{theorem}

\begin{corollary}
    Set $p=\frac{2}{N}$, $\alpha = 1 - \frac{2}{N}$ and $\tau = \frac{0.99\sqrt{2}}{\sqrt{N}L}$ in loopless SVRG-EG. 
    The time complexity to reach $\epsilon$-accuracy on the gap function is 
    % $\mathcal{O}\left( 1 + \frac{L\sqrt{N}}{\epsilon} \right)$. 
    $\mathcal{O}\left( \frac{L\sqrt{N}}{\epsilon} \right)$. 
    \label{crl:sublinear-convergence-iias-loopless}
\end{corollary} 

% The sketch of the proof of Theorem~\ref{thm:sublinear-convergence-iias}: 

%% file: experiments.tex
\section{NUMERICAL EXPERIMENTS} 

In this section, we show applications of our results and provide numerical results on the performance of SVRG-EG algorithm for these applications. 

The Error-Bound Condition is known to hold for the following case: 
\begin{theorem}{\textnormal{\citep[Theorem 2.2 (b)]{tseng1995linear}}} 
    % Let $Z$ be a nonempty closed convex set. 
    % In \eqref{prob:vi}, 
    % assume $g$ is a characteristic function whose domain is $\mathcal{Z} = \dom g$ (i.e. $g = \delta_{\cZ}$). 
    If $F$ is affine and $\cZ$ is a polyhedral set in~\ref{prob:vi}, then 
    for every $\tau \in (0, \infty)$, there exists positive $\epsilon_0$ and $C_0$ such that Error-Bound Condition holds. 
    % whenever one of the following conditions holds. 
    % \begin{itemize}
    %     \item[(a)] (Strongly monotone case) $F$ is strongly monotone and Lipschitz continuous on $\cZ$. 
    %     % and $g = \delta_{\cZ}$. 
    %     % $g$ is a characteristic function of some nonempty closed convex set $\cZ$. 
    %     \item[(b)] (Affine case) $F$ is affine and $\cZ$ is a polyhedral set. 
    %     % $g = \delta_{Z}$ where 
    %     \item[(c)] (Monotone composite case) 
    %     \begin{equation*}
    %         F(z) = E^\top G(Ez) + h, \quad \forall\; z \in \mathbb{R}^n, 
    %     \end{equation*}
    %     where $E \in \mathbb{R}^{m \times n}$ with no zero column, $h \in \mathbb{R}^n$, and $G:\; \mathbb{R}^m \rightarrow \mathbb{R}^m$ is a strongly monotone Lipschitz continuous function, and $\cZ$ is a polyhedral set. 
    %     % $g = \delta_{Z}$ where 
    % \end{itemize}
    \label{thm:affine}
\end{theorem}
% \ck{double check that above is correct}

Equilibrium computatin in matrix games and two-player zero-sum EFGs with perfect recall,
and the primal-dual formulation of certain image segmentation problems,
can all be written as VIs with an affine operator and polyhedral feasible set~\citep{stengel1996efficient}, and hence our linear convergence guarantees apply. 

We contrast (loopless) SVRG-EG with many highly competitive methods including the deterministic extragradient algorithm,
the \emph{primal-dual algorithm} (PDA) \citep{chambolle2011first}, 
\emph{CFR+}~\citep{tammelin2014solving} (in matrix games this is simply regret matching$^+$ with alternation and linear averaging), 
and \emph{optimistic online mirror descent} (OOMD) with the $l_2$ norm and entropy regularizers (or \emph{dilated entropy} for EFGs~\citep{hoda2010smoothing,kroer2020faster,farina2021better}) \citep{rakhlin2013optimization,syrgkanis2015fast}.
We compare these algorithms based on last iterate and linear iterate averaging.
% We track the performance of last iterate, uniform iterate average and linear iterate average for all algorithms. 
More details about these algorithms
% , we use the theoretically correct stepsizes, which 
are specified in~\cref{sec:app-experiment-descriptions}.  % \ck{add ref to specific section}. 

\paragraph{Matrix games}
We consider bilinear zero-sum matrix games of the form: 
\begin{equation}
    \min_{x \in X} \max_{y \in Y} \; x^\top A y, 
    \label{eq:matrix-game}
\end{equation}
where $X$ and $Y$ are simplexes in $n$ and $m$ dimensional spaces, i.e., 
$X = \Delta^n = \{ x \in \mathbb{R} \; \vert \; \sum_i x_i = 1,\; x \ge 0 \}$ and 
$Y = \Delta^m$. 
We can regard bilinear games as VIs by setting $z=(x,y)$, $F(z) = (Ay, -A^\top x)$ and
$g(z) = \delta_X(x) + \delta_Y(y)$. 
% , where $\delta(x)$ (resp. $\delta(y)$) $= 0$ if $x \in X$ (resp. $y \in Y$) and $+\infty$ otherwise. 
% $\mathcal{Z} = \mathcal{X} \times \mathcal{Y} = X \times Y$

We use the oracle discussed in \citet{alacaoglu2022stochastic}. 
For each iteration, we randomly 
% choose one coordinate $j \in [m]$ of $y$ to estimate $\nabla_x f$ and one coordinate $i \in [n]$ of $x$ to estimate $\nabla_y f$. 
% \tianlong{(simplify?)}
% In particular, we 
pick $(i, j) \in [n] \times [m]$ and let $F_\xi(z) = (A_{\cdot j} y_j, - A_i^\top x_i)$ with probability $p_{i, j} = p_i p_j$, 
where $p_i = {\En{A_i}^2}/{\En{A}_F^2}$ and $p_j = {\En{A_{\cdot j}}^2}/{\En{A}_F^2}$. Under this oracle, Assumption~\ref{asp:4} holds with $L = \En{A}_F$. 
% We construct both structured and randomly generated matrix games. 

We run SVRG-EG on four classes of matrix games: policeman and burglar game \citep{nemirovski2013mini}, 
two symmetric matrices from \citet{nemirovski2004prox}, and random matrices. 
Descriptions for these classes are in~\cref{sec:app-experiment-descriptions}. 
% We run SVRG-EG on each class of matrix games with $n = m = 500$. 
For each game, we first compare (loopless) SVRG-EG algorithm with its deterministic counterpart (full gradient) EG algorithm (left plots) with last iterate, uniform, linear, and quadratic iterate averaging for both algorithms. 
Then we compare to all the algorithms on last-iterate behavior (middle plots), and with linear iterate averaging (right plots).
% Note that deterministic EG allows a $1/\En{A}_2$ stepsize. 
Here, we show results of policeman and burglar game in~\cref{fig:pb}. 
See~\cref{sec:app-additional-results} for more results. 

\begin{figure*}[ht] 
    % \vskip 0.2in
    \begin{center}
    \includegraphics[width=0.32\textwidth]{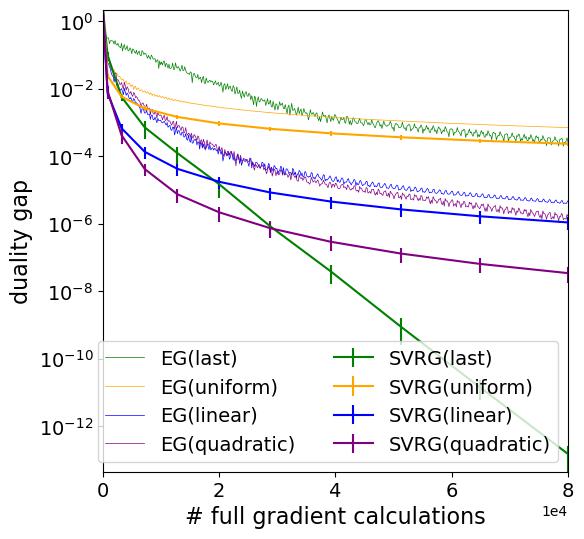}
    \includegraphics[width=0.32\textwidth]{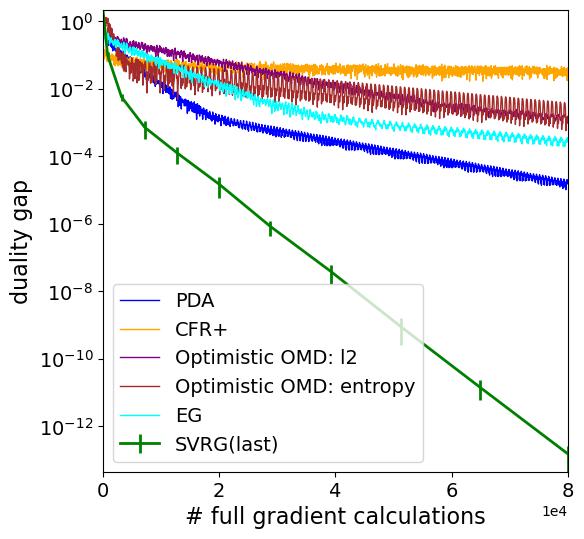}
    \includegraphics[width=0.32\textwidth]{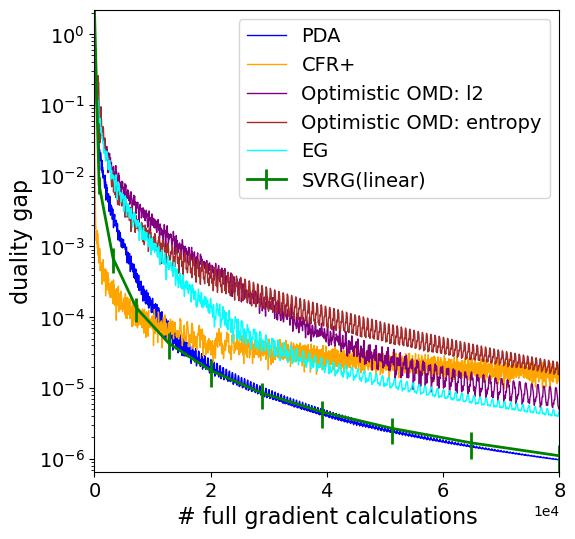}
    \caption{Numerical results on policeman and burglar game. 
    % The left plot shows double-loop SVRG-EG and the right plot shows loopless SVRG-EG. 
    Random algorithms are implemented with seeds 0-9. 
    % , inclusive
    We draw plots with error bars representing standard deviations across different seeds. 
    The y-axis shows performance in terms of the duality gap. 
    The x-axis shows units of computation in terms of the number of evaluations of $F$ (full gradient calculations).
    The left plot compares SVRG-EG with deterministic EG.
    The center plot compares the last-iterate behavior of all algorithms.
    The right plot compares the linear average of all algorithms.
    }
    \label{fig:pb}
    \end{center}
    % \vskip -0.2in
\end{figure*}

% \begin{figure*}[ht]
%     % \vskip 0.2in
%     \begin{center}
%     \includegraphics[width=0.22\textwidth]{matrix_uni_1000_1000_1}
%     \includegraphics[width=0.22\textwidth]{matrix_uni_1000_1000_2}
%     \includegraphics[width=0.22\textwidth]{matrix_uni_1000_1000_3}
%     \caption{Numerical results on randomly generated matrix. The setup is the same as in \cref{fig:pb}.}
%     \label{fig:uni}
%     \end{center}
%     % \vskip -0.2in
% \end{figure*}

% \begin{figure}[ht]
%     \vskip 0.2in
%     \begin{center}
%         \includegraphics[width=0.49\textwidth]{matrixgame_nem1_500_500_d_svrg}
%         \includegraphics[width=0.49\textwidth]{matrixgame_nem1_500_500_l_svrg}
%         \includegraphics[width=0.49\textwidth]{matrixgame_nem2_500_500_d_svrg}
%         \includegraphics[width=0.49\textwidth]{matrixgame_nem2_500_500_l_svrg}
%     \caption{Numerical results on two symmetric matrices in \citet{nemirovski2004prox}. The upper (lower) two plots are corresponding to the first (second) family. 
% The setup is the same as in \cref{fig:pb}. 
%     % \tianlong{Other settings are the same as before. }
%     }
%     \label{fig:symmetric}
%     \end{center}
%     \vskip -0.2in
% \end{figure}

% \begin{figure}[ht]
%     \vskip 0.2in
%     \begin{center}
%         \includegraphics[width=0.49\textwidth]{matrixgame_uni_500_500_d_svrg}
%         \includegraphics[width=0.49\textwidth]{matrixgame_uni_500_500_l_svrg}
%         \caption{Numerical results on (uniformly) random matrix game. The setup is the same as in \cref{fig:pb}. }
%         \label{fig:unif}
%     \end{center}
%     \vskip -0.2in
% \end{figure}

In all cases, we see that the SVRG-EG algorithm outperforms deterministic EG, 
no matter whether we look at last iterate, uniform, linear, or quadratic iterate averaging. 
% There is little performance difference between double-loop and loopless SVRG-EG.
In the policeman and burglar problem, the last iterates of SVRG-EG show a linear convergence rate, unlike any other method. % , while linear iterate averaging converges much faster at the beginning. 
In random matrix game, linear and quadratic iterate averaging performs best, and SVRG-EG with linear averaging beats every method except CFR+.
% Uniform iterate averaging never performs well on any instances, but whether linear iterate averaging or last iterate performs better depends on the game. % dependent on different games. 
% In particular, in randomly generated instances (policeman and burglar problem and uniformly randomized matrix game) linear iterate averaging performs better. 
% In contrast, in deterministically generated instances (two symmetric matrices) last iterates \tianlong{(?)} perform better. 
% Interestingly, loopless SVRG-EG fluctuates more than double-loop SVRG-EG across the random seeds. 
% This is likely due to the deterministic collection of snapshots in double-loop SVRG-EG. %One reason is that double loop SVRG-EG computes the snapshot point based on averages over all iterations in an epoch, which stabilizes the convergence of the algorithm. 
% From these instances, we see potential fast linear convergence rate on last iterates and the significant improvement made by IIAS. 

\paragraph{Extensive-form games}

We also evaluate the numerical performance of SVRG-EG on extensive-form games, 
which can be cast as bilinear saddle-point problems similar to \cref{eq:matrix-game} using the \emph{sequence-form} representation~\citep{stengel1996efficient}. 
Instead of a simplex, the feasible region of EFGs is describled by the \emph{treeplex}, 
which is generalization of simplexes that captures the sequential structure of an EFG~\citep{hoda2010smoothing}. 

We test the EG methods on the \emph{Leduc} poker game \citep{southey2012bayes} and the \emph{Search} game from \citet{kroer2020faster}. 
As before, we use the same oracle such that Assumption $4$ holds with $L = \En{A}_F$. 
% Due to a higher computation overhead, we run the randomized algorithms on random seeds $0, 1, 2$ only. 
The results on Search are in \cref{fig:search5}, while the results on Leduc are deferred to~\cref{sec:app-additional-results}.  % \ck{which appendix section?}.

\begin{figure*}[ht]
    % \vskip 0.2in
    \begin{center}
        \includegraphics[width=0.32\textwidth]{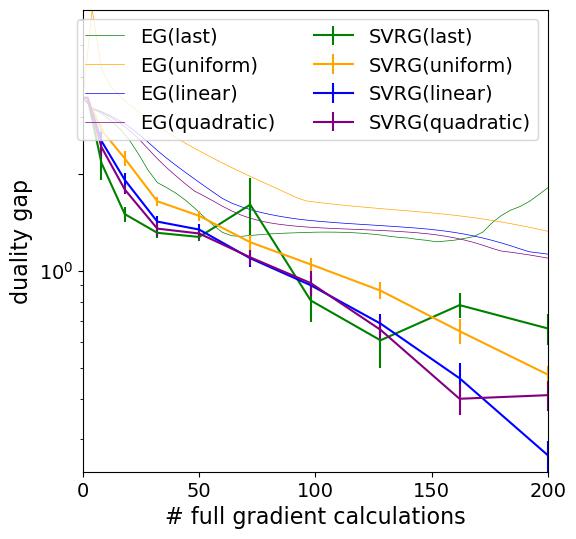}
        \includegraphics[width=0.32\textwidth]{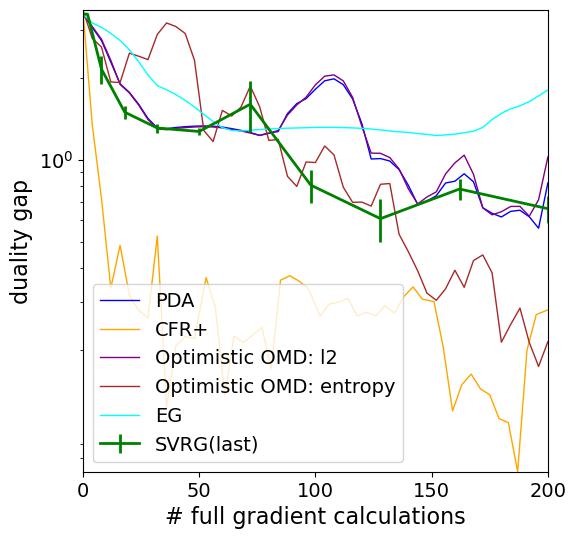}
        \includegraphics[width=0.32\textwidth]{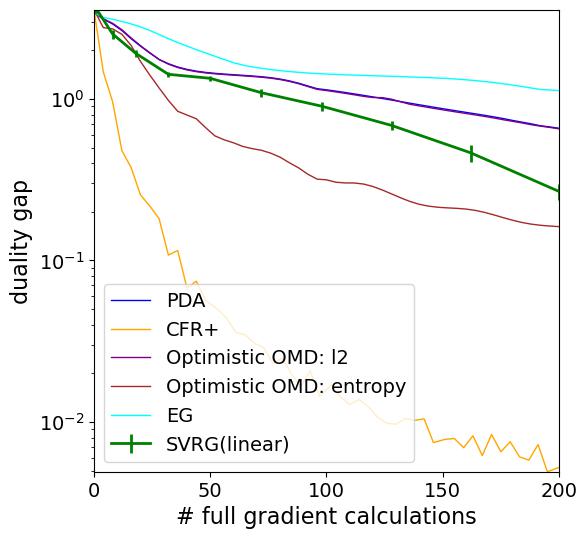}
        \caption{Numerical results on Search (zero sum) game. Random algorithms are implemented with seeds 0-9. The setup is the same as in \cref{fig:pb}. } 
        % \vspace{-2pt}
        \label{fig:search5}
    \end{center}
    % \vskip -0.2in
\end{figure*}

In Search, the SVRG-EG algorithms outperform EG. % , though by a smaller margin than in matrix games.
For Leduc in the appendix, we find that SVRG-EG and EG are very similar. 
%  give similar performance as deterministic EG. 
This is likely because the payoff matrices $A$ in EFGs are more sparse, and $\En{A}_F$ is much larger than $\En{A}_2$. 
Thus, the theoretical stepsizes for SVRG-EG are much smaller, which makes them converge slowly. 
% However, still, SVRG-EG shows performance comparable to deterministic EG. 
IIAS again leads to much better performance than uniform iterate averaging. 
% was demonstrated to provide prominent practical speedup for this game. 
In both cases, though, SVRG-EG is still not competitive with CFR+. It is an interesting future direction to design better gradient estimators or adaptive stepsize schemes for improving the performance of SVRG-EG to be competitive with CFR+.

\paragraph{Image segmentation}
Image segmentation is a process of dividing an $m \times n$ image into a collection of $h$ regions of pixels such that each region is homogenerous, while the total interface between the regions is minimized.
It is essential in image analysis and pattern recognition, and generally difficult~\citep{deng1999color,cheng2001color}. 
\citet{chambolle2011first} showed that under the assumption that the optimal centroids of regions are known, the partitioning problem can be formulated as a SPP of the form: 
\begin{equation}
    \min_{u: u_{ij\cdot} \in \Delta^h} \max_{v: {\En{ v_{ij\cdot} }}_{\infty} \le \frac{1}{2}} \sum_{l=1}^h \left( u_{\cdot \cdot l}^T A v_{\cdot \cdot l} + \inp{d_{\cdot \cdot l}}{u_{\cdot \cdot l}} \right), 
    \label{eq:image-segmentation-spp}
\end{equation}
where $u$ and $v$ are $mnh$ and $2mnh$ dimensional vectors, respectively. $A \in \{ -1, 0, 1\}^{mn \times mn}$ and $d \in \mathbb{R}^{mnh}$. This is a non-strongly convex-concave SPP satisfying the condition in~\cref{thm:affine}, and thus the error-bound condition. 
We run the SVRG-EG algorithm and other Euclidean-distance-based deterministic methods on this problem (the remaining algorithms are not applicable to the dual domain).
We show an image instance in \cref{fig:image-segmentation}, while leaving other results to the appendix. 
SVRG-EG significantly outperforms all other methods with linear averaging and in last iterate (last iterate plots are in~\cref{sec:app-additional-results}). 

\begin{figure}[H]
    % \vskip 0.2in
    \begin{center}
        \begin{minipage}[]{0.33\textwidth}
            \begin{center}
            \includegraphics[width=0.5\textwidth]{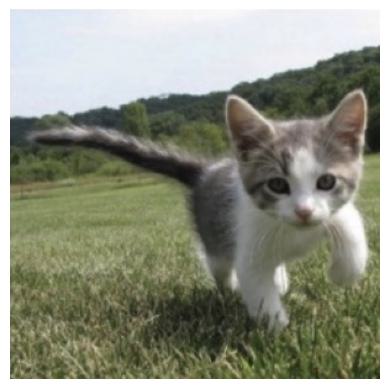}
            \includegraphics[width=0.5\textwidth]{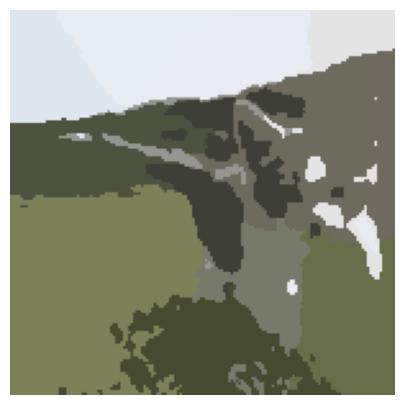}
            \end{center}
        \end{minipage}
        \begin{minipage}[]{0.66\textwidth}
            \begin{center}
            \includegraphics[width=0.49\textwidth]{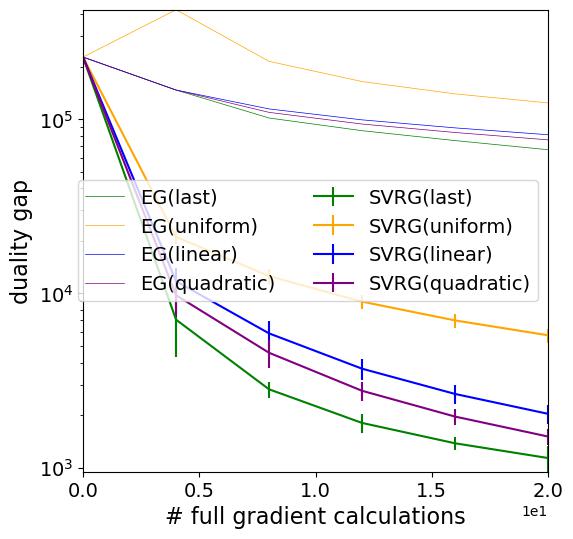}
            \includegraphics[width=0.49\textwidth]{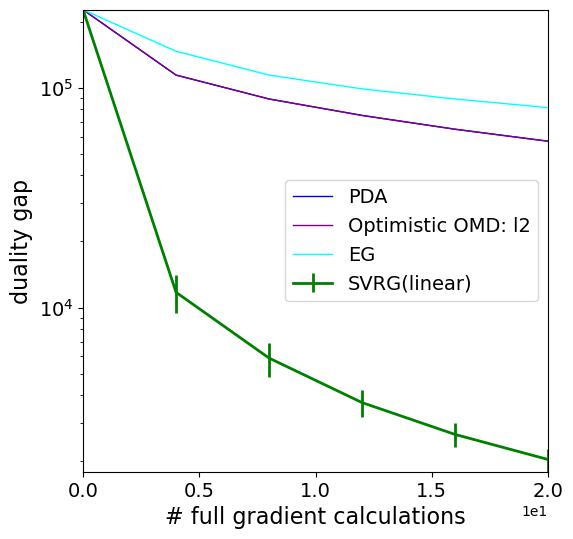}
            \end{center}
        \end{minipage} 
        \caption{An image segmentation instance.
        Top left is the original image. Top right is the segmentation result from solving~\cref{eq:image-segmentation-spp} via SVRG-EG. 
        Bottom left is numerical performance of SVRG-EG compared to EG. 
        Random algorithms are implemented with seeds 0-5. Other setups are the same as in \cref{fig:pb}.
        Bottom right is the numerical performance for all applicable algorithms, with linear averaging.
        Note that PDA and Optimistic OMD overlap.
        } 
        % \vspace{-6pt}
        \label{fig:image-segmentation}
    \end{center}
    % \vskip -0.1in
\end{figure}

% \paragraph{Discussion}
% We showed that the loopless and double-loop SVRG-EG algorithms for VIs enjoy a linear-rate last-iterate convergence under an error-bound condition, and instantiated these results for equilibrium computation in two-player zero-sum games. We then showed that both algorithms are compatible with IIAS (for general monotone VIs), which often leads to better practical performance.
% Numerically, we showed that both the last-iterate and IIAS-based versions of SVRG-EG perform substantially better than the uniform variant, and in many cases better than other state-of-the-art methods.

%% file: appendix.tex
\section{ADDITIONAL RELATED WORK}
\label{app:related-work} 

\paragraph*{Variance reduction.} There is a huge literature on variance reduction techniques in stochastic optimization.
\citet{johnson2013accelerating} is a seminal work in the context of stochastic gradient descent for finite-sum minimization, mainly motivated by supervised learning settings. Via relatively simple analysis, the authors showed that SVRG ensures linear convergence when the minimization problem is smooth and strongly convex.
\citet{roux2012stochastic,schmidt2013minimizing} considered closely related stochastic averaging gradient (SAG) methods and established their sublinear and linear convergence rates for general convex and strongly convex, smooth problems, respectively.
\citet{defazio2014saga} proposed SAGA, an incremental gradient algorithm that combines algorithmic ideas in SVRG and SAG. The authors showed that they can all be derived from a variance reduction approach, using different gradient and function value estimates, but SAGA achieves a faster linear convergence rate for strongly convex problems. 
% \todo{add SAG,SAGA,SVRG}
\citet{kovalev2020don} introduced the loopless SVRG variant for convex minimization 
and showed linear convergence.  
% and \citet{alacaoglu2022stochastic} extended this idea to VIs. 

\paragraph*{Recent work on stochastic extragradient.} There is also recent literature on stochastic (non-SVRG) extragradient methods for saddle-point problems where finite-sum structure and variance reduction are not considered~\citep{li2022convergence,du2022optimal,li2022nesterov}.
There, the problem settings and types of guarantees are very different from the ones we consider in this work.
In particular, none of them allow polyhedral feasible sets, such as strategy sets of EFGs.

\section{LINEAR CONVERGENCE ANALYSIS FOR SVRG-EG} 
\label{sec:app-linear-convergence-for-svrg-eg}

\subsection{Proofs of Lemmas}
\label{sec:app-linear-convergence-lemmas}

\subsection*{Proof of Lemma~\ref{lemma:eb-1}}

\begin{proof}
    The error-bound condition \eqref{eq:error-bound-assumption} tells us that there exist scalars $\epsilon_0 > 0$ and $C_0 > 0$ such that 
    \begin{equation}
    \min_{z^* \in \mathcal{Z}^*} \En*{z - z^*} \le \frac{C_0}{\tau} \En*{z - \Pi_{\mathcal{Z}}(z - \tau F(z))}
    \end{equation}
    for all $z \in \mathcal{Z}$ such that $\En{ z - \Pi_{\mathcal{Z}}(z - \tau F(z)) } \le \tau\epsilon_0$. 
    %For convenience, we can think $\epsilon_0$ as a small scalar and $C_0$ as a larger scalar. 
    %Hence, the desired inequality follows when the natural residual $\En{ z - \prox(z - \tau F(z)) }$ is small enough. 
    
    Obviously, for $z\in \cZ$ such that $\En{ z - \Pi_{\mathcal{Z}}(z - \tau F(z)) } > \tau\epsilon_0$, we have 
    \begin{equation}
        \En{z - \Pi_{\mathcal{Z}^*}(z)} \le D(\mathcal{Z}) = \frac{D(\mathcal{Z})}{\tau\epsilon_0} \tau\epsilon_0 < \frac{D(\mathcal{Z})}{\tau\epsilon_0} \En{z - \Pi_{\mathcal{Z}}(z - \tau F(z))}, 
    \end{equation}
    where $D(\mathcal{Z}) = \max_{z, z' \in \mathcal{Z}} \En{ z - z' }$. 

    Let $C = \max\left\{\frac{C_0}{\tau}, \frac{D(\mathcal{Z})}{\tau\epsilon_0}\right\}$. 
    Consequently, for all $z \in \mathcal{Z}$, the inequality 
    $\min_{z^* \in \mathcal{Z}^*} \En*{z - z^*} \leq C \En*{z - \Pi_{\mathcal{Z}}(z - \tau F(z))}$ holds. 
    % both $\En{ z - \Pi_{\mathcal{Z}}(z - \tau F(z)) } \le \epsilon_0$ and $\En{ z - \Pi_{\mathcal{Z}}(z - \tau F(z)) } > \epsilon_0$. 
    % If $C^\prime < \frac{1}{\tau L}$, then $C = \frac{1}{\tau L}$ is also valid for \eqref{eq:eb}. 
    This completes the proof of the lemma. 
\end{proof}

\subsection*{Proof of Lemma~\ref{lemma:eb-in-alg}}

\begin{proof} 
    To prove this lemma, we introduce a sequence of auxiliary ``quasi $(k+\frac{1}{2})$th iterates'': 
    for each $k$, denote 
    $\tilde{z}_{k+1/2} = \Pi_{\mathcal{Z}}(\bar{z}_k - \tau F(\bar{z}_k))$, 
    which is the standard update of the first step of the deterministic EG starting from $\bar{z}_k$ (using $F(\bar{z}_k)$ instead of $F(w_k)$).
    Intuitively, we show that $z_{k+1/2}$ is an approximation of $\tilde{z}_{k+1/2}$ when $\tau$ is small, while $\tilde{z}_{k+1/2}$ satisfies the error bound condition as in the classical analysis of deterministic EG. As such, error bounds (on $\bar{z}_k$) still hold.

    % We emphasize that this is not computed by the algorithm, 
    % but a variant of $z_{k+1/2}$ computed using $F(\bar{z}_k)$ instead of $F(w_k)$. 
    To bound the distance between $\tilde{z}_{k+1/2}$ and $z_{k+1/2}$, 
    we use non-expansiveness of projection operator $\Pi_\cZ$: 
    \begin{align}
        \En{ \tilde{z}_{k+1/2} - z_{k+1/2} } 
        = \En{ \Pi_{\mathcal{Z}}(\bar{z}_k - \tau F(\bar{z}_k)) - \Pi_{\mathcal{Z}}(\bar{z}_k - & \tau F(w_k)) } 
        \leq \tau \En{ F(\bar{z}_k) - F(w_k) } \nonumber \\ 
        & \leq \tau L \En{ \bar{z}_k - w_k } = \alpha \tau L \En{ z_k - w_k }. 
        \label{eq:dist-quasi-half}
    \end{align} 
    The second inequality follows from the fact that 
    the operator $F$ is also $L$-Lipschitz 
    since $\En{ F(z_1) - F(z_2) }^2 = \En{ \bbE_\xi[ F_\xi(z_1) - F_\xi(z_2) ] } \leq \bbE_\xi \En{ F_\xi(z_1) - F_\xi(z_2) }^2 \leq L^2 \En{z_1 - z_2}^2,\; \forall z_1, z_2 \in \cZ$. 

    With~\cref{eq:dist-quasi-half}, we can bound $\En{ \bar{z}_k - \tilde{z}_{k+1/2} }^2$ as 
    \begin{align}
        &\En{ \bar{z}_k - \tilde{z}_{k+1/2} }^2 \nonumber \\ 
        \leq& 2 \En{ \bar{z}_k - z_{k+1/2} }^2 + 2 \En{ \tilde{z}_{k+1/2} - z_{k+1/2} }^2 \nonumber \\ 
        \leq& 2 \alpha \En{ z_k - z_{k+1/2} }^2 + 2 (1 - \alpha) \En{ w_k - z_{k+1/2} }^2 - 2 \alpha(1 - \alpha) \En{z_k - w_k}^2 + 2 \alpha^2 \tau^2 L^2 \En{ z_k - w_k }^2 \nonumber \\ 
        =& 2 \alpha \En{ z_k - z_{k+1/2} }^2 + 2 (1 - \alpha) \En{ w_k - z_{k+1/2} }^2 + 2\alpha\left(\alpha \tau^2 L^2 - (1 - \alpha)\right) \En{ z_k - w_k }^2, 
        \label{eq:right}
    \end{align}
    where we use $\En{a + b}^2 \le 2 \En{a}^2 + 2 \En{b}^2$, 
    $\En{\alpha a + (1 - \alpha) b}^2 = \alpha \En{a}^2 + (1 - \alpha) \En{b}^2 - \alpha(1 - \alpha)\En{a - b}^2$ 
    and~\cref{eq:dist-quasi-half}. 
    % $\bar{z}^s_k - z^s_{k+1/2} = \alpha (z^s_k - z^s_{k+1/2}) + (1 - \alpha) (w^s - z^s_{k+1/2})$. 

    On the other hand, we have 
    \begin{equation}
        \En{ \bar{z}_k  - \Pi_{\mathcal{Z}^*}(\bar{z}_k) }^2 \ge \frac{1}{2} \En{ z_k - \Pi_{\mathcal{Z}^*}(\bar{z}_k) }^2 - \En{ \bar{z}_k - z_k }^2 = \frac{1}{2} \En{ z_k - \Pi_{\mathcal{Z}^*}(\bar{z}_k) }^2 - (1 - \alpha)^2 \En{ z_k - w_k }^2 
        \label{eq:bound-left-z}
    \end{equation}
    and
    \begin{equation}
        \En{ \bar{z}_k - \Pi_{\mathcal{Z}^*}(\bar{z}_k) }^2 \ge \frac{1}{2} \En{ w_k - \Pi_{\mathcal{Z}^*}(\bar{z}_k) }^2 - \En{ \bar{z}_k - w_k }^2 = \frac{1}{2} \En{ w_k - \Pi_{\mathcal{Z}^*}(\bar{z}_k) }^2 - \alpha^2 \En{ z_k - w_k }^2. 
        \label{eq:bound-left-w} 
    \end{equation}

    Combining~\cref{eq:bound-left-z,eq:bound-left-w} with a weight coefficient $\theta$, we obtain 
    \begin{align}
        \En{ \bar{z}_k - \Pi_{\mathcal{Z}^*}(\bar{z}_k) }^2 
        &= \theta \En{ \bar{z}_k - \Pi_{\mathcal{Z}^*}(\bar{z}_k) }^2 + (1 - \theta) \En{ \bar{z}_k - \Pi_{\mathcal{Z}^*}(\bar{z}_k) }^2 \nonumber \\ 
        &\geq \frac{1}{2} \Phi^\theta_{\Pi_{\mathcal{Z}^*}(\bar{z}_k)}(z_k, w_k) - \left(\theta(1 - \alpha)^2 + (1-\theta)\alpha^2\right) \En{ z_k - w_k }^2 \nonumber \\ 
        &\geq \frac{1}{2} \distt^\theta_{\cZ^*}(z_k, w_k) - \left(\theta(1 - \alpha)^2 + (1-\theta)\alpha^2\right) \En{ z_k - w_k }^2, 
        \label{eq:left} 
    \end{align} 
    where the last inequality follows from the definition of $\distt^\theta_{\cZ^*}(z_k, w_k)$. 

    % Let $\kappa_0 = \theta(1 - \alpha)^2 + (1-\theta)\alpha^2$. 
    By applying \cref{lemma:eb-1}, it follows that 
    \begin{equation}
        \En{ \bar{z}_k - \Pi_{\mathcal{Z}^*}(\bar{z}_k) }^2 
        \leq C^2 \En{ \bar{z}_k - \tilde{z}_{k+1/2} }^2. 
    \end{equation}
    Combining this with \cref{eq:right,eq:left}, and using $C = \frac{\bar{C}}{\tau}$, we obtain that 
    \begin{align}
        &\frac{1}{2} \distt^\theta_{\cZ^*}(z_k, w_k) - \left( \theta(1 - \alpha)^2 + (1-\theta)\alpha^2 \right) \En{ z_k - w_k }^2 \nonumber \\ 
        \leq& \frac{\bar{C}^2}{\tau^2} \left( 2 \alpha \En{ z_k - z_{k+1/2} }^2 + 2 (1 - \alpha) \En{ w_k - z_{k+1/2} }^2 + 2\alpha\left(\alpha \tau^2 L^2 - (1 - \alpha)\right) \En{ z_k - w_k }^2 \right). 
        \label{eq:combined-error-bounds} 
    \end{align}

    Since $\tau \leq \frac{\sqrt{1 - \alpha}}{L}$, we have $\alpha \tau^2 L^2 - (1 - \alpha) \leq 0$ and therefore
    \begin{align}
        &\frac{1}{2} \distt^\theta_{\cZ^*}(z_k, w_k) - \left( \theta(1 - \alpha)^2 + (1-\theta)\alpha^2 \right) \En{ z_k - w_k }^2 \nonumber \\ 
        \leq& \frac{2\bar{C}^2}{\tau^2} \left( \alpha \En{ z_k - z_{k+1/2} }^2 + (1 - \alpha) \En{ w_k - z_{k+1/2} }^2 \right). 
        \label{} 
    \end{align} 

    Re-arranging terms and using $\theta(1 - \alpha)^2 + (1 - \theta) \alpha^2 \leq 2$, we have  
    \begin{equation}
        \frac{1}{2} \distt^\theta_{\cZ^*}(z_k, w_k) \leq \frac{2\bar{C}^2}{\tau^2} \left( \alpha \En{ z_k - z_{k+1/2} }^2 + (1 - \alpha) \En{ w_k - z_{k+1/2} }^2 \right) + 2 \En{ z_k - w_k }^2. 
        \label{} 
    \end{equation} 

    Then, we bound the last term in the above inequality as follows: 
    \begin{align}
        2 \En{ z_k - w_k }^2 & \leq 4 \En{ z_k - z_{k+1/2} }^2 + 4 \En{ w_k - z_{k+1/2} }^2 \tag{$\En{a + b}^2 \le 2 \En{a}^2 + 2 \En{b}^2$} \nonumber \\ 
        & = \frac{4}{\tau^2} \left( \tau^2 \En{ z_k - z_{k+1/2} }^2 + \tau^2 \En{ w_k - z_{k+1/2} }^2 \right) \nonumber \\ 
        & \leq \frac{4}{\tau^2} \left( \frac{1 - \alpha}{L^2} \En{ z_k - z_{k+1/2} }^2 + \frac{1 - \alpha}{L^2} \En{ w_k - z_{k+1/2} }^2 \right) \tag{$\tau \leq \frac{\sqrt{1 - \alpha}}{L}$} \nonumber \\ 
        & \leq \frac{4}{\tau^2 L^2} \left( \En{ z_k - z_{k+1/2} }^2 + (1 - \alpha) \En{ w_k - z_{k+1/2} }^2 \right) \nonumber \\ 
        & \leq \frac{4}{\tau^2 \alpha L^2} \left( \alpha \En{ z_k - z_{k+1/2} }^2 + (1 - \alpha) \En{ w_k - z_{k+1/2} }^2 \right). 
    \end{align}

    This yields that 
    \begin{equation}
        \frac{1}{2} \distt^\theta_{\cZ^*}(z_k, w_k) \leq \left( \frac{2\bar{C}^2}{\tau^2} + \frac{4}{\alpha L^2 \tau^2} \right) \left( \alpha \En{ z_k - z_{k+1/2} }^2 + (1 - \alpha) \En{ w_k - z_{k+1/2} }^2 \right). 
        \label{} 
    \end{equation}

    Equivalently, 
    \begin{equation}
        \tau^2 \frac{\alpha L^2}{ 4 \alpha \bar{C}^2 L^2 + 8} \distt^\theta_{\cZ^*}(z_k, w_k) \leq \alpha \En{ z_k - z_{k+1/2} }^2 + (1 - \alpha) \En{ w_k - z_{k+1/2}}^2. 
    \end{equation} 
    Since $\frac{\gamma\sqrt{1 - \alpha}}{L} \leq \frac{\sqrt{1 - \alpha}}{L}$,~\cref{eq:eb-in-algo} holds. 
\end{proof}

\subsection{Loopless SVRG-EG Proofs}
\label{subsec:app-linear-convergence-for-loopless}

\subsection*{Proof of \cref{thm:linear-convergence-for-loopless}}

\begin{proof}
    First, to make our work self-contained, we derive here a crucial inequality, 
    which is similar to \citet[Equation (10)]{alacaoglu2022stochastic}. 

    For brevity, let $\hat{F}(z_{k + 1/2}) = F(w) + F_{\xi_k}(z_{k+1/2}) - F_{\xi_k}(w)$. 
    By the property of the proximal mapping and the definitions of $z_{k+1}$ and $z_{k + 1/2}$, % and convexity of $g$, 
    we have, for all $z$, 
    \begin{equation}
        % \langle z_{k+1} - \bar{z}_k + \tau \hat{F}(z_{k + 1/2}),\, z - z_{k+1} \rangle + \tau g(z) - \tau g(z_{k+1}) \ge 0, 
        \inp{z_{k+1} - \bar{z}_k + \tau \hat{F}(z_{k + 1/2})}{z - z_{k+1}} \geq 0, 
        \label{eq:proj-ineq-1} 
    \end{equation}
    \begin{equation}
        % \langle z_{k+1/2} - \bar{z}_k + \tau F(w),\, z_{k+1} - z_{k+1/2} \rangle + \tau g(z_{k+1}) - \tau g(z_{k+1/2}) \ge 0 
        \inp{z_{k+1/2} - \bar{z}_k + \tau F(w)}{z_{k+1} - z_{k+1/2}} \geq 0 
    \end{equation}
    for any $k = 0, 1, \ldots$. 
    We sum up these two inequalities, 
    % use the definition of $\hat{F}(z_{k + 1/2})$, 
    and rearrange terms to obtain 
    \begin{align} 
        & \langle z_{k+1} - \bar{z}_k,\, z - z_{k+1} \rangle + \langle z_{k+1/2} - \bar{z}_k,\, z_{k+1} - z_{k+1/2} \rangle \nonumber \\ 
        & \hspace{60pt} + \tau \langle F_{\xi_k}(w_k) - F_{\xi_k}(z_{k + 1/2}),\, z_{k+1} - z_{k+1/2} \rangle + \tau \langle \hat{F}(z_{k + 1/2}),\, z - z_{k+1/2} \rangle
        %  + g(z) - g(z_{k+1/2}) 
        \geq 0.
        \label{eq:sum-of-two-inequalities}
    \end{align}
    
    To bound the inner products we use the definition of $\bar{z}^s_k$ and the identity $2 \langle a,\, b \rangle = {\En{ a + b }}^2 - {\En{ a }}^2 - {\En{ b }}^2$: 
    \begin{align}
        2 \inp{z_{k+1} - \bar{z}_k}{z - z_{k+1}} = \alpha {\En{ z_k - z }}^2 & - {\En{ z_{k+1} - z }}^2 + (1 - \alpha) {\En{ w_k - z }}^2 \nonumber \\ 
        & - \alpha {\En{ z_{k+1} - z_k }}^2 - (1 - \alpha) {\En{ z_{k+1} - w_k }}^2, 
        \label{eq:inner-product-1}
    \end{align}
    \begin{align}
        2 \langle z_{k+1/2} - \bar{z}_k ,\, z_{k+1} - z_{k+1/2} \rangle &= \alpha {\En{ z_{k+1} - z_k }}^2 - {\En{ z_{k+1} - z_{k+1/2} }}^2 + (1 - \alpha) {\En{ z_{k+1} - w_k }}^2 \nonumber \\ 
        & \hspace{20pt} - \alpha {\En{ z_{k+1/2} - z_k }}^2 - (1 - \alpha) {\En{ z_{k+1/2} - w_k }}^2.
        \label{eq:inner-product-2}
    \end{align}
    
    Next, we bound the last two terms in~\cref{eq:sum-of-two-inequalities}. 
    % Different from~\citet{alacaoglu2022stochastic}, we leave $\tau$ as a parameter here. 
    % given $\tau = \frac{\sqrt{1 - \alpha}}{L} \gamma$, 
    \begin{align}
        & \mathbb{E}_k[2\tau \langle F_{\xi_k}(w_k) - F_{\xi_k}(z_{k + 1/2}),\, z_{k+1} - z_{k+1/2} \rangle] \nonumber \\ 
        \leq & \mathbb{E}_k[2\tau \En{ F_{\xi_k}(w_k) - F_{\xi_k}(z_{k + 1/2}) } \cdot \En{ z_{k+1} - z_{k+1/2} }] \tag{Cauchy–Schwarz inequality} \nonumber \\
        \leq & \mathbb{E}_k\left[ \frac{\tau^2}{\gamma} {\En{ F_{\xi_k}(w_k) - F_{\xi_k}(z_{k + 1/2}) }}^2 + \gamma {\En{ z_{k+1} - z_{k+1/2} }}^2 \right] \tag{$2ab \leq a^2 + b^2$} \nonumber \\ 
        = & \frac{\tau^2}{\gamma} \mathbb{E}_k {\En{ F_{\xi_k}(w_k) - F_{\xi_k}(z_{k + 1/2}) }}^2 + \gamma \mathbb{E}_k {\En{ z_{k+1} - z_{k+1/2} }}^2 \nonumber \\ 
        \leq & \frac{\tau^2 L^2}{\gamma} \En{ w_k - z_{k + 1/2} }^2 + \gamma \mathbb{E}_k {\En{ z_{k+1} - z_{k+1/2} }}^2, 
        \label{eq:bound-by-splited-two-norms}
    \end{align} 
    where the last inequality follows from the $L$-Lipschitz continuity of $F_\xi$. 

    Let $z = z^* \in \mathcal{Z}^*$. 
    By the monotonicity of $F$ and the definition of \eqref{prob:vi}, we have 
    % \begin{align}
    %     &\mathbb{E}_k\left[\langle \hat{F}(z_{k + 1/2}), \, z^* - z_{k + 1/2} \rangle + g(z^*) - g(z_{k+1/2})\right] \nonumber \\
    %     =& - \langle F(z_{k + 1/2}), \, z_{k + 1/2} - z^* \rangle + g(z^*) - g(z_{k+1/2}) \nonumber \\ 
    %     \le& - \left( \langle F(z^*), \, z_{k + 1/2} - z^* \rangle + g(z_{k+1/2}) - g(z^*) \right) \le 0.
    %     \label{eq:VI-inequality}
    % \end{align}
    \begin{align}
        \mathbb{E}_k\left[\langle \hat{F}(z_{k + 1/2}), \, z^* - z_{k + 1/2} \rangle\right] = - \langle F(z_{k + 1/2}), \, z_{k + 1/2} - z^* \rangle \leq - \inp{ F(z^*) }{ z_{k + 1/2} - z^* } \leq 0. 
        \label{eq:VI-inequality}
    \end{align}

    % Now, we have a stronger version of \citet[Eq.10]{alacaoglu2022stochastic} since we do not drop $\alpha {\En{ z_{k+1/2} - z_k }}^2$. 
    % In particular, for any $z^* \in \mathcal{Z}^*$, 
    % Using similar techniques, we can derive a key inequality for an inner loop iteration: 

    Therefore, combining~\cref{eq:sum-of-two-inequalities,eq:inner-product-1,eq:inner-product-2,eq:bound-by-splited-two-norms,eq:VI-inequality}, we obtain 
    \begin{align}
        \bbE_k[ \En{ z_{k + 1} - z^* }^2 ] \leq & \alpha {\En{ z_k - z^* }}^2 + (1 - \alpha) {\En{ w_k - z^* }}^2 \nonumber \\ 
        & - \left( (1 - \alpha) - \frac{\tau^2 L^2}{\gamma} \right) {\En{ w_k - z_{k + 1/2} }}^2 - \alpha \En{ z_k - z_{k+1/2} }^2. 
        % - (1 - \gamma) \mathbb{E}_k{\En{ z_{k + 1} - z_{k + 1/2} }}^2. 
        \label{eq:decrease-loopless}
    \end{align} 

    By the definition of $w_{k+1/2}$, 
    \begin{equation}
        \mathbb{E}_{k+1/2}\En{ w_{k+1} - z }^2 = p \En{ z_{k+1} - z }^2 + (1 - p) \En{ w_k - z }^2 \quad \forall\, z \in \cZ. 
    \end{equation} 
    By the tower property $\mathbb{E}_k[\mathbb{E}_{k+1/2}[\cdot]] = \mathbb{E}_k[\cdot]$ and letting $z = z^*$, 
    it follows that 
    \begin{equation}
        \bbE_k\En{ w_{k+1} - z^* }^2 = p \bbE_k \En{ z_{k+1} - z^* }^2 + (1 - p) \En{ w_k - z^* }^2 
    \end{equation} 
    and hence 
    \begin{equation}
        \frac{1 - \alpha}{p} \bbE_k\En{ w_{k+1} - z^* }^2 = (1 - \alpha) \bbE_k \En{ z_{k+1} - z^* }^2 + \frac{(1 - \alpha)(1 - p)}{p} \En{ w_k - z^* }^2. 
        \label{eq:def-w-kplus1} 
    \end{equation}     
    We plug~\cref{eq:def-w-kplus1} into~\cref{eq:decrease-loopless} and then attain that 
    \begin{align}
        & \bbE_k \left( \alpha {\En{ z_{k + 1} - z^* }}^2 + \frac{1 - \alpha}{p} {\En{ w_{k + 1} - z^* }}^2 \right) \nonumber \\ 
        \leq& \alpha {\En{ z_k - z^* }}^2 + \frac{1 - \alpha}{p} {\En{ w_k - z^* }}^2 - \left( (1 - \alpha) - \frac{\tau^2 L^2}{\gamma} \right) \En{ w_k - z_{k + 1/2} }^2 - \alpha \En{ z_k - z_{k+1/2} }^2. 
        \label{eq:decrease-loopless-lyapunov-func}
    \end{align} 

    Since $\tau \leq \frac{\gamma\sqrt{1 - \alpha}}{L}$, we can conclude that 
    \begin{align}
        & \bbE_k \left( \alpha {\En{ z_{k + 1} - z^* }}^2 + \frac{1 - \alpha}{p} {\En{ w_{k + 1} - z^* }}^2 \right) \nonumber \\ 
        \leq& \alpha {\En{ z_k - z^* }}^2 + \frac{1 - \alpha}{p} {\En{ w_k - z^* }}^2 - (1-\gamma)(1-\alpha) \En{ w_k - z_{k + 1/2} }^2 - \alpha \En{ z_k - z_{k+1/2} }^2. 
        \label{eq:decrease-loopless-lyapunov-func-fixed-stepsize} 
    \end{align} 

    Let $\tilde{\theta} = \frac{\alpha}{\alpha + \frac{1-\alpha}{p}}$. Then, 
    \begin{equation}
        \left( \alpha + \frac{1 - \alpha}{p} \right) \Phi^{\tilde{\theta}}_{z^*}(z_k, w_k) = \alpha {\En{ z_k - z^* }}^2 + \frac{1 - \alpha}{p} {\En{ w_k - z^* }}^2. 
        \label{eq:lyapunov-in-alg}
    \end{equation}

    Using the notation of~\cref{eq:lyapunov-in-alg} and $\gamma \in (0, 1)$, we can transform~\cref{eq:decrease-loopless-lyapunov-func-fixed-stepsize} into 
    \begin{align} 
        \left( \alpha + \frac{1 - \alpha}{p} \right) \bbE_k \Phi^{\tilde{\theta}}_{z^*}(z_{k+1}, w_{k+1}) \leq & \left( \alpha + \frac{1 - \alpha}{p} \right) \Phi^{\tilde{\theta}}_{z^*}(z_k, w_k) \nonumber \\ 
        & - (1-\gamma) \left( \alpha \En{ z_k - z_{k+1/2} }^2 + (1-\alpha) \En{ w_k - z_{k + 1/2} }^2 \right). 
        \label{eq:decrease-loopless-lyapunov-func-fixed-stepsize-concise} 
    \end{align} 

    Applying~\cref{lemma:eb-in-alg} with $\theta = \tilde{\theta}$ in~\cref{eq:decrease-loopless-lyapunov-func-fixed-stepsize-concise}, we have 
    \begin{align} 
        & \left( \alpha + \frac{1 - \alpha}{p} \right) \bbE_k \Phi^{\tilde{\theta}}_{z^*}(z_{k+1}, w_{k+1}) \nonumber \\ 
        \leq & \left( \alpha + \frac{1 - \alpha}{p} \right) \Phi^{\tilde{\theta}}_{z^*}(z_k, w_k) - (1-\gamma) \left( \alpha \En{ z_k - z_{k+1/2} }^2 + (1-\alpha) \En{ w_k - z_{k + 1/2} }^2 \right) \nonumber \\ 
        \leq & \left( \alpha + \frac{1 - \alpha}{p} \right) \Phi^{\tilde{\theta}}_{z^*}(z_k, w_k) -         \tau^2 \frac{(1 - \gamma)\alpha L^2}{ 4\alpha\bar{C}^2 L^2 + 8 } \distt^{\tilde{\theta}}_{\cZ^*}(z_k, w_k). 
    \end{align} 

    Let $z^* = \Pi^{\tilde{\theta}}_\cZ(z_k, w_k)$, then $\Phi^{\tilde{\theta}}_{z^*}(z_k, w_k) = \distt^\theta_{\cZ^*}(z_k, w_k)$ and thus 
    \begin{align} 
        \bbE_k \Phi^{\tilde{\theta}}_{z^*}(z_{k+1}, w_{k+1}) 
        \leq & \left( 1 - \tau^2 \frac{(1 - \gamma)\alpha L^2}{ \left( 4\alpha\bar{C}^2 L^2 + 8 \right) \left( \alpha + \frac{1- \alpha}{p} \right) } \right) \distt^{\tilde{\theta}}_{\cZ^*}(z_k, w_k). 
    \end{align} 

    By the definition of $\distt^{\tilde{\theta}}_{\cZ^*}(z_{k+1}, w_{k+1})$, we complete the proof by 
    \begin{align} 
        \bbE_k \distt^{\tilde{\theta}}_{\cZ^*}(z_{k+1}, w_{k+1}) \leq & 
        \bbE_k \Phi^{\tilde{\theta}}_{z^*}(z_{k+1}, w_{k+1}) \nonumber \\ 
        \leq & \left( 1 - \tau^2 \frac{(1 - \gamma)\alpha L^2}{ \left( 4\alpha\bar{C}^2 L^2 + 8 \right) \left( \alpha + \frac{1- \alpha}{p} \right) } \right) \distt^{\tilde{\theta}}_{\cZ^*}(z_k, w_k). 
    \end{align} 

\end{proof}

\subsection*{Proof of \cref{crl:linear-convergence-for-loopless}}

\begin{proof} 
    Let 
    \begin{equation}
        \rho = \tau^2 \frac{(1 - \gamma)\alpha L^2}{ \left( 4\alpha\bar{C}^2 L^2 + 8 \right) \left( \alpha + \frac{1- \alpha}{p} \right) }. \nonumber 
    \end{equation} 
    By iterating~\cref{thm:linear-convergence-for-loopless} and taking the total expectation $\bbE_0$, we have 
    \begin{equation}
        \bbE_0\left[ \distt^{\tilde{\theta}}_{\cZ^*}(z_k, w_k) \right] \leq (1 - \rho)^k \distt^{\tilde{\theta}}_{\cZ^*}(z_0, w_0) = (1 - \rho)^k \min_{z^* \in \cZ^*} \En{z_0 - z^*}^2. 
    \end{equation} 
    To reach expected $\epsilon$-accuracy as measured by the weighted distance (Lyapunov function) to the solution set, 
    % i.e., $\textnormal{dist}^\theta_{\mathcal{Z}^*}(z_k, w_k) = \epsilon$, 
    we need $(1 - \rho)^k \leq \epsilon$ and equivalently, 
    \begin{equation}
        k \geq \frac{1}{\log{\frac{1}{1 - \rho}}} \log{\frac{1}{\epsilon}} \approx \frac{1}{\rho} \log{\frac{1}{\epsilon}} \quad \text{when $\rho$ is small}, 
    \end{equation} 
    where $k$ counts the number of iterations. 

    Recall that $1 - \alpha = p = \frac{2}{N}$, $\gamma = 0.99$ and $\tau = \frac{0.99\sqrt{2}}{\sqrt{N}L}$. 
    Since 
    \begin{align} 
        \frac{1}{\rho} = \frac{\left( 4\alpha\bar{C}^2 L^2 + 8 \right) \left( \alpha + \frac{1- \alpha}{p} \right)}{ (1 - \gamma) \tau^2 \alpha L^2} = \frac{\left( 4(1 - \frac{2}{N}) \bar{C}^2 L^2 + 8 \right) \left( 2 - \frac{2}{N} \right)}{ 0.01 \times \frac{0.99^2 \times 2}{N L^2} \times (1 - \frac{2}{N}) L^2} = \cO(\bar{C}^2 N L^2 + N). 
    \end{align} 
    Since we need $2 + p N$ evaluations of $F_\xi$ per iteration, we need $4 \times k$ evaluations of $F_\xi$ to reach expected $\epsilon$-accuracy. 
    Therefore, we attain that the time complexity is $\mathcal{O}\left( (\bar{C}^2 N L^2 + N) \log{\frac{1}{\epsilon}} \right)$. 
\end{proof}

\subsection{Double-loop SVRG-EG}
\label{subsec:app-linear-convergence-for-double-loop}

Here, we formally introduce the double-loop version of the SVRG-EG algorithm: 

\begin{algorithm}
  \KwIn{Set the inner loop iteration $K$, probability distribution $Q$, step size $\tau$, $\alpha \in (0, 1)$, $z^0_0 = w^0 = z_0$}
  \For{$s = 0, 1, \cdots$}{
    \For{$k = 0, 1, \cdots, K-1$}{
      $\bar{z}^s_k = \alpha z^s_k + (1 - \alpha) w^s$\;
      $z^s_{k + 1/2} = \prox_{\tau g}(\bar{z}^s_k - \tau F(w^s))$\; 
      Draw an index $\xi^s_k$ according to $Q$\;
      $z^s_{k+1} = \prox_{\tau g}\left( \bar{z}^s_k - \tau [F(w^s) + F_{\xi^s_k}(z^s_{k+1/2}) - F_{\xi^s_k}(w^s)] \right)$\;
    }
    $w^s = \frac{1}{K} \sum_{k=1}^{K} z^s_k$\;
    $z^{s+1}_0 = z^s_k$\;
  }
  \caption{Double-loop SVRG-Extragradient (D-SVRG-EG)}
  \label{alg:eg-svrg-double-loop}
\end{algorithm}

In the double-loop SVRG-Extragradient algorithm, the reference points are updated once per outer iteration 
instead of the stochastic updates in the (loopless) SVRG-EG.
Also, the reference points are computed based on averaging over all generated points in the previous inner loop, 
instead of the last iterate in (loopless) SVRG-EG.

For the double-loop SVRG-EG algorithm, we prove the following linear convergence guarantee and corresponding time complexity.

Analogously, we define a Lyapunov function 
\begin{equation}
    \distt^\theta_{\cZ^*}(z^s_0, w^s) = \min_{z^* \in \cZ^*} \theta \En{z^s_0 - z^*}^2 + (1 - \theta) \En{w^s - z^*}^2. 
\end{equation}

\begin{theorem} 
    Assume Assumptions $1$-$4$ and \textnormal{Error-Bound Condition} hold. 
    Let $\{ z^s_k, w^s \}_{s \in \bbN_+, k = 0,\ldots,K-1}$ be iterates generated by~\cref{alg:eg-svrg-double-loop} with $K \in \bbN_+$, $\alpha \in [0, 1), \tau = \gamma \frac{\sqrt{1-\alpha}}{L}$, for $\gamma \in (0, 1)$. 
    Then, we have 
    \begin{equation}
        \mathbb{E}^s_0\left[ \textnormal{dist}^\theta_{\mathcal{Z}^*}(z^{s+1}_0, w^{s+1}) \right] \le \left( 1 - \frac{(1 - \gamma)K}{16\tilde{C}^2(\alpha + K(1 - \alpha))} \right)^s \textnormal{dist}^\theta_{\mathcal{Z}^*}(z^s_0, w^s), 
        \label{eq:linear-convergence-for-double-loop}
    \end{equation}
    where $\theta = \frac{\alpha}{\alpha + K(1-\alpha)}$ and $\tilde{C} = \max\left\{ \frac{\bar{C}}{\tau}, \sqrt{\frac{4K + 2}{\alpha(1 - \alpha)}} \right\}$. 
    % $c = \frac{K\alpha\gamma^2(1-\gamma)\beta}{7C^2 ( 8 + 3 K ) - K\alpha\gamma^2(1-\gamma)\beta}$, $\beta = \min\{ 1, \frac{\gamma^2}{32(1-\gamma)} \}$. 
    \label{thm:linear-convergence-for-double-loop}
\end{theorem}

\begin{corollary}
    Let $K=\frac{N}{2}$, $\alpha = 1 - \frac{2}{N}$ and $\tau = \frac{0.99\sqrt{2}}{\sqrt{N}L}$ in \cref{alg:eg-svrg-double-loop}. 
    Then, the time complexity to reach $\epsilon$-accuracy is $\mathcal{O}\left( \left( \bar{C}^2 N L^2 + N^2 \right) \log{\frac{1}{\epsilon}} \right)$. 
    \label{crl:linear-convergence-for-double-loop}
\end{corollary}

% Here, we allow $\alpha \in [\frac{1}{2}, 1)$ instead of requiring $\alpha = 1 - p$ as in \citet{alacaoglu2022stochastic}.
% When $\alpha = 1$, $\tau = c = 0$. 
% From this theorem, we see that $c$ increases as $\alpha$ decreases. 
\paragraph*{Remarks.} We show that the double-loop SVRG-EG exhibits a linear convergence rate of the same magnitude (in terms of its dependence on $N$) as its loopless counterpart in the typical case where $\bar{C}^2 L^2 \gg 1$. 
% In the above theorem, the exponents on the linear convergence rate are outer loop indices instead of inner loop indices.
% This is needed because, in our analysis, $w^s$ is connected to a sequence of iterates $\{ z^{s-1}_k \}_{k = 1, \ldots, K}$, which means we need to (essentially) use a sum-of-inner-loop Lyapunov function. 
% Additionally, applying \cref{lemma:eb-in-alg} introduces different optimal solutions $z^* \in \mathcal{Z}^*$ for different inner-loop iterates. This rules out a telescoping summation across inner-loop indices.

\subsection*{Lemma of error-bound condition for double-loop SVRG-EG}

To prove the linear convergence of double-loop SVRG-EG, we need a ``new''~\cref{lemma:eb-in-alg}. We state this lemma as follows. 

\begin{lemma} 
    Assume Assumptions $1$-$4$ and \textnormal{Error-Bound Condition} hold. 
    Let $K \in \bbN_+$, $\alpha \in [0, 1), \tau = \gamma \frac{\sqrt{1-\alpha}}{L}$, for $\gamma \in (0, 1)$, and $\theta \in [0, 1]$. 
    Then, for $s \geq 1$, ~\cref{alg:eg-svrg-double-loop} ensures
    \begin{equation}
        \frac{K}{8\tilde{C}^2} \distt^\theta_{\cZ^*}(z^s_0, w^s) \leq \EE^s_0 \left[ \sum_{k=0}^{K-1} \left( \alpha \En{ z^s_{k+1/2} - z^s_k }^2 + (1 - \alpha) \En{ z^s_{k + 1/2} - w^s }^2 \right) \right] + e^s,
        \label{eq:eb-in-alg-d}
    \end{equation}
    where $\tilde{C} \geq \frac{\bar{C}^2}{\tau^2}$, 
    and 
    \begin{equation}
        e^s = \frac{1}{\tilde{C}^2} \sum_{k=0}^{K-1} \bbE^s_0 \En{ z^s_k - w^s }^2 + \frac{5K}{4\tilde{C}^2} \En{z^s_0 - w^s}^2. \nonumber 
    \end{equation}
    % The same statement holds for the iterates $z_k,w_k, z_{k+1/2}$ of
    % the loopless \cref{alg:eg-svrg}.
    \label{lemma:eb-in-alg-d}
\end{lemma}

\begin{proof}
    Let $\tilde{z}^s_{k+1/2} = \Pi_{\mathcal{Z}}(\bar{z}^s_k - \tau F(\bar{z}^s_k))$ denote a ``quasi $(k+\frac{1}{2})$th iterate'' in the $s$th epoch. 
    % We emphasize that this is not computed by the algorithm, but a variant of $z^s_{k+1/2}$ computed using $F(\bar{z}^s_k)$ instead of $F(w^s)$.  
    % \tianlong{(to be made consistent throughout the paper)} 
    Again, this is not computed by the algorithm, but a variant of $z^s_{k+1/2}$ computed using $F(\bar{z}^s_k)$ instead of $F(w^s)$.
    To measure the distance between $\tilde{z}^s_{k+1/2}$ and $z^s_{k+1/2}$, we have 
    % we use non-expensiveness of projection operator and $L$-Lipschitz continuity of operator $F$: 
    \begin{align}
        \En{ \tilde{z}^s_{k+1/2} - z^s_{k+1/2} } & = \En{ \Pi_{\mathcal{Z}}(\bar{z}^s_k - \tau F(\bar{z}^s_k)) - \Pi_{\mathcal{Z}}(\bar{z}^s_k - \tau F(w^s)) } \nonumber \\ 
        & \leq \tau \En{ F(\bar{z}^s_k) - F(w^s) } \tag{non-expansiveness of $\Pi_\cZ$} \nonumber \\ 
        & \leq \tau L \En{ \bar{z}^s_k - w^s } \tag{$L$-Lipschitz of $F$} \nonumber \\ 
        & = \alpha \tau L \En{ z^s_k - w^s }. 
        \label{eq:double-loop-svrg-eg-eb-in-alg-dist-to-dummy}
    \end{align}

    Using~\cref{eq:double-loop-svrg-eg-eb-in-alg-dist-to-dummy}, we have 
    \begin{align}
        & \En{ \bar{z}^s_k - \tilde{z}^s_{k+1/2} }^2 \nonumber \\ 
        \leq& 2 \En{ \bar{z}^s_k - z^s_{k+1/2} }^2 + 2 \En{ \tilde{z}^s_{k+1/2} - z^s_{k+1/2} }^2 \nonumber \\ 
        =& 2 \alpha \En{ z^s_k - z^s_{k+1/2} }^2 + 2 (1 - \alpha) \En{w^s - z^s_{k+1/2}}^2 - 2\alpha(1 - \alpha) \En{z^s_k - w^s}^2 + 2 \alpha^2 \tau^2 L^2 \En{z^s_k - w^s}^2 \nonumber \\ 
        =& 2 \alpha \En{ z^s_k - z^s_{k+1/2} }^2 + 2 (1 - \alpha) \En{w^s - z^s_{k+1/2}}^2 + 2\alpha\left( \alpha \tau^2 L^2 - (1 - \alpha)\right) \En{z^s_k - w^s}^2, 
        % &\leq 4 \alpha^2 \En{ z^s_k - z^s_{k+1/2} }^2 + 4 (1 - \alpha)^2 \En{ w^s - z^s_{k+1/2} }^2 + 2 \alpha^2 \tau^2 L^2 \En{ z^s_k - w^s }^2, 
        \label{eq:right-d}
    \end{align}
    where we use $\En{ a + b }^2 \le 2 \En{ a }^2 + 2 \En{ b }^2$ and $\En{\gamma a + (1 - \gamma) b}^2 = \gamma \En{a}^2 + (1 - \gamma) \En{b}^2 - \gamma(1 - \gamma) \En{a - b}^2$, for $\gamma \in [0, 1]$. 
    % $\bar{z}^s_k - z^s_{k+1/2} = \alpha (z^s_k - z^s_{k+1/2}) + (1 - \alpha) (w^s - z^s_{k+1/2})$. 

    On the other hand, by $\En{ a + b }^2 \le 2 \En{ a }^2 + 2 \En{ b }^2$ and $\Pi_{\cZ^*}(z) = \argmin_{z' \in \cZ^*} \En{z - z'}$, we have 
    % \begin{align}
    %     \En{ \bar{z}^s_k - \Pi_{\mathcal{Z}^*}(\bar{z}^s_k) }^2 &\geq \frac{1}{2} \En{ z^s_k - \Pi_{\mathcal{Z}^*}(\bar{z}^s_k) }^2 - \En{ \bar{z}^s_k - z^s_k }^2 \nonumber \\ 
    %     % &= \frac{1}{2} \En{ z^s_k - \Pi_{\mathcal{Z}^*}(\bar{z}^s_k) }^2 - (1 - \alpha)^2 \En{ z^s_k - w^s }^2 \nonumber \\ 
    %     &\geq \frac{1}{4} \En{ z^s_k - \Pi_{\mathcal{Z}^*}(w^s) }^2 - \frac{1}{2} \En{ \Pi_{\mathcal{Z}^*}(w^s) - \Pi_{\mathcal{Z}^*}(\bar{z}^s_k) }^2 - \En{ \bar{z}^s_k - z^s_k }^2 \nonumber \\ 
    %     &\geq \frac{1}{4} \En{ z^s_k - \Pi_{\mathcal{Z}^*}(w^s) }^2 - \frac{1}{2} \En{ w^s - \bar{z}^s_k }^2 - \En{ \bar{z}^s_k - z^s_k }^2 \nonumber \\ 
    %     &\geq \frac{1}{4} \En{ z^s_k - \Pi_{\mathcal{Z}^*}(w^s) }^2 - \left(\frac{\alpha^2}{2} + (1 - \alpha)^2\right) \En{ z^s_k - w^s }^2
    % \end{align}
    % and
    \begin{align} 
        \En{ \bar{z}^s_k - \Pi_{\mathcal{Z}^*}(\bar{z}^s_k) }^2 \geq & \frac{1}{2} \En{ z^s_0 - \Pi_{\mathcal{Z}^*}(\bar{z}^s_k) }^2 - \En{ \bar{z}^s_k - z^s_0 }^2 \nonumber \\ 
        \geq & \frac{1}{2} \En{ z^s_0 - \Pi_{\mathcal{Z}^*}(z^s_0) }^2 - 2 \alpha^2 \En{ z^s_k - w^s }^2 - 2 \En{z^s_0 - w^s}^2 
        \label{eq:left-d-z}
    \end{align}  
    and 
    \begin{equation}
        \En{ \bar{z}^s_k - \Pi_{\mathcal{Z}^*}(\bar{z}^s_k) }^2 \geq \frac{1}{2} \En{ w^s - \Pi_{\mathcal{Z}^*}(\bar{z}^s_k) }^2 - \En{ \bar{z}^s_k - w^s }^2 \geq \frac{1}{2} \En{ w^s - \Pi_{\mathcal{Z}^*}(w^s) }^2 - \alpha^2 \En{ z^s_k - w^s }^2.
        \label{eq:left-d-w}
    \end{equation} 
    Combining~\cref{eq:left-d-z,eq:left-d-w} with a weight coefficient $\theta$, we obtain 
    \begin{equation}
        \En{ \bar{z}^s_k - \Pi_{\mathcal{Z}^*}(\bar{z}^s_k) }^2 
        \geq \frac{\theta}{2} \En{ z^s_0 - \Pi_{\mathcal{Z}^*}(z^s_0) }^2 + \frac{1 - \theta}{2} \En{ w^s - \Pi_{\mathcal{Z}^*}(w^s) }^2 - 2 \theta \En{z^s_0 - w^s}^2 - ( 1 + \theta ) \alpha^2 \En{ z^s_k - w^s }^2.
        \label{eq:left-d-combined}
    \end{equation} 
    % For each $k = 0, \ldots, K - 1$, combining them with proportional coefficient $\theta_k$
    % \begin{align}
    %     &\En{ \bar{z}^s_k - \Pi_{\mathcal{Z}^*}(\bar{z}^s_k) }^2 \nonumber \\ 
    %     =\;& \theta_k \En{ \bar{z}^s_k - \Pi_{\mathcal{Z}^*}(\bar{z}^s_k) }^2 + (1 - \theta_k) \En{ \bar{z}^s_k - \Pi_{\mathcal{Z}^*}(\bar{z}^s_k) }^2 \nonumber \\ 
    %     \geq\;& \frac{1}{4} \Phi^{\theta_k}_{\Pi_{\mathcal{Z}^*}(w^s)}(z^s_k, w^s) - \left(\theta_k\left(\frac{\alpha^2}{2} + (1 - \alpha)^2\right) + (1-\theta_k)\frac{3\alpha^2}{2}\right) \En{ z^s_k - w^s }^2. 
    %     \label{eq:left-d}
    % \end{align}

    % Let $\theta_k = 0$ for all $k = 0, 1, \ldots, K-1$. 
    % Note that by the tower property of conditional expectations, we have $\bbE^s_0[\bbE^s_k[\cdot]] = \bbE^s_0[\cdot]$ for all $k = 0, 1, \ldots, K - 1$.
    After summing up~\cref{eq:left-d-combined} over $k = 0, \ldots, K-1$ and taking expectation $\EE^s_0$, it follows that 
    \begin{align}
        & \sum_{k = 0}^{K-1} \EE^s_0 \En{ \bar{z}^s_k - \Pi_{\mathcal{Z}^*}(\bar{z}^s_k) }^2 \nonumber \\ 
        \geq & \frac{K}{2} \theta \En{ z^s_0 - \Pi_{\mathcal{Z}^*}(z^s_0) }^2 + \frac{K}{2} (1 - \theta) \En{ w^s - \Pi_{\mathcal{Z}^*}(w^s) }^2 - 2 \theta K \En{z^s_0 - w^s}^2 \nonumber \\ 
        & \hspace{120pt} - ( 1 + \theta ) \alpha^2 \sum_{k=0}^{K-1} \EE^s_0 \En{ z^s_k - w^s }^2 \nonumber \\ 
        \geq & \frac{K}{4} \theta \En{ z^s_0 - \Pi_{\mathcal{Z}^*}(w^s) }^2 - \frac{K}{2} \En{ \Pi_{\cZ^*}(z^s_0) - \Pi_{\mathcal{Z}^*}(w^s) }^2 + \frac{K}{2} (1 - \theta) \En{ w^s - \Pi_{\mathcal{Z}^*}(w^s) }^2 \nonumber \\ 
        & \hspace{20pt} - 2 \theta K \En{z^s_0 - w^s}^2 - ( 1 + \theta ) \alpha^2 \EE^s_0 \sum_{k=0}^{K-1} \En{ z^s_k - w^s }^2 \tag{$\En{a}^2 \geq \frac{1}{2}\En{a + b}^2 - \En{b}^2$} \nonumber \\ 
        \geq & \frac{K}{4} \theta \En{ z^s_0 - \Pi_{\mathcal{Z}^*}(w^s) }^2 - \frac{K}{2} \En{ z^s_0 - w^s }^2 + \frac{K}{4} (1 - \theta) \En{ w^s - \Pi_{\mathcal{Z}^*}(w^s) }^2 \nonumber \\ 
        & \hspace{20pt} - 2 \theta K \En{z^s_0 - w^s}^2 - ( 1 + \theta ) \alpha^2 \EE^s_0 \sum_{k=0}^{K-1} \En{ z^s_k - w^s }^2 \tag{the non-expansiveness of $\Pi_{\cZ^*}$} \nonumber \\ 
        \geq & \frac{K}{4} \distt^\theta_{\cZ^*}(z^s_0, w^s) - \left( 2 \theta + \frac{1}{2} \right) K \En{z^s_0 - w^s}^2 - ( 1 + \theta ) \alpha^2 \sum_{k=0}^{K-1} \EE^s_0 \En{ z^s_k - w^s }^2, 
        % \geq\;& \frac{K}{2} \En{w^s - \Pi_{\mathcal{Z}^*}(w^s)}^2 - \alpha^2 \sum_{k=0}^{K-1}\En{z^s_k - w^s}^2 \nonumber \\
        % =\;& \frac{K\theta}{(1+\theta)} \En{w^s - \Pi_{\mathcal{Z}^*}(w^s)}^2 + \frac{K(1-\theta)}{2(1+\theta)} \En{w^s - \Pi_{\mathcal{Z}^*}(w^s)}^2 - \alpha^2 \sum_{k=0}^{K-1}\En{z^s_k - w^s}^2 \nonumber \\
        % \geq\;& \frac{K\theta}{4(1+\theta)} \En{ z^s_0 - \Pi_{\mathcal{Z}^*}(w^s) }^2 + \frac{K(1 - \theta)}{4(1+\theta)} \En{ w^s - \Pi_{\mathcal{Z}^*}(w^s) }^2 \nonumber \\ 
        % & \hspace{120pt} - \frac{K\theta}{2(1+\theta)} \En{ z^s_0 - w^s }^2 - \frac{3\alpha^2}{2} \sum_{k=0}^{K-1}\En{z^s_k - w^s}^2 \nonumber \\
        % % &\geq \frac{1}{4} \left( \theta \En{z^s_0 - \Pi_{\mathcal{Z}^*}(w^s)}^2 + (K-\theta) \En{w^s - \Pi_{\mathcal{Z}^*}(w^s)}^2 \right) - \frac{3\alpha^2}{2} \sum_{k=0}^{K-1}\En{z^s_k - w^s}^2, 
        % =\;& \frac{K}{4(1+\theta)} \Phi^\theta_{\Pi_{\mathcal{Z}^*}(w^s)}(z^s_0, w^s) - \frac{K\theta}{2(1+\theta)} \En{ z^s_0 - w^s }^2 - \frac{3\alpha^2}{2} \sum_{k=0}^{K-1}\En{z^s_k - w^s}^2.
        \label{eq:left-d-}
    \end{align} 
    where the last inequality follows from the definition of $\distt^\theta_{\cZ^*}(z, w)$. 
    % where the last inequality follows from $\alpha \ge \frac{1}{2}$. 
    % Recall that $\Phi^\theta_{z^*}(z^s_0, w^s) = \theta \En{z^s_0 - z^*}^2 + (1-\theta) \En{w^s - z^*}^2$.
    % This can be an upper bound for the Lyapunov function we are interested (with additional an error term): 
    % $$\left( \theta + \frac{K - \theta}{3} \right) \En{z^s_0 - \Pi_{\mathcal{Z}^*}(w^s)}^2 + \frac{K - \theta}{3} \En{w^s - \Pi_{\mathcal{Z}^*}(w^s)}^2 \le \Phi^\theta_s(w^s) + \frac{2(K - \theta)}{3} \En{z^s_0 - w^s}^2. $$
    % Let $\kappa_0 = \frac{3\alpha^2}{2}$. 
    
    By applying~\cref{lemma:eb-1}, taking expectation $\bbE^s_0$ and using the tower property, we have 
    \begin{equation}
        \bbE^s_0 \En{ \bar{z}^s_k - \Pi_{\mathcal{Z}^*}(\bar{z}^s_k) }^2 \leq \frac{\bar{C}^2}{\tau^2} \bbE^s_0 \En{ \bar{z}^s_k - \tilde{z}^s_{k+1/2} }^2. 
    \end{equation} 
    Let $\tilde{C} \geq \frac{\bar{C}}{\tau}$. After summing up the above inequality over $k = 0, 1, \ldots, K - 1$, it follows that 
    \begin{equation}
        \sum_{k = 0}^{K-1} \bbE^s_0 \En{ \bar{z}^s_k - \Pi_{\mathcal{Z}^*}(\bar{z}^s_k) }^2 \leq \tilde{C}^2 \sum_{k = 0}^{K-1} \bbE^s_0 \En{ \bar{z}^s_k - \tilde{z}^s_{k+1/2} }^2. 
    \end{equation} 
    Combining this with \cref{eq:right-d,eq:left-d-}, we obtain that 
    \begin{align}
        &\frac{K}{4} \distt^\theta_{\cZ^*}(z^s_0, w^s) - \left( 2 \theta + \frac{1}{2} \right) K \En{z^s_0 - w^s}^2 - ( 1 + \theta ) \alpha^2 \sum_{k=0}^{K-1} \bbE^s_0 \En{ z^s_k - w^s }^2 \nonumber \\ 
        \leq & \tilde{C}^2 \sum_{k=0}^{K-1} \left( 2 \alpha \bbE^s_0 \En{ z^s_k - z^s_{k+1/2} }^2 + 2 (1 - \alpha) \bbE^s_0 \En{w^s - z^s_{k+1/2}}^2 - 2\alpha\left( \alpha \tau^2 L^2 - (1 - \alpha)\right) \bbE^s_0 \En{z^s_k - w^s}^2 \right). 
        \label{eq:double-loop-error-bound-combined}
    \end{align} 
    Since $\tau \leq \frac{\sqrt{1 - \alpha}}{L}$, we have $\alpha \tau^2 L^2 - (1 - \alpha) \leq 0$. 
    Then, it follows that 
    \begin{align}
        \frac{K}{4} \distt^\theta_{\cZ^*}(z^s_0, w^s) \leq & \tilde{C}^2 \sum_{k=0}^{K-1} \left( 2 \alpha \bbE^s_0 \En{ z^s_k - z^s_{k+1/2} }^2 + 2 (1 - \alpha) \bbE^s_0 \En{w^s - z^s_{k+1/2}}^2 \right) \nonumber \\ 
        & \hspace{20pt} + ( 1 + \theta ) \alpha^2 \sum_{k=0}^{K-1} \bbE^s_0 \En{ z^s_k - w^s }^2 + \left( 2 \theta + \frac{1}{2} \right) K \En{z^s_0 - w^s}^2. 
        \label{}
    \end{align} 
    Equivalently, 
    \begin{align}
        \frac{K}{8\tilde{C}^2} \distt^\theta_{\cZ^*}(z^s_0, w^s) \leq & \EE^s_0 \left[ \sum_{k=0}^{K-1} \left( \alpha \En{ z^s_{k+1/2} - z^s_k }^2 + (1 - \alpha) \En{ z^s_{k + 1/2} - w^s }^2 \right) \right] \nonumber \\ 
        & \hspace{20pt} + \frac{( 1 + \theta ) \alpha^2}{2\tilde{C}^2} \sum_{k=0}^{K-1} \bbE^s_0 \En{ z^s_k - w^s }^2 + \frac{\left( 2 \theta + \frac{1}{2} \right) K}{2\tilde{C}^2} \En{z^s_0 - w^s}^2. 
        \label{}
    \end{align} 
    Since $\theta \in [0, 1]$ and $\alpha \in [0, 1)$, we can further obtain~\cref{eq:eb-in-alg-d}. 
\end{proof}

\subsection*{Proof of \cref{thm:linear-convergence-for-double-loop}}

\begin{proof}
    % Since 
    % \begin{equation}
    %     {\En{ w^s - z^* }}^2 - {\En{ z^s_{k + 1/2} - w^s }}^2 = \frac{1}{K} \sum_{j=1}^K {\En{ z^{s-1}_j - z^* }}^2 - {\En{ z^s_{k + 1/2} - z^{s-1}_j }}^2, 
    % \end{equation}

    We start from a double-loop version of~\cref{eq:decrease-loopless}. 
    % \ck{let's state what changes in the proof for this version.} 
    % \tianlong{(edit)} 
    Analogous to the analysis for the loopless SVRG-EG, 
    we can obtain that if $\tau \leq \frac{\gamma\sqrt{1 - \alpha}}{L}$, 
    then for any $z^* \in \cZ^*$, 
    \begin{align}
        & \mathbb{E}^s_k{\En{ z^s_{k + 1} - z^* }}^2 \nonumber \\ 
        \leq & \alpha {\En{ z^s_k - z^* }}^2 + (1 - \alpha) {\En{ w^s - z^* }}^2 - (1 - \alpha)(1 - \gamma) {\En{ z^s_{k + 1/2} - w^s }}^2 - \alpha {\En{ z^s_{k+1/2} - z^s_k }}^2. 
        \label{eq:decrease-double-loop}
    \end{align}

    % Let $\theta = \frac{\alpha}{K(1-\alpha) + \alpha}$ and $z^* = z^{s, *} = \Pi^\theta_{\mathcal{Z}^*}(z^s_0, w^s)$. 
    % Denote $\kappa_1 = \frac{1}{8 C^2 ((1 - \alpha)^2 + \alpha^2)} \le \frac{1}{4}$. 
    % Then, by \eqref{eq:decrease-double-loop} for $k=0$ we get for $\beta \in (0, 1]$ 
    % \begin{align}
    %     \mathbb{E}^s_0{\En{ z^s_1 - z^{s, *} }}^2 &\le \alpha {\En{ z^s_0 - z^{s, *} }}^2 + (1 - \alpha) {\En{ w^s - z^{s, *} }}^2 - (1 - \alpha)(1 - \gamma)(1 - \beta) {\En{ z^s_{1/2} - w^s }}^2 \nonumber \\ 
    %     & \hspace{20pt} - (1-\alpha)(1-\gamma)\beta{\En{ z^s_{1/2} - w^s }}^2 - \alpha {\En{ z^s_{1/2} - z^s_0 }}^2 - (1 - \gamma) \mathbb{E}^s_0{\En{ z^s_1 - z^s_{1/2} }}^2 \nonumber \\ 
    %     &\le \alpha {\En{ z^s_0 - z^{s, *} }}^2 + (1 - \alpha) {\En{ w^s - z^{s, *} }}^2 - (1-\alpha)(1-\gamma)\beta{\En{ z^s_{1/2} - w^s }}^2 - \alpha {\En{ z^s_{1/2} - z^s_0 }}^2 \nonumber \\ 
    %     &\le (\alpha - (1-\alpha)(1 - \gamma)\beta\theta\kappa_1) {\En{ z^s_0 - z^{s, *} }}^2 + (1 - \alpha - (1-\alpha)(1 - \gamma)\beta(1-\theta)\kappa_1) {\En{ w^s - z^{s, *} }}^2 \nonumber \\ 
    %     & \hspace{20pt} + (1-\alpha)(1-\gamma)\beta \cdot \frac{\alpha^2}{(1-\alpha)^2} {\En{ z^s_{1/2} - z^s_0 }}^2 - \alpha {\En{ z^s_{1/2} - z^s_0 }}^2, 
    % \end{align}
    % where in the second inequality we drop two negative terms and in the third inequality we apply Lemma~\ref{lemma:eb-in-alg}. 
    Summing~\cref{eq:decrease-double-loop} up over $k = 0, \ldots, K - 1$ and taking expectation $\EE^s_0$ on both sides, we obtain that 
    \begin{align}
        \bbE^s_0 \left[ \sum_{k=0}^{K-1} \EE^s_k{\En{ z^s_{k + 1} - z^* }}^2 \right] \leq & \alpha \sum_{k=0}^{K-1} \EE^s_0{\En{ z^s_k - z^* }}^2 + K(1 - \alpha) {\En{ w^s - z^* }}^2 \nonumber \\ 
        & - \bbE^s_0 \left[ \sum_{k=0}^{K-1} \left( (1 - \alpha)(1 - \gamma) \En{ z^s_{k + 1/2} - w^s }^2 + \alpha \En{ z^s_{k+1/2} - z^s_k }^2 \right) \right]. % \nonumber \\ 
        % & \hspace{130pt} - (1 - \gamma) \sum_{k=0}^{K-1} \mathbb{E}^s_0{\En{ z^s_{k + 1} - z^s_{k + 1/2} }}^2. 
        \label{eq:sum-up-decrease-double-loop}
    \end{align} 
    Note that, in this step, we require the identical $z^*$ over all summands. 
    This is because we need the identical terms $\En{w^s - z^*}^2$ for constructing the Lyapunov function. 

    By the tower property of conditional expectations, we have $\bbE^s_0[\bbE^s_k[\cdot]] = \bbE^s_0[\cdot]$ for all $k = 0, 1, \ldots, K - 1$. Hence,~\cref{eq:sum-up-decrease-double-loop} yields that     
    \begin{align}
        \sum_{k=0}^{K-1} \EE^s_0{\En{ z^s_{k + 1} - z^* }}^2 \leq & \alpha \sum_{k=0}^{K-1} \EE^s_0{\En{ z^s_k - z^* }}^2 + K(1 - \alpha) {\En{ w^s - z^* }}^2 \nonumber \\ 
        & - \bbE^s_0 \left[ \sum_{k=0}^{K-1} \left( (1 - \alpha)(1 - \gamma) \En{ z^s_{k + 1/2} - w^s }^2 + \alpha \En{ z^s_{k+1/2} - z^s_k }^2 \right) \right]. % \nonumber \\ 
        % & \hspace{130pt} - (1 - \gamma) \sum_{k=0}^{K-1} \mathbb{E}^s_0{\En{ z^s_{k + 1} - z^s_{k + 1/2} }}^2. 
        \label{eq:sum-up-decrease-double-loop-}
    \end{align} 

    Using Jensen's inequality and $w^{s+1} = \frac{1}{K} \sum_{k=1}^K z^s_k$, we can lower bound the left hand side of the above inequality as follows: 
    \begin{align}
        & \sum_{k=0}^{K-1} \EE^s_0{\En{ z^s_{k + 1} - z^* }}^2 \nonumber \\ 
        = & \sum_{k=1}^{K} \EE^s_0{\En{ z^s_{k} - z^* }}^2 \nonumber \\ 
        = & (1 - \alpha) \sum_{k=0}^{K - 1} \EE^s_0{\En{ z^s_{k} - z^* }}^2 + \alpha \sum_{k=1}^K \EE^s_0{\En{ z^s_{k} - z^* }}^2 + \alpha\EE^s_0{\En{ z^s_k - z^* }}^2 - \alpha{\En{ z^s_0 - z^* }}^2 \nonumber \\
        \geq & K(1-\alpha) \EE^s_0{\En{ w^{s+1} - z^* }}^2 + \alpha \sum_{k=0}^{K - 1} \EE^s_0{\En{ z^s_{k} - z^* }}^2 + \alpha\EE^s_0{\En{ z^s_k - z^* }}^2 - \alpha{\En{ z^s_0 - z^* }}^2.
        \label{eq:double-loop-jensen-inequality}
    \end{align}

    % (Change the form to || z_0 - z^* ||^2 + || w - z^* ||^2)
    Combining~\eqref{eq:double-loop-jensen-inequality} with~\eqref{eq:sum-up-decrease-double-loop-}, 
    and rearranging terms, 
    we then have 
    \begin{align}
        & \EE^s_0\left[ \alpha{\En{ z^{s+1}_0 - z^* }}^2 + K(1-\alpha) {\En{ w^{s+1} - z^* }}^2 \right] \nonumber \\ 
        \leq & \alpha{\En{ z^s_0 - z^* }}^2 + \sum_{k=0}^{K-1} \EE^s_0{\En{ z^s_{k + 1} - z^* }}^2 - \alpha \sum_{k=0}^{K - 1} \EE^s_0{\En{ z^s_{k} - z^* }}^2 \tag{\cref{eq:double-loop-jensen-inequality} and $z^s_k = z^{s+1}_0$} \nonumber \\ 
        \leq & \alpha {\En{ z^s_0 - z^* }}^2 + K(1 - \alpha) {\En{ w^s - z^* }}^2 \nonumber \\ 
        & \hspace{60pt} - \bbE^s_0 \left[ \sum_{k=0}^{K-1} \left( (1 - \alpha)(1 - \gamma) \En{ z^s_{k + 1/2} - w^s }^2 + \alpha \En{ z^s_{k+1/2} - z^s_k }^2 \right) \right], 
        \label{eq:decrease-double-loop-sum}
    \end{align} 
    where the last inequality follows from~\cref{eq:sum-up-decrease-double-loop-}. 

    Let $\tilde{\theta} = \frac{\alpha}{\alpha + K(1 - \alpha)}$, after applying~\cref{lemma:eb-in-alg-d} we have 
    \begin{align}
        & \EE^s_0\left[ \alpha{\En{ z^{s+1}_0 - z^* }}^2 + K(1-\alpha) {\En{ w^{s+1} - z^* }}^2 \right] \nonumber \\ 
        \leq& \alpha {\En{ z^s_0 - z^* }}^2 + K(1 - \alpha) {\En{ w^s - z^* }}^2 \nonumber \\ 
        & \hspace{80pt} - \frac{1}{2}\EE^s_0 \left[ \sum_{k=0}^{K-1} \left( (1 - \alpha)(1 - \gamma) \En{ z^s_{k + 1/2} - w^s }^2 + \alpha \En{ z^s_{k+1/2} - z^s_k }^2 \right) \right] \nonumber \\ 
        & \hspace{80pt} - \frac{1}{2}\EE^s_0 \left[ \sum_{k=0}^{K-1} \left( (1 - \alpha)(1 - \gamma) \En{ z^s_{k + 1/2} - w^s }^2 + \alpha \En{ z^s_{k+1/2} - z^s_k }^2 \right) \right] \nonumber \\ 
        \leq & \alpha {\En{ z^s_0 - z^* }}^2 + K(1 - \alpha) {\En{ w^s - z^* }}^2 \nonumber \\ 
        & \hspace{80pt} - \frac{1 - \gamma}{2}\EE^s_0 \left[ \sum_{k=0}^{K-1} \left( (1 - \alpha) \En{ z^s_{k + 1/2} - w^s }^2 + \alpha \En{ z^s_{k+1/2} - z^s_k }^2 \right) \right] \nonumber \\ 
        & \hspace{80pt} - \frac{1}{2}\EE^s_0 \left[ \sum_{k=0}^{K-1} \left( (1 - \alpha)(1 - \gamma) \En{ z^s_{k + 1/2} - w^s }^2 + \alpha \En{ z^s_{k+1/2} - z^s_k }^2 \right) \right] \nonumber \\ 
        \leq & \alpha {\En{ z^s_0 - z^* }}^2 + K(1 - \alpha) {\En{ w^s - z^* }}^2 - \frac{1 - \gamma}{2} \frac{K}{8\tilde{C}^2} \distt^\theta_{\cZ^*}(z^s_0, w^s) \nonumber \\ 
        & \hspace{80pt} - \frac{1}{2}\EE^s_0 \left[ \sum_{k=0}^{K-1} \left( (1 - \alpha)(1 - \gamma) \En{ z^s_{k + 1/2} - w^s }^2 + \alpha \En{ z^s_{k+1/2} - z^s_k }^2 \right) \right] \nonumber \\ 
        & \hspace{80pt}+ \frac{1 - \gamma}{2} \frac{1}{\tilde{C}^2} \sum_{k=0}^{K-1} \EE^s_0 \En{ z^s_k - w^s }^2 + \frac{1 - \gamma}{2} \frac{5K}{4\tilde{C}^2} \En{z^s_0 - w^s}^2. 
        \label{eq:d-svrg-eg-crucial-applied-eb}
    \end{align} 

    We can bound the last two terms (the last line) in the above inequality as follows: 
    \begin{align}
        & \frac{1 - \gamma}{2} \frac{1}{\tilde{C}^2} \sum_{k=0}^{K-1} \EE^s_0 \En{ z^s_k - w^s }^2 + \frac{1 - \gamma}{2} \frac{5K}{4\tilde{C}^2} \En{z^s_0 - w^s}^2 \nonumber \\ 
        \leq & \frac{1 - \gamma}{2} \frac{1}{\tilde{C}^2} \sum_{k=0}^{K-1} \EE^s_0 \En{ z^s_k - w^s }^2 + \frac{1 - \gamma}{2} \frac{5K}{4\tilde{C}^2} \sum_{k=0}^{K-1} \EE^s_0 \En{z^s_k - w^s}^2 \nonumber \\ 
        \leq & \frac{(1 - \gamma)(2K + 1)}{2 \tilde{C}^2} \sum_{k=0}^{K-1} \EE^s_0 \En{ z^s_k - w^s }^2 \nonumber \\ 
        \leq & \frac{(1 - \gamma)(2K + 1)}{\tilde{C}^2} \sum_{k=0}^{K-1} \left( \EE^s_0 \En{ z^s_k - z^s_{k+1/2} }^2 + \EE^s_0 \En{w^s - z^s_{k+1/2}}^2 \right) \tag{$\En{a + b}^2 \leq 2 \En{a}^2 + 2 \En{b}^2$} \nonumber \\ 
        \leq & \frac{(1 - \gamma)(2K + 1)}{\alpha(1 - \alpha)\tilde{C}^2} \sum_{k=0}^{K-1} \left( \alpha \EE^s_0 \En{ z^s_k - z^s_{k+1/2} }^2 + (1 - \alpha) \EE^s_0 \En{w^s - z^s_{k+1/2}}^2 \right) \nonumber \\ 
        \leq & \frac{2K + 1}{\alpha(1 - \alpha)\tilde{C}^2} \EE^s_0 \left[ \sum_{k=0}^{K-1} \left( \alpha \En{ z^s_k - z^s_{k+1/2} }^2 + (1 - \alpha)(1 - \gamma) \En{w^s - z^s_{k+1/2}}^2 \right) \right]. 
    \end{align} 

    Since $\tilde{C} \geq \sqrt{\frac{4K + 2}{\alpha(1 - \alpha)}}$, we have $\frac{2K + 1}{\alpha(1 - \alpha)\tilde{C}^2} \leq \frac{1}{2}$. Hence, we can eliminate the last two lines of~\cref{eq:d-svrg-eg-crucial-applied-eb} and then obtain that 
    \begin{align}
        & \EE^s_0\left[ \alpha{\En{ z^{s+1}_0 - z^* }}^2 + K(1-\alpha) {\En{ w^{s+1} - z^* }}^2 \right] \nonumber \\ 
        & \hspace{80pt} \leq \alpha {\En{ z^s_0 - z^* }}^2 + K(1 - \alpha) {\En{ w^s - z^* }}^2 - \frac{(1 - \gamma)K}{16\tilde{C}^2} \distt^\theta_{\cZ^*}(z^s_0, w^s).  
        \label{}
    \end{align} 

    Taking $z^* = \Pi^{\tilde{\theta}}_{\cZ^*}(z^s_0, w^s)$, and dividing $\alpha + K(1 - \alpha)$ on both sides, we have 
    \begin{align}
        & \EE^s_0\left[ \tilde{\theta} {\En{ z^{s+1}_0 - z^* }}^2 + (1 - \tilde{\theta}) {\En{ w^{s+1} - z^* }}^2 \right] \nonumber \\ 
        \leq & \distt^{\tilde{\theta}}_{\cZ^*}(z^s_0, w^s) - \frac{(1 - \gamma)K}{16\tilde{C}^2(\alpha + K(1 - \alpha))} \distt^{\tilde{\theta}}_{\cZ^*}(z^s_0, w^s).  
        \label{}
    \end{align}     

    By the definition of $\text{dist}^{\tilde{\theta}}_{\mathcal{Z}^*}$, we obtain~\cref{eq:linear-convergence-for-double-loop}. 
    % \begin{equation}
    %     \mathbb{E}[\text{dist}^\theta_{\mathcal{Z}^*}(z^{s+1}_0, w^{s+1})] \le \mathbb{E}[\Phi^\theta_{z^{s, *}}(z^{s+1}_0, w^{s+1})] \le \left( 1 - \frac{(1 - \gamma)K}{16\tilde{C}^2(\alpha + K(1 - \alpha))} \right) \mathbb{E}[\text{dist}^\theta_{\mathcal{Z}^*}(z^s_0, w^s)]. 
    %     \label{eq:induction-for-double-loop}
    % \end{equation}
    % By iterating \eqref{eq:induction-for-double-loop}, we obtain the linear convergence. 
\end{proof}

\subsection*{Proof of \cref{crl:linear-convergence-for-double-loop}}

\begin{proof} 
    Let 
    \begin{equation}
        \rho = \frac{(1 - \gamma)K}{16\tilde{C}^2(\alpha + K(1 - \alpha))}. \nonumber 
    \end{equation} 
    By iterating~\cref{thm:linear-convergence-for-double-loop} and taking the total expectation $\bbE = \bbE^0_0$, we have 
    \begin{equation}
        \bbE\left[ \distt^{\tilde{\theta}}_{\cZ^*}(z^s_0, w^s) \right] \leq (1 - \rho)^s \distt^{\tilde{\theta}}_{\cZ^*}(z^0_0, w_0) = (1 - \rho)^s \min_{z^* \in \cZ^*} \En{z_0 - z^*}^2. 
    \end{equation} 
    To reach expected $\epsilon$-accuracy as measured by the weighted distance (Lyapunov function) to the solution set, 
    % i.e., $\textnormal{dist}^\theta_{\mathcal{Z}^*}(z_k, w_k) = \epsilon$, 
    we need $(1 - \rho)^s \leq \epsilon$ and equivalently, 
    \begin{equation}
        s \geq \frac{1}{\log{\frac{1}{1 - \rho}}} \log{\frac{1}{\epsilon}} \approx \frac{1}{\rho} \log{\frac{1}{\epsilon}} \quad \text{when $\rho$ is small}, 
    \end{equation} 
    where $s$ counts the number of epochs (outer loops). 

    Recall that $1 - \alpha = \frac{1}{K} = \frac{2}{N}$, $\gamma = 0.99$ and $\tau = \frac{0.99\sqrt{2}}{\sqrt{N}L}$. 
    Note that 
    \begin{equation}
        \tilde{C}^2 = \max\left\{ \frac{\bar{C}^2}{\tau^2}, \frac{4K + 2}{\alpha(1 - \alpha)} \right\} \leq \frac{\bar{C}^2}{\tau^2} + \frac{4K + 2}{\alpha(1 - \alpha)} = \frac{\bar{C}^2 N L^2}{0.99^2 \times 2} + \frac{2N + 2}{(1 - \frac{2}{N})\frac{2}{N}} = \cO(\bar{C}^2 N L^2 + N^2). 
    \end{equation} 
    Since 
    \begin{align} 
        \frac{1}{\rho} = \frac{16\tilde{C}^2(\alpha + K(1 - \alpha))}{(1 - \gamma)K} = \frac{16\tilde{C}^2(2 - \frac{2}{N})}{0.01 \frac{N}{2}} = \cO(\bar{C}^2 L^2 + N). 
    \end{align} 
    Since we need $2 K + N$ evaluations of $F_\xi$ per epoch, we need $2N \times s$ evaluations of $F_\xi$ to reach expected $\epsilon$-accuracy. 
    Therefore, we attain that the time complexity is $\mathcal{O}\left( (\bar{C}^2 N L^2 + N^2) \log{\frac{1}{\epsilon}} \right)$. 
\end{proof}

\subsection{Convergence of (Deterministic) Extragradient method}

We follows the steps in the proof of linear convergence in~\citet{tseng1995linear} and shows a linear convergence rate for the (deterministic) Extragradient (EG) algorithm. 
We re-state the result as follows: 
\begin{theorem}{\citep[Corollary 3.3]{tseng1995linear}}
    Let Assumptions $1$-$4$ and \textnormal{Error-Bound Condition} hold. 
    Let $\tau = \gamma / L$, for $\gamma \in (0, 1)$, 
    then (deterministic) EG ensures 
    \begin{equation} 
        \distt_{\cZ^*}(z_{k+1}) \leq \left(1 - \tau^2 \frac{(1 - \gamma)^3}{2 \bar{C}^2} \right) \distt_{\cZ^*}(z_k), 
        \label{eq:thm-paul-tseng} 
    \end{equation} 
    where $\distt_{\cZ^*}(z) = \min_{z' \in \cZ^*} \En{z - z'}^2$. 
    % c = \frac{(1 - \gamma)^3}{2C^2 - (1 - \gamma)^3}
    \label{thm:thm-paul-tseng}
\end{theorem}

\begin{proof}
    In this proof, we use the same notation as this work. 
    The iterates in the (deterministic) EG algorithm are updated as the following two steps: 
        $$z_{k+1/2} = \Pi_\cZ \left( z_k - \tau F(z_k) \right) \quad\text{and}\quad z_{k+1} = \Pi_\cZ \left( z_k - \tau F(z_{k+1/2}) \right)$$
    for $k \in \mathbb{N}$. 
    By the property of the ($l_2$ norm) projection onto a convex set, we attain
    \begin{equation}
        \langle z_{k+1/2} - z_k + \tau F(z_k),\, z_{k+1} - z_{k+1/2} \rangle \geq 0 
        \label{eq:det-eg-proj-inequality-1} 
    \end{equation} 
    and 
    \begin{equation}
        \langle z_{k+1} - z_k + \tau F(z_{k+1/2}),\, z^* - z_{k+1} \rangle \geq 0 
        \label{eq:det-eg-proj-inequality-2} 
    \end{equation} 
    for some $z^* \in \cZ^*$. 
    By $2\inp{a}{b} = \En{a}^2 + \En{b}^2 - \En{a - b}^2$, we have 
    \begin{equation}
        2 \langle z_{k+1/2} - z_k,\, z_{k+1} - z_{k+1/2} \rangle = \En{ z_{k+1} - z_k }^2 - \En{ z_{k+1/2} - z_k }^2 - \En{ z_{k+1} - z_{k+1/2} }^2 
        \label{eq:det-eg-inp-identity-1}
    \end{equation}
    and 
    \begin{equation}
        2 \langle z_{k+1} - z_k,\, z^* - z_{k+1} \rangle = \En{ z_k - z^* }^2 - \En{ z_{k+1} - z_k }^2 - \En{ z_{k+1} - z^* }^2. 
        \label{eq:det-eg-inp-identity-2}
    \end{equation} 
    Also, we have 
    \begin{equation}
        \inp{F(z_{k+1/2})}{z^* - z_{k+1/2}} \leq \inp{F(z^*)}{z^* - z_{k+1/2}} \leq 0 
        \label{eq:det-eg-inp-inequality-1}
    \end{equation} 
    and 
    \begin{align} 
        \inp{F(z_k) - F(z_{k+1/2})}{z_{k+1} - z_{k+1/2}} &\leq \En{ F(z_{k+1/2}) - F(z_k) } \En{ z_{k+1} - z_{k+1/2} } \nonumber \\ 
        &\leq L \En{z_{k+1/2} - z_k} \En{z_{k+1} - z_{k+1/2}}. 
        \label{eq:det-eg-inp-inequality-2}
    \end{align}  

Summing up~\cref{eq:det-eg-proj-inequality-1,eq:det-eg-proj-inequality-2}, and using~\cref{eq:det-eg-inp-identity-1,eq:det-eg-inp-identity-2,eq:det-eg-inp-inequality-1,eq:det-eg-inp-inequality-2}, we obtain 
\begin{align}
    & \En{ z_{k+1} - z^* }^2 \nonumber \\ 
    \leq& \En{ z_k - z^* }^2 - \En{ z_k - z_{k+1/2} }^2 - \En{ z_{k+1/2} - z_{k+1} }^2 + 2 \tau L \En{z_{k+1/2} - z_k} \En{ z_{k+1/2} - z_{k+1} }.
    \label{eq:decrease-det-eg}
\end{align}

Let $z^*$ be the solution minimizing the Euclidean distance to $z_k$, i.e., $z^* = \Pi_{\cZ^*}(z_k) = \arg\min_{z^* \in \mathcal{Z}^*} \En{ z_k - z^* }$.
\begin{align}
    & \distt_{\cZ^*}(z_{k+1}) \nonumber \\ 
    \leq & \En{ z_{k+1} - z^*_k }^2 \nonumber \\ 
    \leq & \En{ z_k - z^*_k }^2 - \En{ z_k - z_{k+1/2} }^2 - \En{ z_{k+1/2} - z_{k+1} }^2 + 2 \tau L \En{ z_k - z_{k+1/2} } \En{ z_{k+1/2} - z_{k+1} } \tag{\cref{eq:decrease-det-eg}} \nonumber \\ 
    \leq & \En{ z_k - z^*_k }^2 - \En{ z_k - z_{k+1/2} }^2 - \En{ z_{k+1/2} - z_{k+1} }^2 + \tau L \En{ z_k - z_{k+1/2} }^2 + \tau L \En{ z_{k+1/2} - z_{k+1} }^2 \nonumber \\ 
    = & \En{ z_k - z^*_k }^2 - (1 - \tau L) \left( \En{ z_k - z_{k+1/2} }^2 + \En{ z_{k+1/2} - z_{k+1} }^2 \right) \nonumber \\ 
    \leq & \En{ z_k - z^*_k }^2 - \frac{1}{2}(1 - \tau L) \En{ z_k - z_{k+1} }^2 \nonumber \\ 
    = & \distt_{\cZ^*}(z_k) - \frac{1}{2}(1 - \gamma) \En{ z_k - z_{k+1} }^2,
    \label{eq:decrease-eg-psi}
\end{align}
where in the first inequality and the last equality we use the definition of $\distt_{\cZ^*}(\cdot)$, 
in the last two inequalities we use $2ab \leq a^2 + b^2$ and $ 2\En{a}^2 + 2\En{b}^2 \geq \En{a + b}^2$, 

By the triangle inequality, the non-expansiveness of the projection, and the Lipschitz continuity of $F$, we have 
\begin{align}
    \En{ z_k - z_{k+1} } \geq\;& \En{ z_k - z_{k+1/2} } - \En{ z_{k+1/2} - z_{k+1} } \nonumber \\
    =\;& \En{ z_k - z_{k+1/2} } - \En{ \Pi_\cZ \left( z_k - \tau F(z_k) \right) - \Pi_\cZ \left( z_k - \tau F(z_{k+1/2}) \right) } \nonumber \\
    \geq\;& \En{ z_k - z_{k+1/2} } - \tau \En{ F(z_k) - F(z_{k+1/2}) } \nonumber \\
    \geq\;& (1 - \tau L) \En{ z_k - z_{k+1/2} } = (1 - \gamma) \En{ z_k - z_{k+1/2}}. 
    \label{eq:right-eg}
\end{align}
% From Lemma 3.1, we have
% $$|| z_k - z_{k+1} || \geq (1 - \tau L) || R_{\tau}(z_k) ||.$$

The error-bound condition tells us  
\begin{equation}
    \En{z_k - \Pi_{\cZ^*}(z_k)} \le \frac{\bar{C}^2}{\tau^2} \En{ z_k - z_{k+1/2} }^2.
    \label{eq:eb-eg}
\end{equation}

Then, combining~\cref{eq:decrease-eg-psi,eq:right-eg,eq:eb-eg}, we have  
\begin{align}
    \distt_{\cZ^*}(z_{k+1}) \leq \distt_{\cZ^*}(z_k) - \frac{1}{2}(1 - \gamma) \En{ z_k - z_{k+1} }^2 
    &\leq \distt_{\cZ^*}(z_k) - \frac{1}{2}(1 - \gamma)^3 \En{ z_k - z_{k+1/2} }^2 \nonumber \\ 
    &\leq \distt_{\cZ^*}(z_k) - \tau^2 \frac{(1 - \gamma)^3}{2 \bar{C}^2} \distt_{\cZ^*}(z_k). 
\end{align}

% Let $\tau = \frac{\gamma}{L}$.
% (e.g. $\gamma = 0.99$), 
% Then, $\psi(z_{k+1}) \leq \left( 1 - \frac{1}{2 C^2} (1 - \gamma)^3 \right) \psi(z_k)$.
% By iterating over $k = 0, 1, \ldots$, it follows that 
% \begin{equation}
%     \psi(z_k) \le \rho^k \psi(z_0), 
% \end{equation}
% where $\rho = \frac{1}{1 + c}$ and $c = \frac{(1 - \gamma)^3}{2C^2 - (1 - \gamma)^3} \ge \frac{(1 - \gamma)^3}{2 C^2}$, 
% i.e., we obtain the linear convergence rate. 
\end{proof}

\begin{corollary} 
    Set $\tau = \frac{\gamma}{L}$ in the deterministic EG algorithm. 
    Then, the time complexity to reach $\epsilon$-accuracy is $\mathcal{O}\left( \bar{C}^2 N L^2 \log(\frac{1}{\epsilon}) \right)$. 
\end{corollary} 

\begin{proof} 
    Let 
    \begin{equation}
        \rho = \tau^2 \frac{(1 - \gamma)^3}{2 \bar{C}^2}. \nonumber 
    \end{equation} 
    By iterating~\cref{eq:thm-paul-tseng}, we have 
    \begin{equation}
        \distt_{\cZ^*}(z_k) \leq (1 - \rho)^k \distt_{\cZ^*}(z_0) = (1 - \rho)^k \min_{z^* \in \cZ^*} \En{z_0 - z^*}^2. 
    \end{equation} 
    To reach $\epsilon$-accuracy as measured by $\distt_{\cZ^*}$, 
    % i.e., $\textnormal{dist}^\theta_{\mathcal{Z}^*}(z_k, w_k) = \epsilon$, 
    we need $(1 - \rho)^k \leq \epsilon$ and equivalently, 
    \begin{equation}
        k \geq \frac{1}{\log{\frac{1}{1 - \rho}}} \log{\frac{1}{\epsilon}} \approx \frac{1}{\rho} \log{\frac{1}{\epsilon}} \quad \text{when $\rho$ is small}, 
    \end{equation} 
    where $k$ counts the number of iterations. 
    % Recall that $1 - \alpha = p = \frac{2}{N}$, $\gamma = 0.99$ and $\tau = \frac{0.99\sqrt{2}}{\sqrt{N}L}$. 
    Since $\tau = \frac{\gamma}{L}$, $\gamma = 0.99$
    \begin{align} 
        \frac{1}{\rho} = \frac{2\bar{C}^2}{(1 - \gamma)^3 \tau^2} = \frac{2\bar{C}^2 L^2}{0.01^3 \times 0.99^2} = \cO(\bar{C}^2 L^2). 
    \end{align} 
    Since we need $N$ evaluations of (equivalent) $F_\xi$ per iteration, we need $N \times k$ evaluations of $F_\xi$ to reach $\epsilon$-accuracy. 
    Therefore, we attain that the time complexity is $\mathcal{O}\left( \bar{C}^2 N L^2 \log{\frac{1}{\epsilon}} \right)$. 
\end{proof} 

\paragraph*{Remark.} While it is possible to obtain a smaller Lipschitz constant $L_f$ for $F$ than $L$, the difference between $L$ and $L_f$ is not increasing with the size of the problem. 

\section{LINEAR CONVERGENCE UNDER WEAK SHARPNESS} 
\label{app:linear-convergence-under-weak-sharpness}

\subsection*{Proof of~\cref{lem:rsi}} 

\begin{proof}
    By definition, if the solution set of the VI is weakly sharp, we have~\cref{eq:weak-sharpness-VI}. This implies that there exists a positive number $\mu$ such that 
    \begin{equation}
        \mu \bB \subset F(z^*) + \left[ T_{\mathcal{Z}}(z') \cap N_{\mathcal{Z}^*}(z') \right]^\circ, 
        \label{eq:weak-sharpness-geo-2} 
    \end{equation}
    for every $z^*, z' \in \mathcal{Z}^*$, where $\bB$ represents a unit ball in $V$. 
    It is known that~\cref{eq:weak-sharpness-geo-2} is equivalent to 
    \begin{equation}
        \langle F(z^*), z \rangle \geq \alpha \En{ z }, \quad z \in T_{\mathcal{Z}}(z') \cap N_{\mathcal{Z}^*}(z'), 
        \label{eq:weak-sharpness-geo-equivalence} 
    \end{equation}
    for every $z^*, z' \in \mathcal{Z}^*$ (e.g., see~\citet{marcotte1998weak}). 
    Indeed, if~\cref{eq:weak-sharpness-geo-2} hold, for each $y \in \bB$ we have 
    $\mu y - F(z^*) \subset \left[ T_{\mathcal{Z}}(z') \cap N_{\mathcal{Z}^*}(z') \right]^\circ$ for every $z^*, z' \in \mathcal{Z}^*$. 
    It follows that $\inp{\mu y - F(z^*)}{z} \leq 0$ for all $z \in T_{\mathcal{Z}}(z') \cap N_{\mathcal{Z}^*}(z')$. 
    Taking $y = \frac{z}{\En{z}}$, we attain~\cref{eq:weak-sharpness-geo-equivalence}. 
    We omitted the proof for the other direction here. 
    As it can be verified that $z - \Pi_\cZ(z) \in T_{\mathcal{Z}}(\Pi_\cZ(z)) \cap N_{\mathcal{Z}^*}(\Pi_\cZ(z))$ for all $z \in \mathcal{Z}^*$, 
    \cref{eq:weak-sharpness-error-bound} follows from~\cref{eq:weak-sharpness-geo-equivalence}. 
\end{proof}
% Note that this inequality only holds for $z_p$, not necessarily for other $z^* \in \mathcal{Z}^*$. 
\paragraph*{Remark.} There are some works using~\cref{eq:weak-sharpness-error-bound} (or similar inequalities) as the definition of ``weak sharpness''~\citep{nguyen2021weak,kannan2019optimal}. 
While we assume a more classical assumption in this paper, our results can be generalized to their settings naturally. 

Before proving \cref{thm:linear-convergence-for-loopless-weak-sharpness}, we first provide a technical lemma.
Recall that, in the proof of \cref{thm:linear-convergence-for-loopless}, to derive \cref{eq:decrease-loopless-lyapunov-func-fixed-stepsize}, 
we used \cref{eq:VI-inequality} to bound $\mathbb{E}_k\left[\langle \hat{F}(z_{k + 1/2}), \, z^* - z_{k + 1/2} \rangle\right]$. Here, instead, we need a tighter bound for the same term.
\begin{lemma}
    % Let $\theta = \frac{\alpha}{\alpha + \frac{1 - \alpha}{p}}$. 
    Let Assumptions $1$-$4$ and \textnormal{Weak Sharpness} hold. 
    Let $\{z_k, w_k, z_{k+1/2}\}_{k \in \bbN_+}$ be generated by~\cref{alg:eg-svrg} with 
    $p \in (0, 1]$, $\alpha \in [0, 1)$, $\gamma \in (0, 1)$ and $\tau = \frac{\gamma\sqrt{1 - \alpha}}{L}$. 
    Then, for $\theta \in [0, 1]$ and any $z^* \in \cZ^*$, 
    % and $\lambda \in [0, 1]$. 
    \begin{equation}
        c \, \distt^\theta_{\cZ^*}(z_k, w_k) \leq 2 \tau\langle F(z_{k+1/2}),\, z_{k+1/2} - z^* \rangle + 2 c \, \Phi^\theta_{z_{k+1/2}}(z_k, w_k), 
        \label{eq:rsi-in-alg}
    \end{equation} 
    where $0 < c \leq \tau\frac{\mu}{D(\cZ)}$. 
    % whenever $\En{ z_{k+1/2} - \theta z_k - (1 - \theta) w_k } \leq \zeta := \min\{ \frac{\mu}{2L}, \sqrt{c} \}$. 
    \label{lem:rsi-in-alg}
\end{lemma}
\begin{proof}
    % Let $\tilde{z}_k = \theta z_k + (1 - \theta) w_k$. 
    Let $z^*$ be any point in $\cZ^*$. 
    By $\En{ a + b }^2 \leq 2 \En{ a }^2 + 2 \En{ b }^2$, we have 
    \begin{equation}
        \En{ z_{k+1/2} - z^* }^2 \geq \frac{1}{2} \En{ z_k - z^* }^2 - \En{ z_k - z_{k+1/2} }^2, 
        \label{eq:weak-sharpness-eb-left-1}
    \end{equation}
    and 
    \begin{equation}
        \En{ z_{k+1/2} - z^* }^2 \geq \frac{1}{2} \En{ w_k - z^* }^2 - \En{ w_k - z_{k+1/2} }^2. 
        \label{eq:weak-sharpness-eb-left-2}
    \end{equation}
    Summing up~\cref{eq:weak-sharpness-eb-left-1} and~\cref{eq:weak-sharpness-eb-left-2}, we attain 
    \begin{align}
        & \En{ z_{k+1/2} - z^* }^2 \nonumber \\ 
        \geq & \frac{1}{2} \left( \theta \En{ z_k - z^* }^2 + (1 - \theta) \En{ w_k - z^* }^2 \right) - \theta \En{ z_k - z_{k+1/2} }^2 - (1 - \theta) \En{ w_k - z_{k+1/2} }^2 \nonumber \\ 
        = & \frac{1}{2} \Phi^\theta_{z^*}(z_k, w_k) - \Phi^\theta_{z_{k+1/2}}(z_k, w_k) \nonumber \\ 
        \geq & \frac{1}{2} \distt^\theta_{\cZ^*}(z_k, w_k) - \Phi^\theta_{z_{k+1/2}}(z_k, w_k). 
        \label{eq:weak-sharpness-eb-left-combined}
    \end{align} 
    by the definitions of $\Phi^\theta_z(z_k, w_k)$ and $\distt^\theta_{\cZ^*}(z_k, w_k)$. 
    On the other hand, 
    \begin{align}
        & \langle F(z_{k+1/2}),\, z_{k+1/2} - z^* \rangle \nonumber \\ 
        \geq& \langle F(z^*),\, z_{k+1/2} - z^* \rangle \nonumber \\ 
        =& \langle F(z^*),\, z_{k+1/2} - \Pi_{\cZ^*}(z_{k+1/2}) \rangle + \langle F(z^*),\, \Pi_{\cZ^*}(z_{k+1/2}) - z^* \rangle \nonumber \\ 
        \geq& \langle F(z^*),\, z_{k+1/2} - \Pi_{\cZ^*}(z_{k+1/2}) \rangle \nonumber \\ 
        \geq& \mu \En{z_{k+1/2} - \Pi_{\cZ^*}(z_{k+1/2})} \nonumber \\ 
        \geq& \frac{\mu}{D(\cZ)} \En{z_{k+1/2} - \Pi_{\cZ^*}(z_{k+1/2})}^2 \tag{$\En{z_{k+1/2} - \Pi_{\cZ^*}(z_{k+1/2})} \leq D(\cZ)$}. 
    \end{align}

    Therefore, we obtain that 
    \begin{align}
        \tau \langle F(z_{k+1/2}),\, z_{k+1/2} - z^* \rangle \geq & \tau\frac{\mu}{D(\cZ)} \En{z_{k+1/2} - \Pi_{\cZ^*}(z_{k+1/2})}^2 \nonumber \\ 
        \geq & c \En{z_{k+1/2} - \Pi_{\cZ^*}(z_{k+1/2})}^2 \nonumber \\ 
        \geq & \frac{c}{2} \distt^\theta_{\cZ^*}(z_k, w_k) - c \Phi^\theta_{z_{k+1/2}}(z_k, w_k), 
    \end{align}
    where the last inequality follows from~\cref{eq:weak-sharpness-eb-left-combined}. Hence, we have~\cref{eq:rsi-in-alg}. 
\end{proof}

% \begin{theorem}
%     Let $\{ z_k, w_k, z_{k+1/2} \}_{k \in \bbN^+}$ be a sequence of iterates generated by (loopless) \textnormal{\textsf{SVRG-EG}}, with parameters $\alpha \geq \frac{1}{2}$, $p \in (0, 1]$, $\gamma \in (0, 1)$ and $\tau = \frac{\gamma\sqrt{1 - \alpha}}{L_f}$. We have 
%     \begin{equation}
%         \EE_k[ \dist^\theta_{k+1} ] \leq \left( 1 - \frac{(1 - \gamma)c}{16\gamma^2\left(\alpha + \frac{1 - \alpha}{p}\right)} \right) \dist^\theta_k , 
%     \end{equation}
%     where $\theta = \frac{\alpha}{\alpha + \frac{1 - \alpha}{p}}, c = \min\{ \frac{\tau \mu}{2},\, 2\gamma^2(1 - \alpha) \}$. 
%     % and $\zeta = \min\{ \frac{\mu}{2L}, \sqrt{c} \}$.  
%     % for all sufficiently large $k$ such that $\En{ z_{k+1/2} - \theta z_k - (1 - \theta) w_k } \leq \zeta$. 
% \end{theorem}

\subsection*{Proof of~\cref{thm:linear-convergence-for-loopless-weak-sharpness}} 

\begin{proof} 

Let $\tilde{\theta} = \frac{\alpha}{\alpha + \frac{1 - \alpha}{p}}$.  
% Denote $\tilde{z}_k = \tilde{\theta} z_k + (1 - \tilde{\theta}) w_k$. Note that $\Pi_{\cZ^*}(\tilde{z}_k) = \Pi^{\tilde{\theta}}_{\cZ^*}(z_k, w_k)$ since 
% \begin{align}
%     \Pi_{\cZ^*}(\tilde{z}_k) &= \arg\min_{z \in \mathcal{Z}^*} \En{ \tilde{\theta} z_k + (1 - \tilde{\theta}) w_k - z }^2 \nonumber \\ 
%     &= \arg\min_{z \in \mathcal{Z}^*} \En{ \tilde{\theta} (z_k - z) + (1 - \tilde{\theta}) (w_k - z) }^2 \nonumber \\ 
%     &= \arg\min_{z \in \mathcal{Z}^*} \tilde{\theta} \En{ z_k - z }^2 + (1 - \tilde{\theta}) \En{ w_k - z }^2 - \tilde{\theta}(1 - \tilde{\theta}) \En{ z_k - w_k }^2 \nonumber \\ 
%     &= \Pi^{\tilde{\theta}}_{\cZ^*}(z_k, w_k), 
% \end{align}
% where in the second to last equality we use the identity $\En{ \lambda a + (1 - \lambda) b }^2 = \lambda \En{ a }^2 + (1 - \lambda) \En{ b }^2 - \lambda(1 - \lambda) \En{ a - b }^2$, for $\lambda \in [0, 1]$. 
Recall that the crucial inequality~\cref{eq:decrease-loopless-lyapunov-func-fixed-stepsize} in our proof (when $\tau = \frac{\gamma \sqrt{1 - \alpha}}{L}$) is 
\begin{align}
    & \bbE_k \left( \alpha {\En{ z_{k + 1} - z^* }}^2 + \frac{1 - \alpha}{p} {\En{ w_{k + 1} - z^* }}^2 \right) \nonumber \\ 
    \leq& \alpha {\En{ z_k - z^* }}^2 + \frac{1 - \alpha}{p} {\En{ w_k - z^* }}^2 - (1-\gamma)(1-\alpha) \En{ w_k - z_{k + 1/2} }^2 - \alpha \En{ z_k - z_{k+1/2} }^2. \nonumber
\end{align} 

Instead of using~\cref{eq:VI-inequality}  
\begin{align}
    \mathbb{E}_k\left[\langle \hat{F}(z_{k + 1/2}), \, z^* - z_{k + 1/2} \rangle\right] = - \langle F(z_{k + 1/2}), \, z_{k + 1/2} - z^* \rangle \leq - \inp{ F(z^*) }{ z_{k + 1/2} - z^* } \leq 0, \nonumber 
\end{align} 
in the derivation of~\cref{eq:decrease-loopless-lyapunov-func-fixed-stepsize}, 
we use~\cref{lem:rsi-in-alg} as follows: 
\begin{align}
    2 \tau \mathbb{E}_k\left[\langle \hat{F}(z_{k + 1/2}), \, z^* - z_{k + 1/2} \rangle\right] \leq & - 2 \tau \inp{ F(z^*) }{ z_{k + 1/2} - z^* } \nonumber \\ 
    \leq & -c \, \distt^{\tilde{\theta}}_{\cZ^*}(z_k, w_k) + 2c\, \Phi^{\tilde{\theta}}_{k+1/2}(z_k, w_k) 
\end{align} 
and attain that 
% By~\cref{lem:rsi-in-alg}, we have 
\begin{align}
    & \bbE_k \left( \alpha {\En{ z_{k + 1} - z^* }}^2 + \frac{1 - \alpha}{p} {\En{ w_{k + 1} - z^* }}^2 \right) \nonumber \\ 
    \leq& \alpha {\En{ z_k - z^* }}^2 + \frac{1 - \alpha}{p} {\En{ w_k - z^* }}^2 -c \, \distt^{\tilde{\theta}}_{\cZ^*}(z_k, w_k) \nonumber \\ 
    & \quad - (1-\gamma)(1-\alpha) \En{ w_k - z_{k + 1/2} }^2 - \alpha \En{ z_k - z_{k+1/2} }^2 + 2c\, \Phi^{\tilde{\theta}}_{k+1/2}(z_k, w_k). 
    \label{eq:crucial-inequality-applied-weak-sharpness}
\end{align} 

Recall that $\Phi^{\tilde{\theta}}_{k+1/2}(z_k, w_k) = \tilde{\theta} \En{ w_k - z_{k + 1/2} }^2 + (1 - \tilde{\theta}) \En{ z_k - z_{k+1/2} }^2$. 
Let $c \leq \frac{(1 - \gamma)(1 - \alpha)}{2}$. Given $\alpha \geq \frac{1}{2}$, we have 
$$2c (1 - \tilde{\theta}) \leq (1-\gamma)(1-\alpha)$$ 
and 
$$2c \tilde{\theta} \leq (1 - \gamma)(1 - \alpha) \tilde{\theta} \leq (1 - \alpha) \leq \alpha. $$ 
This means the last line of~\cref{eq:crucial-inequality-applied-weak-sharpness} is non-positive. 
Hence, letting $z^* = \Pi^{\tilde{\theta}}_{\cZ^*}(z_k, w_k)$, we have 
% \begin{align}
%     \EE_k \En{ z_{k+1} - \tilde{z}^p_k }^2 \leq& \alpha \En{ z_k - \tilde{z}^p_k }^2 + (1 - \alpha) \En{ w_k - \tilde{z}^p_k }^2 - \frac{\lambda c}{2} \dist^\theta_k. % \nonumber \\ 
%     % & \quad - \frac{(1 - \gamma)(1 - \alpha)}{2} \En{ w_k - z_{k+1/2} }^2 - \frac{1 - \gamma}{2} \EE_k \En{ z_{k+1} - z_{k + 1/2} }^2. 
%     \label{eq:decrease}
% \end{align}
% Since by definition of $w_{k+1}$ 
% \begin{equation}
%     \EE_{k+1/2} \En{ w_{k+1} - z }^2 = p \En{ z_{k+1} - z }^2 + (1 - p) \En{ w_k - z }^2
% \end{equation}
% for all $z \in \mathcal{Z}$, 
% by taking $\EE_k$ on both sides and the tower property, we have 
% \begin{equation}
%     - \frac{1 - \alpha}{p} p \EE_k \En{ z_{k+1} - \tilde{z}^p_k }^2 + \frac{1 - \alpha}{p} \EE_k \En{ w_{k+1} - \tilde{z}^p_k }^2 = \frac{1 - \alpha}{p} (1 - p) \En{ w_k - \tilde{z}^p_k }^2. 
% \end{equation}
% After adding this to~\cref{eq:decrease}, we obtain 
\begin{align}
    & \left(\alpha + \frac{1 - \alpha}{p}\right)\bbE_k \left( \tilde{\theta} {\En{ z_{k + 1} - z^* }}^2 + (1 - \tilde{\theta}) {\En{ w_{k + 1} - z^* }}^2 \right) \nonumber \\ 
    \leq & \left(\alpha + \frac{1 - \alpha}{p}\right) \distt^{\tilde{\theta}}_{\cZ^*}(z_k, w_k) -c \, \distt^{\tilde{\theta}}_{\cZ^*}(z_k, w_k). 
\end{align}
Therefore, we have 
\begin{align}
    \EE_k[ \distt^{\tilde{\theta}}_{\cZ^*}(z_{k+1}, w_{k+1}) ] \leq \left( 1 - \frac{c'}{2 (\alpha + \frac{1 - \alpha}{p})} \right) \distt^{\tilde{\theta}}_{\cZ^*}(z_k, w_k), 
\end{align} 
where $c' \leq 2c \leq (1 - \gamma)(1 - \alpha)$ and $c' \leq c \leq \tau \frac{\mu}{D(\cZ)}$. 
\end{proof} 

\subsection*{Proof of~\cref{crl:linear-convergence-for-loopless-weak-sharpness}} 

\begin{proof} 
    Let 
    \begin{equation}
        \rho = \frac{c}{2 \left( \alpha + \frac{1 - \alpha}{p} \right)}. \nonumber 
    \end{equation} 
    By iterating~\cref{thm:linear-convergence-for-loopless} and taking the total expectation $\bbE_0$, we have 
    \begin{equation}
        \bbE_0\left[ \distt^{\tilde{\theta}}_{\cZ^*}(z_k, w_k) \right] \leq (1 - \rho)^k \distt^{\tilde{\theta}}_{\cZ^*}(z_0, w_0) = (1 - \rho)^k \min_{z^* \in \cZ^*} \En{z_0 - z^*}^2. 
    \end{equation} 
    To reach expected $\epsilon$-accuracy as measured by the weighted distance (Lyapunov function) to the solution set, 
    % i.e., $\textnormal{dist}^\theta_{\mathcal{Z}^*}(z_k, w_k) = \epsilon$, 
    we need $(1 - \rho)^k \leq \epsilon$ and equivalently, 
    \begin{equation}
        k \geq \frac{1}{\log{\frac{1}{1 - \rho}}} \log{\frac{1}{\epsilon}} \approx \frac{1}{\rho} \log{\frac{1}{\epsilon}} \quad \text{when $\rho$ is small}, 
    \end{equation} 
    where $k$ counts the number of iterations. 

    Recall that $1 - \alpha = p = \frac{2}{N}$, $\gamma = 0.99$ and $\tau = \frac{0.99\sqrt{2}}{\sqrt{N}L}$. 
    Note that 
    \begin{align} 
        \frac{1}{c} = \max\left\{ \frac{D(\cZ)}{\tau \mu}, \frac{1}{(1 - \gamma)(1 - \alpha)} \right\} \leq \frac{D(\cZ)}{\tau \mu} + \frac{1}{(1 - \gamma)(1 - \alpha)} &= \frac{D(\cZ)\sqrt{N} L}{0.99 \sqrt{2} \mu} + \frac{N}{0.01 \times 2} \nonumber \\ 
        &= \cO\left( \frac{\sqrt{N} L}{\mu} + N \right). 
    \end{align}  
    Since 
    \begin{align} 
        \frac{1}{\rho} = \frac{2(\alpha + \frac{1 - \alpha}{p})}{c} = \frac{2(2 - \frac{2}{N})}{c} \leq \cO\left( \frac{\sqrt{N} L}{\mu} + N \right). 
    \end{align} 
    Since we need $2 + p N$ evaluations of $F_\xi$ per iteration, we need $4 \times k$ evaluations of $F_\xi$ to reach expected $\epsilon$-accuracy. 
    Therefore, we attain that the time complexity is $\cO\left( \left( \frac{\sqrt{N} L}{\mu} + N \right) \log{\frac{1}{\epsilon}} \right)$. 
\end{proof}

\section{INCREASING ITERATE AVERAGING SCHEMES (IIAS) FOR SVRG-EG}

\subsection{Proofs of Lemmas}
\label{subsec:app-iias-lemmas}

\subsection*{Proof of \cref{lemma:poly-sum-bound}}

\begin{proof} 
    We prove this lemma for both the loopless SVRG-EG and the double-loop SVRG-EG in order. 
    In this section, we denote $\cZ^*$ as the solution set to~\cref{prob:generalized-vi}. 

    (i) For the loopless version of SVRG-EG, recall that we have (different from~\cref{eq:decrease-loopless-lyapunov-func-fixed-stepsize}, in deriving the following inequality we retain $\bbE_k \En{ z_{k + 1} - z_{k + 1/2} }^2$ and omit $\alpha \En{z_k - z_{k+1/2}}^2$ instead) for any $z^* \in \cZ^*$: 
    \begin{align}
        &\bbE_k \left( \alpha {\En{ z_{k + 1} - z^* }}^2 + \frac{1 - \alpha}{p} {\En{ w_{k + 1} - z^* }}^2 \right) \nonumber \\ 
        \leq& \alpha {\En{ z_k - z^* }}^2 + \frac{1 - \alpha}{p} {\En{ w_k - z^* }}^2 \nonumber \\ 
        & \hspace{80pt}- (1-\gamma)(1-\alpha) \En{ w_k - z_{k + 1/2} }^2 - (1 - \gamma) \bbE_k \En{ z_{k + 1} - z_{k + 1/2} }^2. 
    \end{align}
    
    Denote $\tilde{\Phi}_k(z) = \alpha {\En{ z_k - z }}^2 + \frac{1 - \alpha}{p} {\En{ w_k - z }}^2$ for any $z \in \dom\, g$. 
    Then, we have 
    \begin{equation}
        \bbE_k \tilde{\Phi}_{k+1}(z^*) \leq \tilde{\Phi}_k(z^*) - (1-\gamma)(1-\alpha) \En{ w_k - z_{k + 1/2} }^2 - (1 - \gamma) \bbE_k \En{ z_{k + 1} - z_{k + 1/2} }^2. 
    \end{equation}

    Taking expectation $\bbE = \bbE_0$, using the tower property of conditional expectations, we obtain that 
    \begin{equation}
        \bbE \tilde{\Phi}_{k+1}(z^*) \leq \bbE \tilde{\Phi}_k(z^*) - (1-\gamma)(1-\alpha) \bbE \En{ w_k - z_{k + 1/2} }^2 - (1 - \gamma) \bbE \En{ z_{k + 1} - z_{k + 1/2} }^2. 
        \label{eq:iias-descent-inequality-loopless}
    \end{equation} 
    
    Summing up the above inequalities over $k = 0, \ldots, K - 1$ with increasing weights $k^q$, we have 
    \begin{align} 
        & \sum_{k=0}^{K-1} k^q \left( (1 - \gamma) \bbE \En{ z_{k + 1} - z_{k + 1/2} }^2 + (1 - \alpha)(1 - \gamma) \bbE \En{ z_{k + 1/2} - w_k }^2 \right) \nonumber \\ 
        \leq & \sum_{k=0}^{K-1} k^q \left( \mathbb{E}\tilde{\Phi}_k(z^*) - \mathbb{E}\tilde{\Phi}_{k+1}(z^*) \right) \tag{\cref{eq:iias-descent-inequality-loopless}} \nonumber \\ 
        = & \sum_{k=1}^{K} \big( k^q - (k-1)^q \big) \mathbb{E} \tilde{\Phi}_k(z^*) - K^q \mathbb{E} \tilde{\Phi}_K(z^*) \nonumber \\ 
        \leq & \tilde{\Phi}_0(z^*) \sum_{k=1}^{K} \big( k^q - (k-1)^q \big) = K^q \tilde{\Phi}_0(z^*), 
    \end{align} 
    where the last inequality follows from $\bbE \tilde{\Phi}_k(z^*) \leq \tilde{\Phi}_0(z^*)$ for any $z^* \in \cZ^*$ and $k^q - (k-1)^q \geq 0$ for each $k$. 

    % When $\cZ = \dom\,g$ is bounded, we can conclude by bounding $\tilde{\Phi}_0(z)$ with $\left( \alpha + \frac{1 - \alpha}{p} \right) D(\cZ) = \left( \alpha + \frac{1 - \alpha}{p} \right) \max_{z, z'} \En{z - z'}^2$. Otherwise, we can bound $\tilde{\Phi}_0(z)$ with $\left( \alpha + \frac{1 - \alpha}{p} \right) \max_{z^* \in \cZ^*} \En{z_0 - z}^2$ instead. 

    (ii) For the double-loop version of SVRG-EG, recall that we have (different from~\cref{eq:decrease-double-loop-sum}, in deriving the following inequality we retain $\bbE_k \En{ z^s_{k + 1} - z^s_{k + 1/2} }^2$ and omit $\alpha \En{z^s_k - z^s_{k+1/2}}^2$ instead) for any $z^* \in \cZ^*$: 
    % we sum \eqref{eq:decrease-double-loop} over $k = 0, \ldots, K - 1$ and eliminate the same terms on both sides, and apply the tower property, then we have 
    \begin{align}
        & \EE^s_0\left[ \alpha{\En{ z^{s+1}_0 - z^* }}^2 + K(1-\alpha) {\En{ w^{s+1} - z^* }}^2 \right] \nonumber \\ 
        \leq & \alpha {\En{ z^s_0 - z^* }}^2 + K(1 - \alpha) {\En{ w^s - z^* }}^2 \nonumber \\ 
        & \hspace{80pt} - (1 - \gamma) \bbE^s_0 \left[ \sum_{k=0}^{K-1} (1 - \alpha) \En{ z^s_{k + 1/2} - w^s }^2 + \En{ z^s_{k+1} - z^s_{k+1/2} }^2 \right]. 
    \end{align} 
    % where $e^s = \sum_{k=0}^{K-1} (1 - \alpha)(1 - \gamma) {\En{ z^s_{k + 1/2} - w^s }}^2 + \alpha {\En{ z^s_{k+1/2} - z^s_k }}^2 + (1 - \gamma) \mathbb{E}^s_k{\En{ z^s_{k + 1} - z^s_{k + 1/2} }}^2$. 

    % By Jensen's inequality and definition of $w^{s+1}$ and $z^{s+1}_0$, it yields that  
    % \begin{align}
    %     & \alpha \mathbb{E}^s_0{\En{ z^{s+1}_0 - z^* }}^2 + K(1-\alpha) \mathbb{E}^s_0{\En{ w^{s+1} - z^* }}^2 \le \alpha {\En{ z^s_0 - z^* }}^2 + K(1 - \alpha) {\En{ w^s - z^* }}^2 - e^s. 
    % \end{align} 
    Denote $\hat{\Phi}^s = \alpha \En{ z^s_0 - z }^2 + K(1 - \alpha) \En{ w^s - z }^2$. for any $z \in \dom\, g$. 
    Then, we have 
    \begin{align}
        \EE^s_0 \hat{\Phi}^{s+1}(z^*) \leq \hat{\Phi}^s(z^*) - (1 - \gamma) \bbE^s_0 \left[ \sum_{k=0}^{K-1} (1 - \alpha) \En{ z^s_{k + 1/2} - w^s }^2 + \En{ z^s_{k+1} - z^s_{k+1/2} }^2 \right]. 
    \end{align} 
    Taking expectation $\bbE = \bbE^0_0$, using the tower property of conditional expectations, we obtain that 
    \begin{align}
        \bbE \hat{\Phi}^{s+1}(z^*) \leq \bbE \hat{\Phi}^s(z^*) - (1 - \gamma) \bbE \left[ \sum_{k=0}^{K-1} (1 - \alpha) \En{ z^s_{k + 1/2} - w^s }^2 + \En{ z^s_{k+1} - z^s_{k+1/2} }^2 \right]. 
        \label{eq:iias-descent-inequality-double-loop} 
    \end{align} 

    Summing up the above inequalities over $s = 0, \ldots, S - 1$ with increasing weights $s^q$, we have 
    \begin{align}
        & (1 - \gamma) \sum_{s=0}^{S-1} s^q \sum_{k=0}^{K-1} (1 - \alpha) \bbE \En{ z^s_{k + 1/2} - w^s }^2 + \bbE \En{ z^s_{k+1} - z^s_{k+1/2} }^2 \nonumber \\ 
        \leq & \sum_{s=0}^{S-1} s^q \left( \mathbb{E}\hat{\Phi}^s(z^*) - \mathbb{E}\hat{\Phi}^{s+1}(z^*) \right) \tag{\cref{eq:iias-descent-inequality-double-loop}} \nonumber \\ 
        = & \sum_{s=1}^S \left( s^q - (s-1)^q \right) \bbE \hat{\Phi}^s(z^*) - S^q \mathbb{E} \hat{\Phi}^S(z^*) \nonumber \\ 
        \leq & \hat{\Phi}^0(z^*) \sum_{s=1}^S \left( s^q - (s-1)^q \right) = S^q \hat{\Phi}^0(z^*), 
    \end{align} 
    where the last inequality follows from $\bbE \tilde{\Phi}^s(z^*) \leq \tilde{\Phi}^0(z^*)$ for any $z^* \in \cZ^*$ and $s^q - (s-1)^q \geq 0$ for each $s$. 
\end{proof}

\subsection*{Proof of \cref{lemma:bound-exp-max-inp}}

\begin{proof} 
    We prove this lemma for both the loopless SVRG-EG and the double-loop SVRG-EG in order. 

    (i) For loopless version, we define the sequence $z_{k+1} = z_k + u_{k+1}$. 
    % It is easy to see that $z_k$ is $\mathcal{F}_k$-measurable. 
    Then, it follows that 
    \begin{equation}
        {\En{ z_{k+1} - z }}^2 = {\En{ z_k - z }}^2 + 2\inp{u_{k+1}}{z_k - z} + {\En{ u_{k+1} }}^2. \nonumber 
    \end{equation}

    After multiplying the above identity by $k^q$ and summing the resulting equality over $k = 0, \ldots, K - 1$, we attain 
    \begin{align}
        2 \sum_{k=0}^{K-1} k^q \inp{u_{k+1}}{z - z_k} &= \sum_{k=0}^{K-1} k^q ({\En{ z_k - z }}^2 - {\En{ z_{k+1} - z }}^2) + \sum_{k=0}^{K-1} k^q {\En{ u_{k+1} }}^2 \nonumber \\ 
        &= \sum_{k=1}^K (k^q - (k-1)^q) \En{ z_k - z }^2 - K^q \En{ z_K - z }^2 + \sum_{k=0}^{K-1} k^q \En{ u_{k+1} }^2 % &\leq K^q \max_{z, z' \in \mathcal{Z}} {\En{ z - z' }}^2 + \sum_{k=0}^{K-1} k^q {\En{ u_{k+1} }}^2. 
        \label{eq:bound-sum-inp-loopless} 
    \end{align}

    Since $z_k$ is $\mathcal{F}_k$-measurable, we have 
    $\bbE_k [k^q \inp{u_{k+1}}{z_k}] = k^q \inp{\bbE_k [u_{k+1}]}{z_k} = 0$. 
    By the tower property, $\mathbb{E}[\sum_{k=0}^{K-1} k^q \inp{u_{k+1}}{z_k}] = 0$. 
    Taking maximum over $z \in \cC$ on both sides of \eqref{eq:bound-sum-inp-loopless}, we have 
    \begin{align}
        2 \max_{z \in \cC} \sum_{k=0}^{K-1} k^q \inp{u_{k+1}}{z} & \leq \max_{z \in \cC} \sum_{k=1}^K (k^q - (k-1)^q) \En{ z_k - z }^2 + \sum_{k=0}^{K-1} k^q \En{ u_{k+1} }^2 \nonumber \\ 
        & \leq \sum_{k=1}^K (k^q - (k-1)^q) \max_{z \in \cC} \En{ z_k - z }^2 + \sum_{k=0}^{K-1} k^q \En{ u_{k+1} }^2 \nonumber \\ 
        & \leq \max_{z, z' \in \cC} \En{ z - z' }^2 \sum_{k=1}^K (k^q - (k-1)^q) + \sum_{k=0}^{K-1} k^q \En{ u_{k+1} }^2 \nonumber \\ 
        & = K^q \max_{z, z' \in \cC} \En{ z - z' }^2 + \sum_{k=0}^{K-1} k^q \En{ u_{k+1} }^2,  
    \end{align} 
    Then, by taking expectation $\bbE$ and applying the tower property of conditional expectations, we obtain~\cref{eq:bound-exp-max-inp}. 

    (ii) For double-loop version, we define the sequence $z^s_{k+1} = z^s_k + u^s_{k+1}$. 
    % It is easy to see that $z_k$ is $\mathcal{F}_k$-measurable. 
    Then, by summing this over $k = 0, \ldots, K-1$, it follows that 
    \begin{equation}
        {\En{ z^{s+1}_0 - z }}^2 = {\En{ z^s_0 - z }}^2 + 2\sum_{k=0}^{K-1}\inp{u^s_{k+1}}{z^s_k - z} + \sum_{k=0}^{K-1}{\En{ u^s_{k+1} }}^2. \nonumber 
    \end{equation}

    After multiplying the above identity by $s^q$ and summing the resulting equality over $s = 0, \ldots, S - 1$, we attain 
    \begin{align}
        2\sum_{s=0}^{S-1} s^q \sum_{k=0}^{K-1}\inp{u^s_{k+1}}{z - z^s_k} 
        =& \sum_{s=0}^{S-1} s^q \left( {\En{ z^s_0 - z }}^2 - {\En{ z^{s+1}_0 - z }}^2 \right) + \sum_{s=0}^{S-1} s^q \sum_{k=0}^{K-1}{\En{ u^s_{k+1} }}^2 \nonumber \\ 
        =& \sum_{s=1}^{S} (s^q - (s-1)^q) \En{ z^s_0 - z }^2 - S^q \En{ z^S_0 - z }^2 + \sum_{s=0}^{S-1} s^q \sum_{k=0}^{K-1}\En{ u^s_{k+1} }^2 \nonumber \\
        % \leq& S^q \max_{z, z' \in \mathcal{Z}} \En{ z - z' }^2 + \sum_{s=0}^{S-1} s^q \sum_{k=0}^{K-1}\En{ u^s_{k+1} }^2. 
        \label{eq:bound-sum-inp-double-loop}
    \end{align} 

    Since $z^s_k$ is $\mathcal{F}^s_k$-measurable, we have 
    $\bbE^s_k [s^q \inp{u^s_{k+1}}{z^s_k}] = s^q \inp{\bbE^s_k [u^s_{k+1}]}{z^s_k} = 0$. 
    By the tower property, $\mathbb{E}[\sum_{s=0}^{S-1} s^q \sum_{k=0}^{K-1} \inp{u^s_{k+1}}{z^s_k}] = 0$. 
    Taking maximum over $z \in \cC$ on both sides of \eqref{eq:bound-sum-inp-double-loop}, 
    we have 
    \begin{align}
        2 \max_{z \in \cC} \sum_{s=0}^{S-1} s^q \sum_{k=0}^{K-1}\inp{u^s_{k+1}}{z - z^s_k} 
        & \leq \max_{z \in \cC} \sum_{s=1}^{S} (s^q - (s-1)^q) \En{ z^s_0 - z }^2 + \sum_{s=0}^{S-1} s^q \sum_{k=0}^{K-1}\En{ u^s_{k+1} }^2 \nonumber \\ 
        & \leq \sum_{s=1}^{S} (s^q - (s-1)^q) \max_{z \in \cC} \En{ z^s_0 - z }^2 + \sum_{s=0}^{S-1} s^q \sum_{k=0}^{K-1}\En{ u^s_{k+1} }^2 \nonumber \\ 
        & \leq \max_{z, z' \in \cC} \En{ z - z' }^2 \sum_{s=1}^{S} (s^q - (s-1)^q) + \sum_{s=0}^{S-1} s^q \sum_{k=0}^{K-1}\En{ u^s_{k+1} }^2 \nonumber \\ 
        &= S^q \max_{z, z' \in \cC} \En{ z - z' }^2 + \sum_{s=0}^{S-1} s^q \sum_{k=0}^{K-1}\En{ u^s_{k+1} }^2 
    \end{align}    
    Then, by taking expectation $\bbE$ and applying the tower property of conditional expectations, we obtain the result of this lemma for the double-loop version of SVRG-EG: 
    \begin{equation}
        \bbE \left[ \max_{z \in \cC} \sum_{s=0}^{S-1} s^q \sum_{k=0}^{K-1}\inp{u^s_{k+1}}{z} \right] \le \frac{1}{2} S^q \max_{z, z' \in \cC} {\En{ z - z' }}^2 + \frac{1}{2} \sum_{s=0}^{S-1} s^q \sum_{k=0}^{K-1}\mathbb{E}{\En{ u^s_{k+1} }}^2. 
    \end{equation} 
\end{proof}

\subsection{Loopless SVRG-EG}
\label{subsec:app-iias-loopless}

\subsection*{Proof of \cref{thm:sublinear-convergence-iias-loopless}}

\begin{proof} 
    For any $z \in \cZ = \dom\, g$, 
    let 
    \begin{equation}
        \Theta_{k+1/2}(z) = \inp{F(z_{k+1/2})}{z_{k+1/2} - z} + g(z_{k+1/2}) - g(z). \nonumber 
    \end{equation} 

    Using \cref{eq:inner-product-1,eq:inner-product-2} in \cref{eq:sum-of-two-inequalities}, we have 
    \begin{align} 
        & 2 \tau \Theta_{k+1/2}(z) + \En{z_{k+1} - z}^2 \nonumber \\ 
        \leq & \alpha \En{z_k - z}^2 + (1 - \alpha) \En{w_k - z}^2 + 2 \tau \inp*{F_{\xi_k}(w_k) - F_{\xi_k}(z_{k+1/2})}{z_{k+1} - z_{k+1/2}} \nonumber \\ 
        & \hspace{10pt} - (1 - \alpha) \En{z_{k+1/2} - w_k}^2 - \En{z_{k+1} - z_{k+1/2}}^2 + 2 \tau e_1(z, k), 
        \label{eq:big-theta-bound-inequality}
    \end{align} 
    where $e_1(z, k) = \inp*{F(z_{k+1/2}) - \hat{F}(z_{k+1/2})}{z_{k+1/2} - z}$. 

    Different from \citet{alacaoglu2022stochastic}, we remove the constraint $\alpha = 1 - p$ and derive more general formulas. 
    As prohibited to take expectation after taking maximum, we add $e_2(z, k)$ defined as 
    \begin{equation}
        e_2(z, k) = \En{w_{k+1} - z}^2 - p \En{z_{k+1} - z}^2 - (1 - p) \En{w_k - z}^2, 
    \end{equation}
    which can be eliminated immediately if we take $\mathbb{E}_{k+1/2}$ of it. 
    % We can handle $e_2(z, k)$ using $\ell_2$ norm property and Lemma~\ref{lemma:bound-exp-max-inp}, following \citet[Theorem 2.5]{alacaoglu2022stochastic}. 

    Recall that $\tilde{\Phi}_k(z) = \alpha \En{ z_k - z }^2 + \frac{1 - \alpha}{p} \En{ w_k - z }^2$. 
    With $e_1(z, k)$, $e_2(z, k)$, we can cast \cref{eq:big-theta-bound-inequality} as 
    \begin{align}
        & 2 \tau \Theta_{k+1/2}(z) \nonumber \\ 
        \leq & \tilde{\Phi}_k(z) - \tilde{\Phi}_{k+1}(z) + 2 \tau e_1(z, k) + \frac{1 - \alpha}{p} e_2(z, k) \nonumber \\ 
        & + 2 \tau \inp*{F_{\xi_k}(w_k) - F_{\xi_k}(z_{k+1/2})}{z_{k+1} - z_{k+1/2}} - (1 - \alpha) \En{z_{k+1/2} - w_k}^2 - \En{z_{k+1} - z_{k+1/2}}^2. 
        \label{eq:big-theta-bound-inequality-with-errors}
    \end{align} 

    Note that, by the convexity of the affine function $\inp{F(z)}{\cdot - z}$ and $g$, it follows that 
    \begin{align}
        \EE[\gap(z^K)] &= \EE\left[\gap\left(\frac{1}{Q_K}\sum_{k=0}^{K-1} k^q z_{k + 1/2}\right)\right] \nonumber \\ 
        &\leq \EE\left[ \max_{z \in \cC} \frac{1}{Q_K} \sum_{k=0}^{K-1} k^q \left( \inp{F(z)}{z_{k + 1/2} - z} + g(z_{k+1/2}) - g(z) \right) \right] \tag{Jensen's inequality} \nonumber \\ 
        &\leq \frac{1}{Q_K} \EE\left[ \max_{z \in \cC} \sum_{k=0}^{K-1} k^q \big( \inp{F(z_{k + 1/2})}{z_{k + 1/2} - z} + g(z_{k+1/2}) - g(z) \big) \right] \tag{the monotonicity of $F$} \nonumber \\ 
        &= \frac{1}{Q_K} \EE\left[ \max_{z \in \cC} \sum_{k=0}^{K-1} k^q \Theta_{k+1/2}(z)\right], 
        \label{eq:gap-bounded-by-exp-max} 
    \end{align}
    where $Q_K = \sum_{k=0}^{K-1} k^q$. 

    Then, we take maximum of both sides over $z \in \cC$, sum \eqref{eq:big-theta-bound-inequality-with-errors} over $k = 0, \ldots, K-1$ with weights $k^q$, and then take total expectation to obtain that 
    \begin{align}
        & 2 \tau s_k \EE[\gap(z^K)] \nonumber \\ 
        \leq& 2 \tau \EE\left[ \sum_{k=0}^{K-1} k^q \max_{z \in \cC} \Theta_{k+1/2}(z)\right] \tag{\cref{eq:gap-bounded-by-exp-max}} \nonumber \\ 
        \leq& K^q \left( \alpha + \frac{1-\alpha}{p} \right) \max_{z \in \cC} \En{z_0 - z}^2 + \EE\left[ \max_{z \in \cC} \sum_{k=0}^{K-1} 2 \tau k^q e_1(z, k) \right] + \EE\left[ \max_{z \in \cC} \sum_{k=0}^{K-1} \frac{1 - \alpha}{p} k^q e_2(z, k) \right] \nonumber \\ 
        & \quad + 2 \tau \sum_{k=0}^{K-1} k^q \EE\left[\inp*{F_{\xi_k}(w_k) - F_{\xi_k}(z_{k+1/2})}{z_{k+1} - z_{k+1/2}}\right] \nonumber \\ 
        & \quad - \sum_{k=0}^{K-1} k^q \EE\left[ (1 - \alpha) \En{z_{k+1/2} - w_k}^2 + \En{z_{k+1} - z_{k+1/2}}^2\right], 
    \end{align} 
    where in the second inequality we use 
    % (as in Lemma~\ref{lemma:poly-sum-bound}) 
    \begin{align}
        \sum_{k=0}^{K-1} k^q \left( \bbE \tilde{\Phi}_k(z) - \bbE \tilde{\Phi}_{k+1}(z) \right) = \sum_{k=1}^{K} \left( k^q - (k-1)^q \right) \bbE \tilde{\Phi}_k(z) - K^q \bbE \tilde{\Phi}_K(z) \leq K^q \tilde{\Phi}_0(z). 
    \end{align}

    By \eqref{eq:bound-by-splited-two-norms} and the tower property, we can eliminate the second last two lines. 
    % Also, by negativity of the last line we eliminate it. 
    Then, we have 
    \begin{align}
        2 \tau Q_K \EE[\gap(z^K)] \leq & K^q \big( \alpha + \frac{1-\alpha}{p} \big) \max_{z \in \cC} \En{z_0 - z}^2 \nonumber \\ 
        & + \EE\left[ \max_{z \in \cC} \sum_{k=0}^{K-1} 2 \tau k^q e_1(z, k) \right] + \EE\left[ \max_{z \in \cC} \sum_{k=0}^{K-1} \frac{1 - \alpha}{p} k^q e_2(z, k) \right]. 
        \label{eq:bound-with-errors}
    \end{align}

    As in~\citet{alacaoglu2022stochastic}, we next derive upper bounds for the error terms. 
    % Instead of using their uniform iterate averaging version, we apply IIAS version lemmas. 
    First, for $e_1(z, k)$, let $\mathcal{F}_k = \sigma(\xi_0, \cdots, \xi_{k-1}, w_k)$ and 
    $u_{k+1} = 2 \tau \left( [F_{\xi_k}(z_{k+1/2}) - F_{\xi_k}(w_k)] - [F(z_{k+1/2}) - F(w_k)] \right)$, 
    where for definition we set $\mathcal{F}_0 = \sigma(\xi_0, \xi_{-1}, w_0) = \sigma(\xi_0)$. 
    With this, we obtain the bound 
    \begin{align}
        \EE_k \left[ \max_{z \in \cC} \sum_{k=0}^{K-1} 2 \tau k^p e_1(z, k) \right] &= \EE_k \left[ \max_{z \in \cC} \sum_{k=0}^{K-1} k^q \inp{u_{k+1}}{z} \right] - \EE_k \left[ \sum_{k=0}^{K-1} k^q \inp{u_{k+1}}{z_{k+1/2}} \right] \nonumber \\ 
        &= \EE_k \left[ \max_{z \in \cC} \sum_{k=0}^{K-1} k^q \inp{u_{k+1}}{z} \right] \tag{$z_{k+1/2}$ is $\mathcal{F}_k$-measurable and $\EE_k[u_{k+1}] = 0$} \nonumber \\ 
        &\leq \frac{1}{2} K^q \max_{z, z' \in \cC} {\En{ z - z' }}^2 + \frac{1}{2} \sum_{k=0}^{K-1} k^q \EE_k {\En{ u_{k+1} }}^2 \tag{\cref{lemma:bound-exp-max-inp}} \nonumber \\ 
        &\leq \frac{1}{2} K^q \max_{z, z' \in \cC} \En{ z - z' }^2 + 2 \tau^2 L^2 \sum_{k=0}^{K-1} k^q \En{z_{k+1/2} - w_k}^2, 
        \label{} 
    \end{align} 
    where the last inequality holds because 
    \begin{align}
        \EE_k \En{u_{k+1}}^2 \leq 4\tau^2 \EE_k\En{F_{\xi_k}(z_{k+1/2}) - F_{\xi_k}(w_k)}^2 \leq 4 \tau^2 L^2 \En{z_{k+1/2} - w_k}^2, 
    \end{align}
    where we use $\EE_k \En{X - \EE X}^2 = \EE_k \En{X}^2 - (\EE_k \En{X})^2 \leq \EE_k \En{X}^2$ for some random variable $X$, and the $L$-Lipschitz continuity of $F_\xi$. 
    Taking expectation $\bbE$ and applying the tower property of conditional expectations, we obtain that 
    \begin{equation}
        \bbE \left[ \max_{z \in \cC} \sum_{k=0}^{K-1} 2 \tau k^p e_1(z, k) \right] \leq \frac{1}{2} K^q \max_{z, z' \in \cC} \En{ z - z' }^2 + 2 \tau^2 L^2 \sum_{k=0}^{K-1} k^q \bbE \En{z_{k+1/2} - w_k}^2. 
        \label{eq:bound-error-1}
    \end{equation}

    Secondly, for $e_2(z, k)$, we let $u'_{k+1} = p z_{k+1} + (1 - p) w_k - w_{k+1}$. 
    Note that 
    \begin{align}
        e_2(z, k) &= \En{w_{k+1} - z}^2 - p \En{z_{k+1} - z}^2 - (1 - p) \En{w_k - z}^2 \nonumber \\ 
        &= 2\inp{p z_{k+1} + (1 - p) w_k - w_{k+1}}{z} - p\En{z_{k+1}}^2 - (1 - p)\En{w_k}^2 + \En{w_{k+1}}^2 \nonumber \\ 
        &= 2\inp{u'_{k+1}}{z} - (p\En{z_{k+1}}^2 + (1 - p)\En{w_k}^2 - \En{w_{k+1}}^2). 
    \end{align}

    Since $\EE\left[ \EE_{k+1/2}\left[ \En{w_{k+1}}^2 - p \En{z_{k+1}}^2 - (1-p)\En{w_k}^2 \right] \right] = 0$, we can obtain the bound 
    \begin{align}
        & \EE\left[ \max_{z \in \cC} \sum_{k=0}^{K-1} \frac{1 - \alpha}{p} k^q e_2(z, k) \right] \nonumber \\ 
        =& \frac{2(1 - \alpha)}{p} \EE\left[ \max_{z \in \cC} \sum_{k=0}^{K-1} k^q \inp{u'_{k+1}}{z} \right] \nonumber \\
        \leq& \frac{1 - \alpha}{p} K^q \max_{z, z' \in \cC} {\En{ z - z' }}^2 + \frac{1 - \alpha}{p} \sum_{k=0}^{K-1} k^q \EE {\En{ u_{k+1} }}^2 \tag{\cref{lemma:bound-exp-max-inp}} \nonumber \\ 
        =& \frac{1 - \alpha}{p} K^q \max_{z, z' \in \cC} {\En{ z - z' }}^2 + (1 - \alpha)(1-p) \sum_{k=0}^{K-1} k^q \EE\En{z_{k+1} - w_k}^2, 
        \label{}
    \end{align}
    where the second equality holds because 
    \begin{align}
        \EE_{k+1/2}\En{u'_{k+1}}^2 &= \EE_{k+1/2}\En{\EE_{k+1/2}[w_{k+1}] - w_{k+1}}^2 \nonumber \\ 
        &= \EE_{k+1/2}\En{w_{k+1}}^2 - \En{\EE_{k+1/2}[w_{k+1}]}^2 \nonumber \\ 
        &= p\En{z_{k+1}}^2 + (1-p)\En{w_k}^2 - \En{p z_{k+1} + (1-p) w_k}^2 \nonumber \\ 
        &= p(1-p) \En{z_{k+1} - w_k}^2, 
    \end{align}
    and thus $\bbE \En{u'_{k+1}}^2 \leq p(1-p) \bbE \En{z_{k+1} - w_k}^2$ by the tower property of conditional expectations, where in the second equality we use $\EE\En{X - \EE X}^2 = \EE\En{X}^2 - \En{\EE X}^2$ for some random variable $X$, and in the last equality we use the identity $\En{\alpha a + (1 - \alpha) b}^2 = \alpha \En{a}^2 + (1 - \alpha) \En{b}^2 - \alpha(1 - \alpha)\En{a - b}^2$. 
    Taking expectation $\bbE$ and applying the tower property of conditional expectations, we obtain that 
    \begin{equation}
        \EE\left[ \max_{z \in \cC} \sum_{k=0}^{K-1} \frac{1 - \alpha}{p} k^q e_2(z, k) \right] \leq \frac{1 - \alpha}{p} K^q \max_{z, z' \in \cC} {\En{ z - z' }}^2 + (1 - \alpha)(1-p) \sum_{k=0}^{K-1} k^q \EE\En{z_{k+1} - w_k}^2. 
        \label{eq:bound-error-2}
    \end{equation} 

    Combining \cref{eq:bound-with-errors,eq:bound-error-1,eq:bound-error-2}, we finally arrive at 
    \begin{align}
        & 2 \tau Q_K \EE[\gap(z^K)] \nonumber \\ 
        \leq & K^q \big( \alpha + \frac{1-\alpha}{p} \big) \max_{z \in \cC} \En{z_0 - z}^2 + \EE\left[ \max_{z \in \cC} \sum_{k=0}^{K-1} 2 \tau k^q e_1(z, k) \right] + \EE\left[ \max_{z \in \cC} \sum_{k=0}^{K-1} \frac{1 - \alpha}{p} k^q e_2(z, k) \right] \tag{\cref{eq:bound-with-errors}} \nonumber \\ 
        \leq & K^q \big( \alpha + \frac{1-\alpha}{p} \big) \max_{z \in \cC} \En{z_0 - z}^2 + \frac{1}{2} K^q \max_{z, z' \in \cC} {\En{ z - z' }}^2 + 2 \tau^2 L^2 \sum_{k=0}^{K-1} k^q \EE \En{z_{k+1/2} - w_k}^2 \nonumber \\ 
        & \hspace{60pt} + \frac{1 - \alpha}{p} K^q \max_{z, z' \in \cC} {\En{ z - z' }}^2 + (1-\alpha)(1-p) \sum_{k=0}^{K-1} k^q \EE\En{z_{k+1} - w_k}^2 \tag{\cref{eq:bound-error-1,eq:bound-error-2}} \nonumber \\ 
        \leq & 2 \big( \alpha + \frac{1-\alpha}{p} \big) K^q \max_{z, z' \in \cC} {\En{ z - z' }}^2 + 2 (1 - \alpha) \gamma^2 \sum_{k=0}^{K-1} k^q \EE \En{z_{k+1/2} - w_k}^2 \nonumber \\ 
        & \hspace{20pt} + 2(1-\alpha)(1-p) \sum_{k=0}^{K-1} k^q \EE\En{z_{k+1} - z_{k+1/2}}^2 + 2(1-\alpha)(1-p) \sum_{k=0}^{K-1} k^q \EE \En{z_{k+1/2} - w_k}^2 \tag{$\alpha \geq \frac{1}{2}$, $\tau = \frac{\gamma\sqrt{1 - \alpha}}{L}$, and $\En{a + b}^2 \leq \En{a}^2 + \En{b}^2$} \nonumber \\
        \leq & 2 \big( \alpha + \frac{1-\alpha}{p} \big) K^q \max_{z, z' \in \cC} {\En{ z - z' }}^2 \nonumber \\ 
        & \hspace{80pt} + 2(\gamma^2 + 1) \sum_{k=0}^{K-1} k^q (\mathbb{E}_k[{\En{ z_{k + 1} - z_{k + 1/2} }}^2] + (1 - \alpha){\En{ z_{k + 1/2} - w_k }}^2 ) \nonumber \\ 
        \leq & 2 \big( \alpha + \frac{1-\alpha}{p} \big) K^q \max_{z, z' \in \cC} {\En{ z - z' }}^2 + \frac{2(\gamma^2 + 1)}{1 - \gamma} \big( \alpha + \frac{1-\alpha}{p} \big) K^q \max_{z \in \cC} \En{z_0 - z}^2, 
    \end{align} 
    where in the last inequality we use~\cref{lemma:poly-sum-bound}. 

    Let $z_0 \in \cC$. Then, 
    \begin{equation} 
        2 \tau Q_K \EE[\gap(z^K)] \leq \left( 2 + \frac{2(\gamma^2 + 1)}{1 - \gamma} \right) K^q \left( \alpha + \frac{1-\alpha}{p} \right) \max_{z, z' \in \cC} {\En{ z - z' }}^2. 
        \label{eq:gap-bounded-by-constant} 
    \end{equation} 

    Also, we have, 
    \begin{align}
        Q_K = \sum_{k=0}^{K-1} k^q \geq \int_0^{K-1} x^q dx = \frac{(K-1)^{q+1}}{q+1} &= \frac{(K-1)K^q}{(q+1)} \left(\frac{K-1}{K}\right)^q \nonumber \\ 
        &\geq \frac{(K-1)K^q}{(q+1)} \prod_{m = 1}^q \left(\frac{K-m}{K-m+1}\right) \nonumber \\
        &= \frac{(K-1)K^q}{(q+1)} \cdot \frac{K-q}{K} \ge \frac{(K-q-1)K^q}{(q+1)}. 
        \label{eq:bound-Q-K} 
    \end{align}     
    By combining~\cref{eq:gap-bounded-by-constant,eq:bound-Q-K}, the $O(1/T)$ convergence rate follows. 
\end{proof} 

\subsection*{Proof of \cref{crl:sublinear-convergence-iias-loopless}} 

\begin{proof} 
    To reach expected $\epsilon$-accuracy as measured by the gap function to the solution set, 
    % i.e., $\textnormal{dist}^\theta_{\mathcal{Z}^*}(z_k, w_k) = \epsilon$, 
    we need 
    \begin{equation}
        K - q - 1 \geq \frac{\frac{(q+1)}{\tau} \big( 1 + \frac{\gamma^2 + 1}{1-\gamma} \big) \big( \alpha + \frac{1 - \alpha}{p} \big) \max_{z, z' \in \cC} \En{z - z'}^2}{\epsilon}, 
    \end{equation} 
    where $K$ counts the number of iterations. 
    Recall that $q = 1, 2, 3, \ldots$, $1 - \alpha = p = \frac{2}{N}$, $\gamma = 0.99$ and $\tau = \frac{0.99\sqrt{2}}{\sqrt{N}L}$. 
    Since 
    \begin{align} 
        \frac{(q+1)}{\tau} \big( 1 + \frac{\gamma^2 + 1}{1-\gamma} \big) \big( \alpha + \frac{1 - \alpha}{p} \big) = \frac{(q+1)\sqrt{N}L}{0.99 \sqrt{2}} \big( 1 + \frac{0.99^2 + 1}{0.01} \big) \big( 2 - \frac{2}{N} \big) = \cO(\sqrt{N}L). 
    \end{align} 
    Since we need $2 + p N$ evaluations of $F_\xi$ per iteration, we need $4 \times K$ evaluations of $F_\xi$ to reach expected $\epsilon$-accuracy. 
    Therefore, we attain that the time complexity is $\cO\left( \frac{\sqrt{N}L}{\epsilon} \right)$. 
\end{proof} 

% As $p = \frac{2}{N}$, $\alpha = 1 - p$, $\tau = \frac{\sqrt{p}\gamma}{L}$ and $\gamma = 0.99$, it holds that 
% \begin{align}
%     \EE\left[ \gap(z^K) \right] &\le \frac{(q+1)}{\tau (K-q-1)}\left(1 + \frac{\gamma^2 + 1}{1-\gamma}\right) \left( \alpha + 1 \right) \max_{z, z' \in \mathcal{Z}} \En{z - z'}^2 \nonumber \\ 
%     &\le \frac{\sqrt{N}(q+1)L}{0.99\sqrt{2} (K-q-1)} \cdot 402 \max_{z, z' \in \mathcal{Z}} \En{z - z'}^2.
% \end{align}

% % To reach $\epsilon$-accuracy measured on the gap function, we need $\mathcal{O}\left( \frac{L\sqrt{N}}{\epsilon} + q + 1 \right)$ iterations. 
% % In loopless SVRG-EG, one iteration requires $pN$ (full) evaluation of $F$ and $2$ of $F_\xi$ in expectation, therefore in total we need $pN + 2 = 4$. 
% % Hence, the time complexity is $\mathcal{O}\left( 1 + \frac{L\sqrt{N}}{\epsilon} \right)$. 

\subsection{Double-loop SVRG-EG}
\label{subsec:app-iias-double-loop} 

For double-loop SVRG-EG, we prove the following $O(1/T)$ convergence rate. 
\begin{theorem}
    Let Assumptions $1$-$4$ hold, $K \in \mathbb{N}$, $\alpha \in [\frac{1}{2}, 1)$, and $\tau = \frac{\sqrt{1-\alpha}}{L}\gamma$, for $\gamma \in (0, 1)$. Then, for $q \in \{0, 1, 2, \ldots\}$ and $z^S = \frac{1}{K Q_S} \sum_{s=0}^{S-1} s^q \sum_{k=0}^{K-1} z^s_{k+1/2}$ where $Q_S = \sum_{s=0}^{S-1} s^q$, it follows that 
    \begin{equation}
        \EE\left[ \gap(z^S) \right] \le \frac{(q+1)}{2 \tau K(S-q-1)}D_2 \max_{z, z'} \En{z - z'}^2, 
    \end{equation} 
    where $D_2 = 2 \left( 1 + \frac{\gamma^2}{1 - \gamma} \right) \left( \alpha + K(1 - \alpha) \right)$. 
    \label{thm:sublinear-convergence-iias-double-loop}
\end{theorem}

\begin{corollary}
    Set $K=\frac{N}{2}$, $\alpha = 1 - \frac{1}{K}$ and $\gamma = 0.99$ in~\cref{alg:eg-svrg-double-loop}. 
    The time complexity to reach $\epsilon$-accuracy on the gap function is 
    $\mathcal{O}\left( N + \frac{L\sqrt{N}}{\epsilon} \right)$. 
    \label{crl:sublinear-convergence-iias-double-loop}
\end{corollary}

\subsection*{Proof of \cref{thm:sublinear-convergence-iias-double-loop}} 

\begin{proof} 
    For any $z \in \cZ = \dom\, g$, 
    let 
    \begin{equation}
        \Theta^s_{k+1/2}(z) = \inp{F(z^s_{k+1/2})}{z^s_{k+1/2} - z} + g(z^s_{k+1/2}) - g(z). \nonumber 
    \end{equation} 
    % We start from the key inequality~\eqref{eq:decrease-double-loop} without keeping $\En{z^s_{k+1/2} - z^s_k}$. 
    Analogous to~\cref{eq:big-theta-bound-inequality}, for the double-loop SVRG-EG, we have 
    \begin{align} 
        & 2 \tau \Theta^s_{k+1/2}(z) + \En{z^s_{k+1} - z}^2 \nonumber \\ 
        \leq & \alpha \En{z^s_k - z}^2 + (1 - \alpha) \En{w^s - z}^2 + 2 \tau \inp*{F_{\xi_k}(w^s) - F_{\xi_k}(z^s_{k+1/2})}{z^s_{k+1} - z^s_{k+1/2}} \nonumber \\ 
        & - (1 - \alpha) \En{z^s_{k+1/2} - w^s}^2 - \En{z^s_{k+1} - z^s_{k+1/2}}^2 + 2 \tau e(z, s, k), 
    \end{align} 
    where $e(z, s, k) = \inp*{F(z^s_{k+1/2}) - \hat{F}(z^s_{k+1/2})}{z^s_{k+1/2} - z}$. 

    Summing up the above inequalities over $k = 0, 1, \ldots, K-1$, we have 
    \begin{align} 
        &2 \tau \sum_{k=0}^{K-1} \Theta^s_{k+1/2}(z) + \alpha \En{z^{s+1}_0 - z}^2 + (1-\alpha)\sum_{k=0}^{K-1} \En{z^s_{k+1} - z}^2 \nonumber \\ 
        \leq& \alpha \En{z^s_0 - z}^2 + K (1-\alpha) \En{w^s - z}^2 + 2 \tau \sum_{k=0}^{K-1} e(z, s, k) \nonumber \\ 
        & + 2 \tau \sum_{k=0}^{K-1} \inp*{F_{\xi_k}(w^s) - F_{\xi_k}(z^s_{k+1/2})}{z^s_{k+1} - z^s_{k+1/2}} \nonumber \\ 
        & - (1 - \alpha) \sum_{k=0}^{K-1} \En{z^s_{k+1/2} - w^s}^2 - \sum_{k=0}^{K-1} \En{z^s_{k+1} - z^s_{k+1/2}}^2, 
    \end{align}
    where we rearrange terms and use the definition $z^{s+1}_0 = z^s_K$. 
    By Jensen's inequality and the definition of $w^{s+1}$, it follows that 
    \begin{align}
        &2 \tau \sum_{k=0}^{K-1} \Theta^s_{k+1/2}(z) + \alpha \En{z^{s+1}_0 - z}^2 + K(1-\alpha) \En{w^{s+1} - z}^2 \nonumber \\ 
        \leq& \alpha \En{z^s_0 - z}^2 + K (1-\alpha) \En{w^s - z}^2 + 2 \tau \sum_{k=0}^{K-1} e(z, s, k) \nonumber \\ 
        & + 2 \tau \sum_{k=0}^{K-1} \inp*{F_{\xi_k}(w^s) - F_{\xi_k}(z^s_{k+1/2})}{z^s_{k+1} - z^s_{k+1/2}} \nonumber \\ 
        & - (1 - \alpha) \sum_{k=0}^{K-1} \En{z^s_{k+1/2} - w^s}^2 - \sum_{k=0}^{K-1} \En{z^s_{k+1} - z^s_{k+1/2}}^2. 
        \label{eq:big-theta-bound-inequality-double-loop} 
    \end{align} 

    Recall that $\hat{\Phi}^s(z) =  \alpha \En{z^s_0 - z}^2 + K(1-\alpha) \En{w^s - z}^2$. 
    With this, we can cast~\cref{eq:big-theta-bound-inequality-double-loop} as 
    \begin{align}
        2 \tau \sum_{k=0}^{K-1} \Theta^s_{k+1/2}(z) 
        \leq \hat{\Phi}^s(z) - \hat{\Phi}^{s+1}(z) + 2 \tau \sum_{k=0}^{K-1} e(z, s, k) + \delta^s_k, 
        \label{} 
    \end{align} 
    where $\delta^s_k = 2 \tau \sum_{k=0}^{K-1} \inp*{F_{\xi_k}(w^s) - F_{\xi_k}(z^s_{k+1/2})}{z^s_{k+1} - z^s_{k+1/2}} - (1 - \alpha) \sum_{k=0}^{K-1} \En{z^s_{k+1/2} - w^s}^2 - \sum_{k=0}^{K-1} \En{z^s_{k+1} - z^s_{k+1/2}}^2$. 
    By the monotonicity of $F$, Jensen's inequality, the convexity of $g$, it follows that 
    \begin{align}
        \frac{1}{K} \sum_{k=0}^{K-1} \Theta^s_{k+1/2}(z) &= \frac{1}{K} \sum_{k=0}^{K-1} \inp{F(z^s_{k+1/2})}{z^s_{k+1/2} - z} + g(z^s_{k+1/2}) - g(z) \nonumber \\ 
        &\geq \frac{1}{K} \sum_{k=0}^{K-1} \inp{F(z)}{z^s_{k+1/2} - z} + g(z^s_{k+1/2}) - g(z) \tag{the monotonicity of $F$} \nonumber \\ 
        &\geq \inp{F(z)}{\frac{1}{K} \sum_{k=0}^{K-1} z^s_{k+1/2} - z} + g\left(\frac{1}{K} \sum_{k=0}^{K-1} z^s_{k+1/2}\right) - g(z), \tag{Jensen's inequality}
    \end{align}
    where the last inequality follows from the convexity of the affine function $\inp{F(z)}{\cdot - z}$ and $g$. 
    Denote 
    \begin{equation}
        \Theta^s(z) := \inp{F(z)}{\frac{1}{K} \sum_{k=0}^{K-1} z^s_{k+1/2} - z} + g(\frac{1}{K} \sum_{k=0}^{K-1} z^s_{k+1/2}) - g(z). \nonumber \\ 
    \end{equation} 
    Then, we have 
    \begin{equation}
        2 \tau K \Theta^s(z) \leq \hat{\Phi}^s(z) - \hat{\Phi}^{s+1}(z) + 2 \tau \sum_{k=0}^{K-1} e(z, s, k) + \delta^s_k. 
        \label{eq:big-theta-bound-inequality-with-errors-double-loop}
    \end{equation} 

    Note that, by , it follows that 
    \begin{align}
        \EE[\gap(z^S)] &= \EE\left[\gap\left(\frac{1}{Q_S}\sum_{s=0}^{S-1} s^q \left( \frac{1}{K} \sum_{k=0}^{K-1} z^s_{k+1/2} \right) \right)\right] \nonumber \\ 
        &\leq \EE\left[ \max_{z \in \cC} \frac{1}{Q_S} \sum_{s=0}^{S-1} s^q \big( \inp{F(z)}{\frac{1}{K} \sum_{k=0}^{K-1} z^s_{k+1/2} - z} + g(\frac{1}{K} \sum_{k=0}^{K-1} z^s_{k+1/2}) - g(z) \big) \right] \nonumber \\ 
        &= \frac{1}{Q_S} \EE\left[ \max_{z \in \cC} \sum_{s=0}^{S-1} s^q \Theta^s(z)\right], 
        \label{} 
    \end{align} 
    where $Q_S = \sum_{s=0}^{S-1} s^q$. 

    Then, we sum \eqref{eq:big-theta-bound-inequality-with-errors-double-loop} over $k = 0, \ldots, K-1$ with multiplier $k^q$, 
    take maximum of both sides over $z \in \cC$, and then take total expectation to obtain 
    \begin{align}
        2 \tau K Q_S \EE[\gap(z^S)] &\leq 2 \tau K \EE\left[ \max_{z \in \cC} \sum_{s=0}^{S-1} s^q \Theta^s(z)\right] \nonumber \\ 
        &\leq \sum_{s=0}^{S-1} s^q \left( \hat{\Phi}^s - \hat{\Phi}^{s+1} \right) + 2 \tau \EE\left[ \max_{z \in \cC} \sum_{s=0}^{S-1} s^q \sum_{k=0}^{K-1} e(z, k) \right] \nonumber \\ 
        &\leq S^q \left( \alpha + (1-\alpha) K \right) \max_{z \in \cC}\En{z_0 - z}^2 + 2 \tau \EE\left[ \max_{z \in \cC} \sum_{s=0}^{S-1} s^q \sum_{k=0}^{K-1} e(z, s, k) \right], 
        \label{eq:sum-need-to-be-bounded}
    \end{align} 
    where in the first inequality we use Jensen's inequality and convexity of $g$, 
    in the second inequality we use \eqref{eq:big-theta-bound-inequality-with-errors-double-loop}, 
    in the third inequality we use 
    \begin{equation}
        \sum_{s=0}^{S-1} s^q \left( \hat{\Phi}^s - \hat{\Phi}^{s+1} \right) = \sum_{s=1}^S (s^q - (s-1)^q) \hat{\Phi}^s - S^q \hat{\Phi}^S \le S^q \hat{\Phi}^0(z). 
        \label{eq:upper-bound-weighted-sum-sublinear-double-loop}
    \end{equation}. 

    By \eqref{eq:bound-by-splited-two-norms} and the tower property, we have $\bbE \delta^s_k \leq 0$. 
    % Also, by negativity of the last line we eliminate it. 
    Then, we have 
    \begin{align}
        2 \tau K Q_S \EE[\gap(z^S)] \leq S^q \left( \alpha + (1-\alpha) K \right) \max_{z \in \cC}\En{z_0 - z}^2 + 2 \tau \EE\left[ \max_{z \in \cC} \sum_{s=0}^{S-1} s^q \sum_{k=0}^{K-1} e(z, s, k) \right]. 
        \label{}
    \end{align} 

    Next, we will bound the error bound. 
    For $s = 0, \ldots, S - 1$ and $k = 0, \ldots, K - 1 $, 
    we set $\mathcal{F}^s_k = \sigma(z^0_{1/2}, \ldots, z^0_{K-1/2}, \ldots, z^s_{1/2}, \ldots, z^s_{k+1/2})$, 
    $u^s_{k+1} = 2 \tau [\hat{F}(z^s_{z^s_{k+1/2}}) - F(z^s_{k+1/2})] = 2 \tau [F(w^s) - F_{\xi^s_k}(w^s) - F(z^s_{k+1/2}) + F_{\xi^s_k}(z^s_{k+1/2})]$, 
    then it follows that 
    \begin{align}
        2 \tau \EE\left[ \max_{z \in \cC} \sum_{s=0}^{S-1} s^q \sum_{k=0}^{K-1} e(z, k) \right] 
        =\;& 2 \tau \EE\left[ \max_{z \in \cC} \sum_{s=0}^{S-1} s^q \sum_{k=0}^{K-1} \inp{F(z^s_{k+1/2}) - \hat{F}(z^s_{z^s_{k+1/2}})}{z^s_{k+1/2} - z} \right] \nonumber \\ 
        =\;& \EE\left[ \max_{z \in \cC} \sum_{s=0}^{S-1} s^q \sum_{k=0}^{K-1} \inp{u^s_{k+1}}{z} \right] - \sum_{s=0}^{S-1} s^q \sum_{k=0}^{K-1} \EE \left[ \inp{u^s_{k+1}}{z^s_{k+1/2}} \right] \nonumber \\ 
        =\;& \EE\left[ \max_{z \in \cC} \sum_{s=0}^{S-1} s^q \sum_{k=0}^{K-1} \inp{u^s_{k+1}}{z} \right], 
    \end{align}
    where in the second equality we use $\EE \left[ \inp{u^s_{k+1}}{z^s_{k+1/2}} \right] = 0$, which follows from $\mathcal{F}^s_k$-measurability of $z^s_{k+1/2}$, $\EE^s_k[u^s_{k+1}] = 0$, and the tower property of conditional expectations. 

    Using the double-loop version of \cref{lemma:bound-exp-max-inp,lemma:poly-sum-bound}, we have 
    \begin{align}
        & 2 \tau \EE\left[ \max_{z \in \cC} \sum_{s=0}^{S-1} s^q \sum_{k=0}^{K-1} e(z, s, k) \right] \nonumber \\ 
        \leq & \frac{1}{2} S^q \max_{z, z' \in \cC} {\En{ z - z' }}^2 + \frac{1}{2} \sum_{s=0}^{S-1} s^q \sum_{k=0}^{K-1} 4 \tau^2 \EE\En{F_{\xi^s_k}(z^s_{k+1/2}) - F_{\xi^s_k}(w^s) - [F(z^s_{k+1/2}) - F(w^s)]}^2 \tag{\cref{lemma:bound-exp-max-inp}} \nonumber \\ 
        \leq & \frac{1}{2} S^q \max_{z, z' \in \cC} {\En{ z - z' }}^2 + \frac{1}{2} \sum_{s=0}^{S-1} s^q \sum_{k=0}^{K-1} 4 \tau^2 \EE\En{F_{\xi^s_k}(z^s_{k+1/2}) - F_{\xi^s_k}(w^s)}^2 \tag{$\EE\En{X - \EE X}^2 \le \EE\En{X}^2$} \nonumber \\ 
        \leq & \frac{1}{2} S^q \max_{z, z' \in \cC} {\En{ z - z' }}^2 + 2 \tau^2 L^2 \sum_{s=0}^{S-1} s^q \sum_{k=0}^{K-1} \EE\En{z^s_{k+1/2} - w^s}^2 \tag{$L$-Lipschitz continuity of $F_\xi$} \nonumber \\ 
        \leq & \frac{1}{2} S^q \max_{z, z' \in \cC} {\En{ z - z' }}^2 + \frac{2 \tau^2 L^2}{(1-\alpha)(1-\gamma)} S^q \left( \alpha + K(1 - \alpha) \right) \max_{z, z'} \En{z - z'}^2 \tag{\cref{lemma:poly-sum-bound}} \nonumber \\ 
        = & \frac{1}{2} S^q \max_{z, z' \in \cC} {\En{ z - z' }}^2 + \frac{2 \gamma^2}{1 - \gamma} S^q \left( \alpha + K(1 - \alpha) \right) \max_{z \in \cC} \En{z_0 - z}^2, 
        \label{eq:upper-bound-e-double-loop}
    \end{align}
    where in the last equality we use $\tau = \frac{\gamma\sqrt{1 - \alpha}}{L}$. 
    Combining \eqref{eq:upper-bound-e-double-loop} with \eqref{eq:sum-need-to-be-bounded}, we have 
    \begin{align}
        2 \tau K Q_S \EE[\gap(z^S)] &\leq \left( 1 + \frac{2\gamma^2}{1 - \gamma} \right) S^q \left( \alpha + K(1 - \alpha) \right) \max_{z \in \cC} \En{z_0 - z}^2 + \frac{1}{2} S^q \max_{z, z' \in \cC} {\En{ z - z' }}^2. 
    \end{align} 
    Let $z_0 \in \cC$. Then, 
    \begin{align}
        2 \tau K Q_S \EE[\gap(z^S)] &\leq \left( 2 + \frac{2\gamma^2}{1 - \gamma} \right) S^q \left( \alpha + K(1 - \alpha) \right) \max_{z, z' \in \cC} {\En{ z - z' }}^2. 
    \end{align} 

    Also, we have 
    \begin{align} 
        Q_S = \sum_{s=0}^{S-1} s^q \ge \int_0^{S-1} x^q dx = \frac{(S-1)^{q+1}}{q+1} &= \frac{(S-1)S^q}{(q+1)} \left(\frac{S-1}{S}\right)^q \nonumber \\
        &\ge \frac{(S-1)S^q}{(q+1)} \prod_{m = 1}^q \left(\frac{S-m}{S-m+1}\right) \nonumber \\ 
        &= \frac{(S-1)S^q}{(q+1)} \cdot \frac{S-q}{S} \ge \frac{(S-q-1)S^q}{(q+1)}. 
    \end{align}
    Hence, we obtain the $O(1/T)$ convergence rate. 

\end{proof}

\subsection*{Proof of \cref{crl:sublinear-convergence-iias-double-loop}}

\begin{proof} 
    To reach expected $\epsilon$-accuracy as measured by the gap function to the solution set, 
    % i.e., $\textnormal{dist}^\theta_{\mathcal{Z}^*}(z_k, w_k) = \epsilon$, 
    we need 
    \begin{equation}
        S - q - 1 \geq \frac{\frac{(q+1)}{\tau} \big( 1 + \frac{\gamma^2}{1-\gamma} \big) \frac{ \alpha + K(1 - \alpha) }{K} \max_{z, z' \in \cC} \En{z - z'}^2}{\epsilon}, 
    \end{equation} 
    where $S$ counts the number of epoches (outer loops). 
    Recall that $q = 1, 2, 3, \ldots$, $1 - \alpha = \frac{1}{K} = \frac{2}{N}$, $\gamma = 0.99$ and $\tau = \frac{0.99\sqrt{2}}{\sqrt{N}L}$. 
    Since 
    \begin{align} 
        \frac{(q+1)}{\tau} \big( 1 + \frac{\gamma^2}{1-\gamma} \big) \big( \alpha + K(1 - \alpha) \big) = \frac{(q+1)\sqrt{N}L}{0.99 \sqrt{2}} \big( 1 + \frac{0.99^2}{0.01} \big) \frac{2(2 - \frac{2}{N}}{N} = \cO(\frac{1}{\sqrt{N}}L). 
    \end{align} 
    Since we need $2 K + N$ evaluations of $F_\xi$ per epoch, we need $2N \times S$ evaluations of $F_\xi$ to reach expected $\epsilon$-accuracy. 
    Therefore, we attain that the time complexity is $\cO\left( \frac{\sqrt{N}L}{\epsilon} \right)$. 
\end{proof} 

% As $K = \frac{N}{2}$, $\alpha = 1 - \frac{1}{K} \ge \frac{1}{2}$, $\tau = \frac{\gamma}{\sqrt{K}L}$ and $\gamma = 0.99$, it holds that 
% \begin{align}
%     \EE\left[ \gap(z^S) \right] &\le \frac{(q+1)}{2 \tau K(S-q-1)} \left( \left( 1 + \frac{2\gamma^2}{1 - \gamma} \right) \left( \alpha + K(1-\alpha) \right) + \frac{1}{2} \right) \max_{z, z' \in \mathcal{Z}} \En{z - z'}^2, \nonumber \\ 
%     &= \frac{(q+1)L}{\gamma \sqrt{2N} (S-q-1)}\left( ( 1 + 200\gamma^2 ) ( \alpha + 1 ) + \frac{1}{2} \right) \max_{z, z' \in \mathcal{Z}} \En{z - z'}^2 \nonumber \\ 
%     &\le \frac{(q+1)L}{0.99\sqrt{2N} (S-q-1)} \cdot 402.5 \max_{z, z' \in \mathcal{Z}} \En{z - z'}^2. 
% \end{align}

% To reach $\epsilon$-accuracy measured on the gap function, we need $\mathcal{O}\left( \frac{L}{\sqrt{N} \epsilon} + q + 1 \right)$ epoches. 
% In double-loop SVRG-EG, one epoch requires one (full) evaluation of $F$ and $2K$ of $F_\xi$, therefore in total we need $N + 2K = 2N$. 
% Hence, the time complexity is $\mathcal{O}\left( N + \frac{L\sqrt{N}}{\epsilon} \right)$. 

\section{EXPERIMENTAL DETAILS AND ADDITIONAL NUMERICAL RESULTS} 

\subsection{Game Descriptions and Implementation Details}
\label{sec:app-experiment-descriptions}

In all experiments, 
we use the duality gap $\max_{y \in \mathcal{Y}} f(x, y) - \min_{x \in \mathcal{X}} f(x, y)$ to measure performance. 
To measure the cost of computation, we count the (expected) number of times $F$ is (fully) evaluated. 
In particular, for deterministic EG we count $2N$ in each iteration; 
for double loop EG we count $N + 2K$ in each epoch (outer loop); 
for loopless EG we count $pN + 2$ in each iteration. 

For matrix games, we consider the following classes of games: 
\paragraph{Policeman and burglar game \citep{nemirovski2013mini}.} There are $n$ houses. A burglar chooses a house $i$ associated with wealth $w_i$ to burglarize. A policeman chooses her post near a house $j$, and then the probability she catches the burglar is $\exp(-\theta d(i, j))$. This can be represented as a matrix game \eqref{eq:matrix-game} with payoff matrix 
\begin{equation}
    A_{ij} = w_i (1 - e^{-\theta d(i, j)}), \quad i \in [n], j \in [n]. 
\end{equation}
We generate random instances according to $w_i = \En{ w_i' }, \; w_i' \sim \mathcal{N}(0, 1)$, $\theta=0.8$ and $d(i, j) = \En{ i - j }$. 

For the policeman and burglar game, we implemented the algorithms on an $m = n = 100$ instance generated with random seed $2023$.
We set the total number of gradient calculations to be $8 \times 10^4$ (to see fast convergence of SVRG-EG after many iterations).

\paragraph{Symmetric matrices in \citet{nemirovski2004prox}.} We test two families of symmetric matrices. 
They are given by  
$$A_{ij} = \left( \frac{i + j - 1}{2n - 1} \right)^\alpha$$ 
for $i \in [n], j \in [n]$ and 
$$A_{ij}^\prime = \left( \frac{ \En{ i - j } + 1}{2n - 1} \right)^\alpha$$ 
for $i \in [n], j \in [n]$, 
for some $n \in \mathbb{N}$. 
We generate instances by setting $\alpha = 1$ and $2$, respectively.
For each family, we generate an $m = n = 2000$ instance (since most algorithms converge very fast on smaller instances).
We set the total number of gradient calculations to be $4 \times 10^3$.

\paragraph{Uniformly random matrix games.} We also generate games where each $A_{ij}$ is chosen from $\{0,\ldots,10\}$ uniformly. 
We set $m = n = 1000$ and the total number of gradient calculations to be $2 \times 10^4$.

\paragraph{EFGs} In the Leduc game ($m = n = 337$) we set the total number of gradient calculations to be $1 \times 10^3$. 
In the Search (zero sum) game ($m = 11830, n = 69$) we set the total number of gradient calculations to be $2 \times 10^2$. 

\paragraph{Image segmentation.} In the image segmentation experiment, due to the size of the instances, we run the number of iterations corresponding to only $20$ full gradient calculations.
The size of the problem is $m = n = 22500$.

For the algorithms that we compare to, we next describe the (theoretically correct) baseline constant stepsizes that we will use (sometimes we will use a multiple of the baseline stepsize).
For PDA we use $0.99 / \En{A}_2$ as shown in~\citet{chambolle2011first}; 
for EG we use $0.99 / \En{A}_2$ as shown in~\citet{tseng1995linear}; 
for OOMD with $l_2$ norm, we use $0.5 / \En{A}_2$ as shown in various recent literature.
For OOMD with entropy, the theoretically correct stepsizes for last-iterate convergence are proved under the assumption of a unique Nash equilibrium (which is not satisfied by our instances).
For OOMD with entropy, we always use larger stepsizes than theoretically correct, as it too conservative and leads to very slow convergence.
In particular, we use $1$ in matrix games and $0.125$ in EFGs.

All experiments are implemented on a personal computer and with Python 3.9.12.

\subsection{Additional Experimental Results}
\label{sec:app-additional-results}

\graphicspath{ {./plots/} }

\subsection*{Numerical results on randomly generated matrix} 

\begin{figure}[H]
    % \vskip 0.2in
    \begin{center}
    \includegraphics[width=0.28\textwidth]{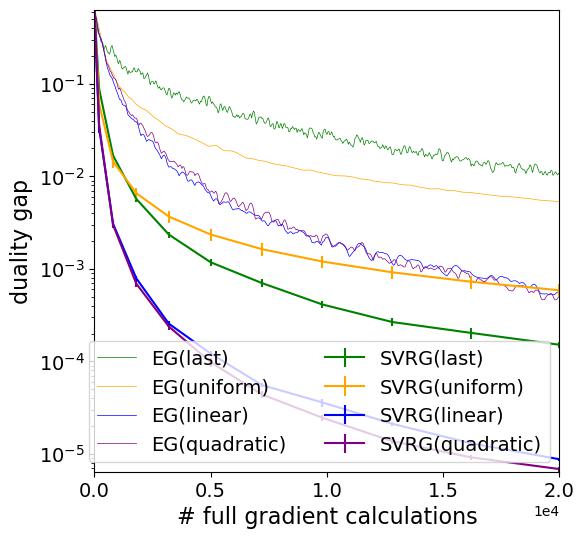}
    \includegraphics[width=0.28\textwidth]{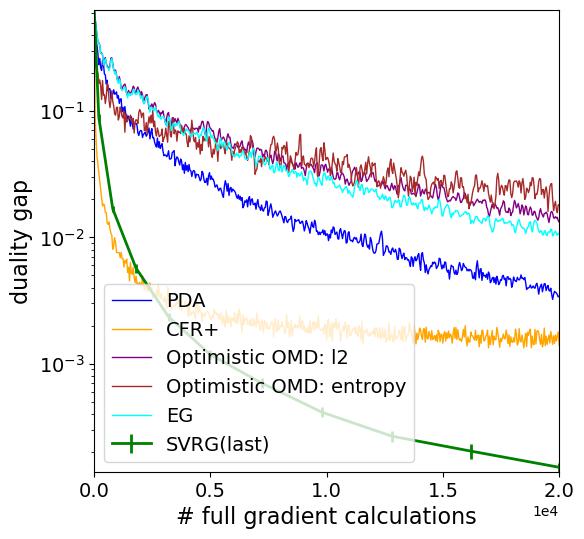}
    \includegraphics[width=0.28\textwidth]{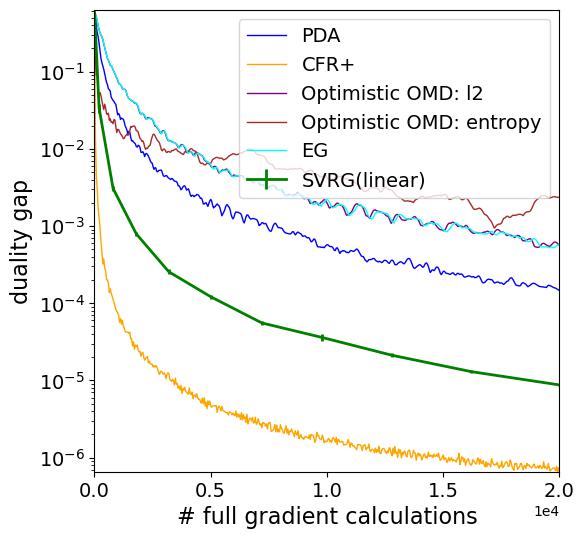}
    \caption{Numerical results on randomly generated matrix. The setup is the same as in \cref{fig:pb}.}
    \label{fig:uni}
    \end{center}
    \vskip -0.2in
\end{figure}

\subsection*{Numerical results on Nemirovski's two symmetric matrices}

We compare (loopless) SVRG-EG with EG (on the left) and other algorithms (the middle shows last-iterate and the right shows linear iterate averaging), 
based on Nemirovski's two symmetric matrix classes. 
The results are shown in~\cref{fig:nem1} and~\cref{fig:nem2}.
To show fast convergence of EG-type methods, we show results of experiments in which all EG-type methods are implemented using larger stepsizes.
In particular, in the first matrix game we use $20 \times$ stepsizes, and in the second matrix game we use $10 \times$ stepsizes.
For both games, all EG-type methods still converge under the larger stepsizes. 
Note that even CFR+ converges fast in last iterates on some of these cases, 
they are not guaranteed to converge (under this stepsize). 

\begin{figure}[H]
    % \vskip 0.2in
    \begin{center}
    \includegraphics[width=0.28\textwidth]{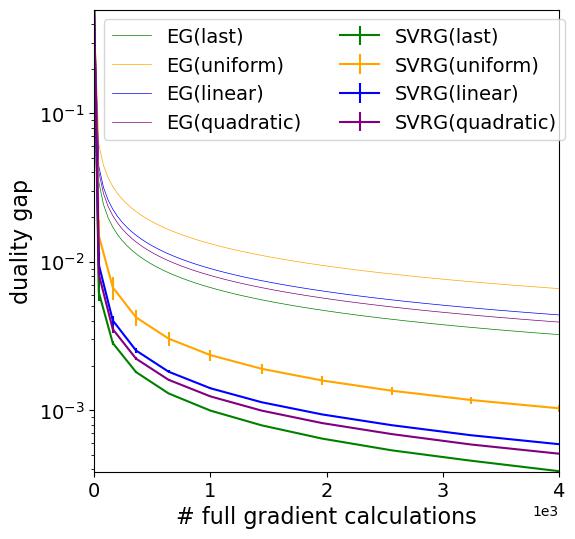}
    \includegraphics[width=0.28\textwidth]{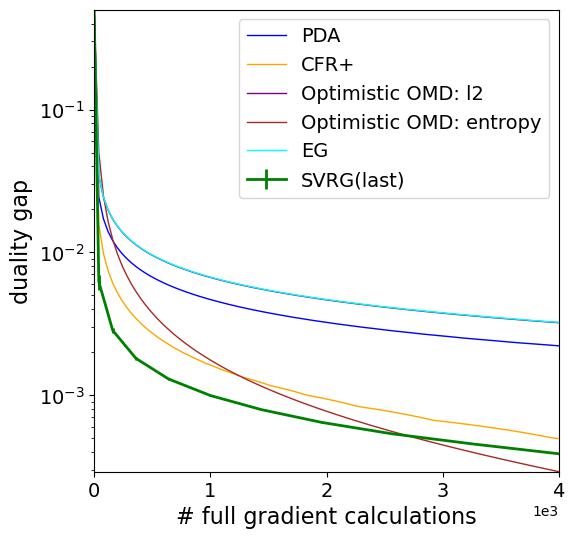}
    \includegraphics[width=0.28\textwidth]{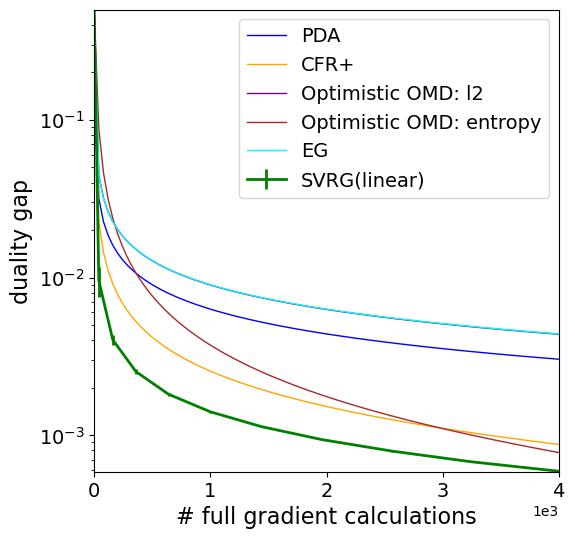}
    \caption{Numerical results on the first symmetric matrix in \citet{nemirovski2004prox}.
    Note that Optimistic OMD (l2 norm) and EG overlap.
    The setup is the same as in \cref{fig:pb}.
    }
    \label{fig:nem1}
    \end{center}
    \vskip -0.2in
\end{figure}

\begin{figure}[H]
    % \vskip 0.2in
    \begin{center}
    \includegraphics[width=0.28\textwidth]{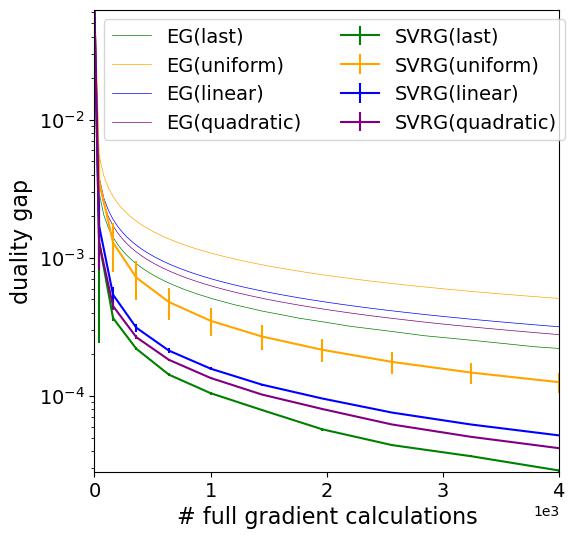}
    \includegraphics[width=0.28\textwidth]{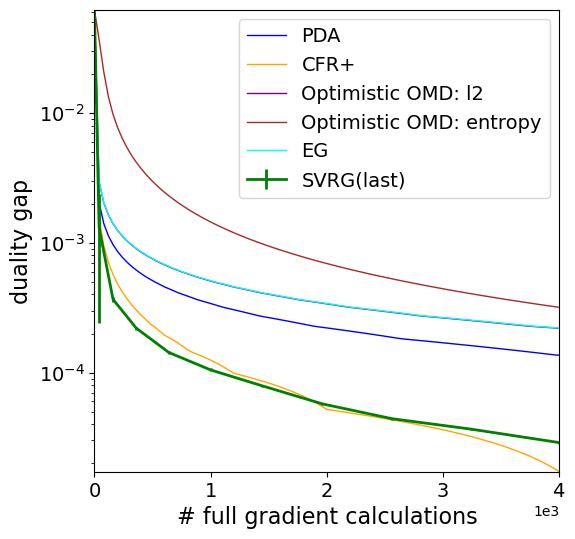}
    \includegraphics[width=0.28\textwidth]{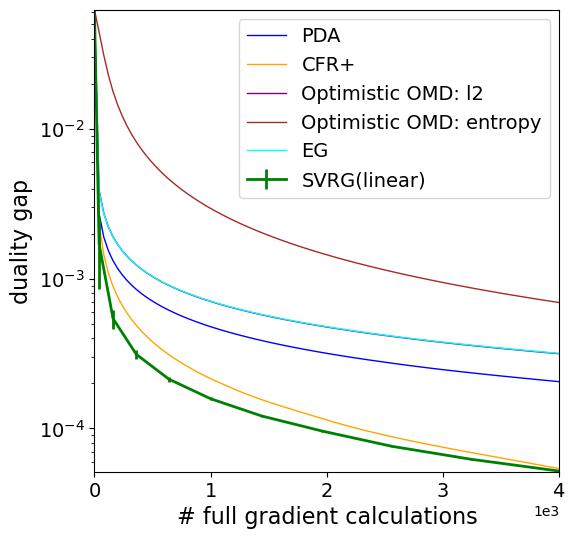}
    \caption{Numerical results on the second symmetric matrix in \citet{nemirovski2004prox}.
    Note that Optimistic OMD (l2 norm) and EG overlap.
    The setup is the same as in \cref{fig:pb}.
    }
    \label{fig:nem2}
    \end{center}
    \vskip -0.2in
\end{figure}

\subsection*{Numerical results on Leduc game}

\begin{figure}[H]
    % \vskip 0.2in
    \begin{center}
        \includegraphics[width=0.28\textwidth]{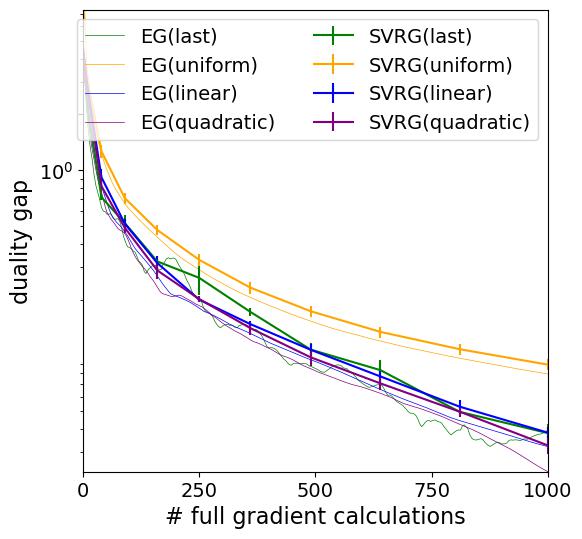}
        \includegraphics[width=0.28\textwidth]{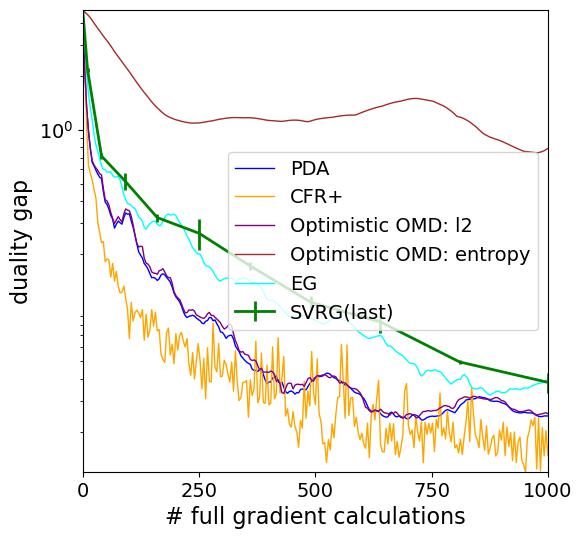}
        \includegraphics[width=0.28\textwidth]{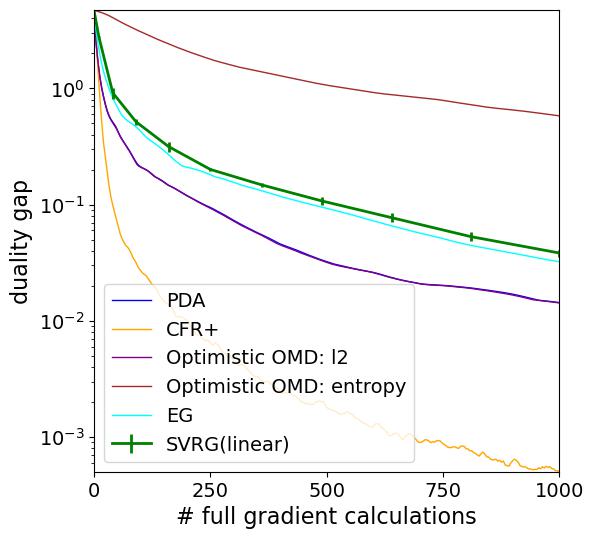}
        \caption{Numerical results on Leduc game. 
        Random algorithms were implemented with seeds 0, 1, 2. 
        The setup is the same as in \cref{fig:pb}.} 
        \label{fig:leduc}
    \end{center}
    \vskip -0.2in
\end{figure}

\subsection*{More results on image segmentation}

To construct the image segmentation SPP, we use the \emph{segmentation.slic} module in Scikit-image (0.19.2) package to initialize the centroids of regions using the $k$ mean method. 
This module segments images using \emph{k-means clustering} in Color-(x,y,z) space.
We show the initial image segmentation result compared to our optimized result, and additional numerical results in~\cref{fig:image-additional}.

\begin{figure}[H]
    % \vskip 0.2in
    \begin{center}
        \begin{minipage}[]{0.63\textwidth}
            \begin{center}
                \includegraphics[width=0.4\textwidth]{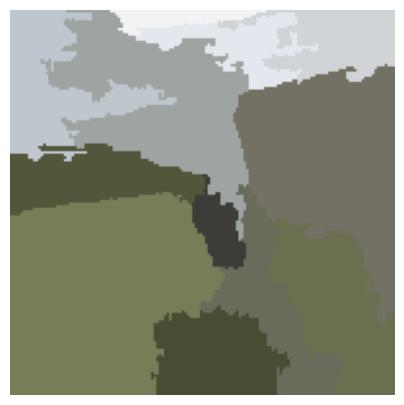}
                \includegraphics[width=0.4\textwidth]{cat-optimized}  
            \end{center}          
        \end{minipage}
        \hspace{-20pt}
        \begin{minipage}[]{0.36\textwidth}
            \begin{center}
                \includegraphics[width=0.78\textwidth]{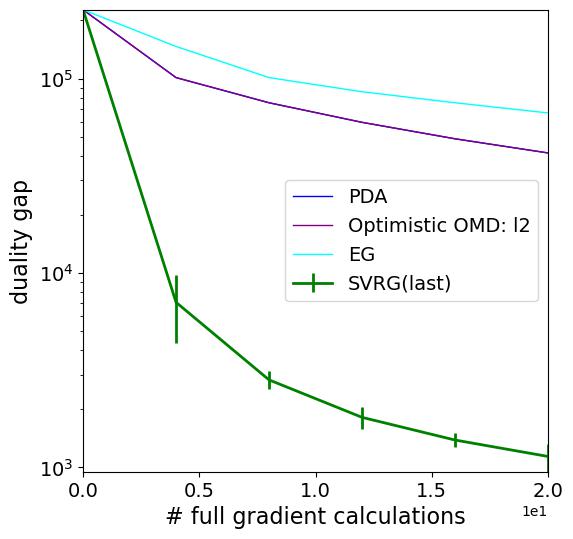}
            \end{center}          
        \end{minipage}
    \caption{Left is the segmentation result from the initial $k$-means method. Middle is the segmentation result from solving~\cref{eq:image-segmentation-spp} via SVRG-EG. 
    Random algorithms are implemented with seeds 0-5. Other setup is the same as in \cref{fig:pb}.
    Right is numerical performance on all applicable algorithms, with last iterate.
    Note that PDA and Optimistic OMD overlap.
    }
    \label{fig:image-additional}
    \end{center}
    \vskip -0.2in
\end{figure}

% \subsection*{Numerical results on conservative OOMD (entropy)}

\subsection*{Numerical results of double-loop SVRG-EG}

We implement experiments using double-loop SVRG-EG as well, and show results as follows.
All setups are the same as in \cref{fig:pb}, except for using double-loop version of SVRG-EG.

\begin{figure}[H]
    % \vskip 0.2in
    \begin{center}
    \includegraphics[width=0.28\textwidth]{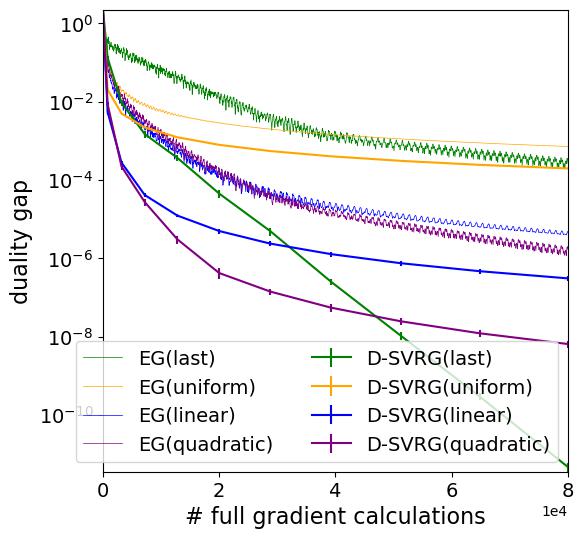}
    \includegraphics[width=0.28\textwidth]{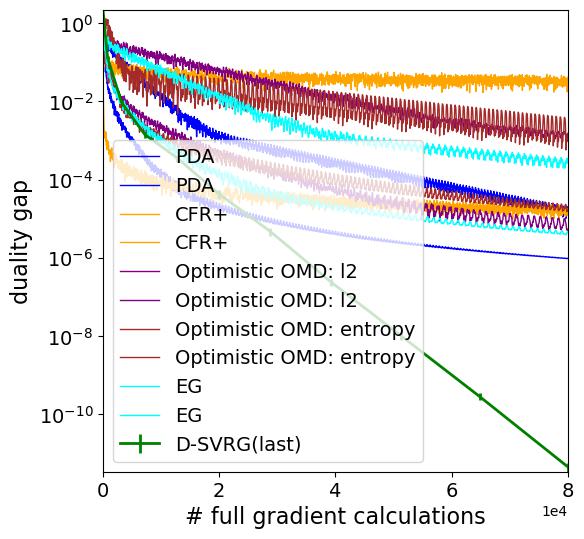}
    \includegraphics[width=0.28\textwidth]{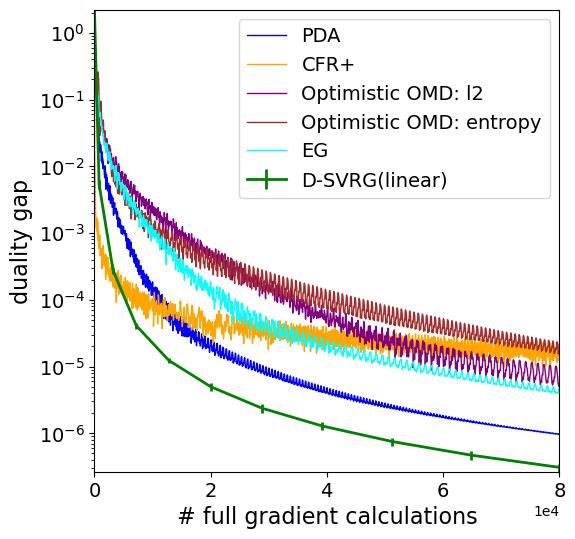}
    \caption{Numerical results on policeman and burglar game (with double-loop SVRG-EG). 
    }
    \label{fig:pb-d}
    \end{center}
\end{figure}

\begin{figure}[H]
    % \vskip 0.2in
    \begin{center}
    \includegraphics[width=0.28\textwidth]{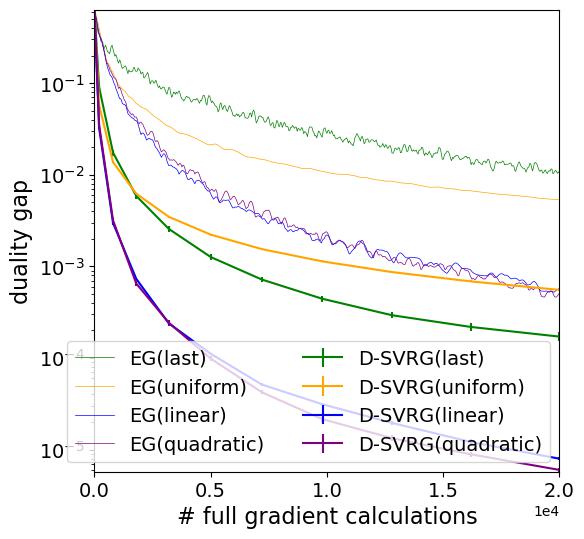}
    \includegraphics[width=0.28\textwidth]{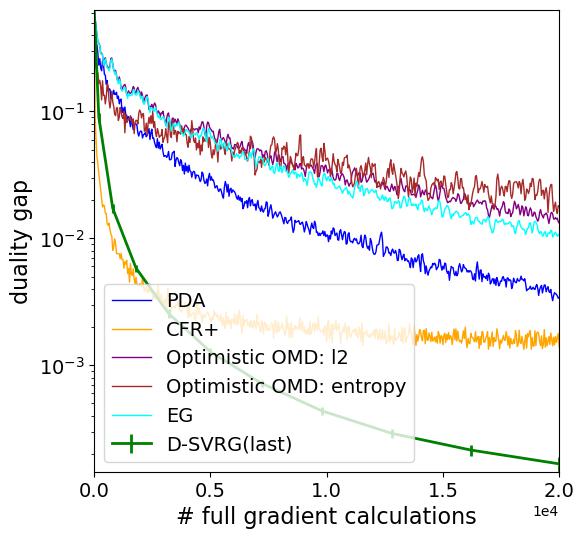}
    \includegraphics[width=0.28\textwidth]{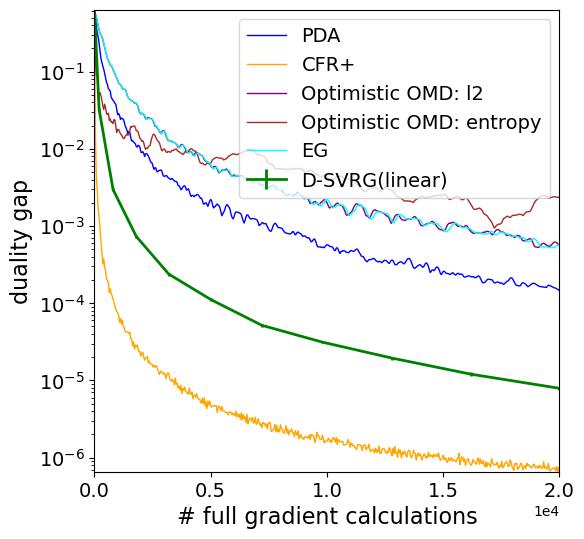}
    \caption{Numerical results on randomly generated matrix (with double-loop SVRG-EG).}
    \label{fig:uni-d}
    \end{center}
    \vspace{-4pt}
\end{figure}

\begin{figure}[H]
    % \vskip 0.2in
    \begin{center}
    \includegraphics[width=0.28\textwidth]{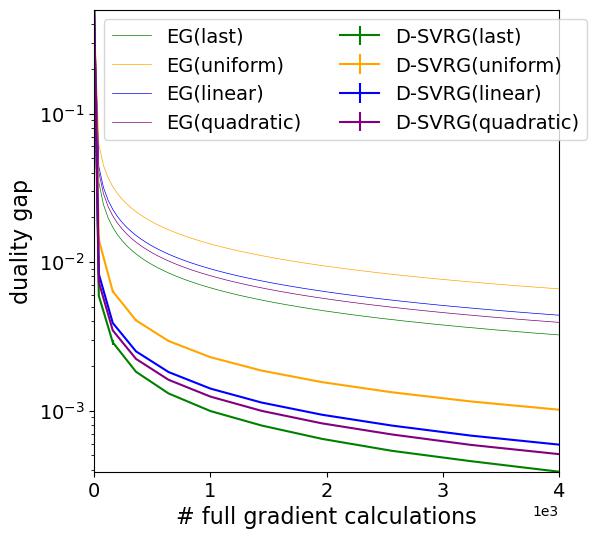}
    \includegraphics[width=0.28\textwidth]{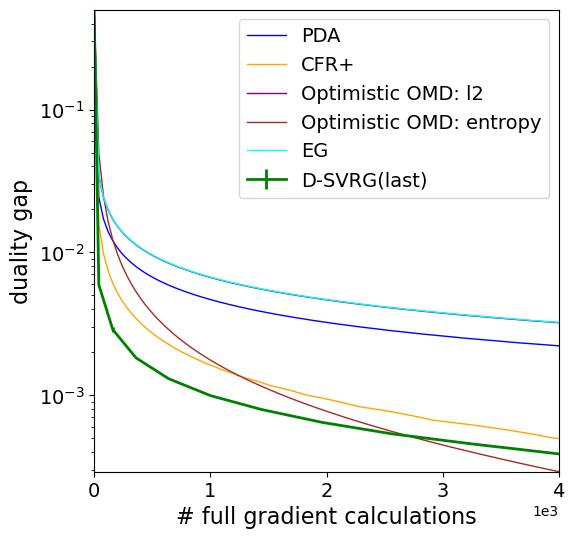}
    \includegraphics[width=0.28\textwidth]{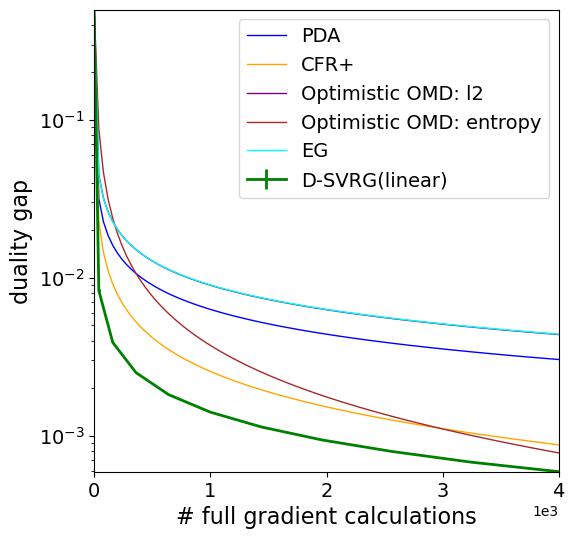}
    \caption{Numerical results on the first symmetric matrix in \citet{nemirovski2004prox} (with double-loop SVRG-EG)
    }
    \label{fig:nem1-d}
    \end{center}
    \vspace{-4pt}
\end{figure}

\begin{figure}[H]
    \begin{center}
    \includegraphics[width=0.28\textwidth]{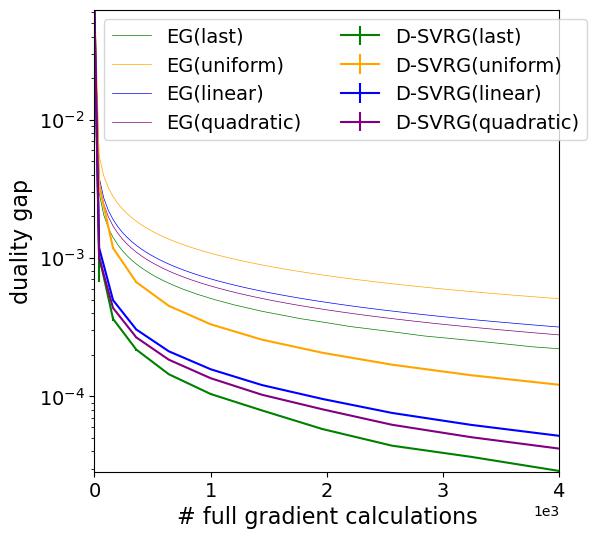}
    \includegraphics[width=0.28\textwidth]{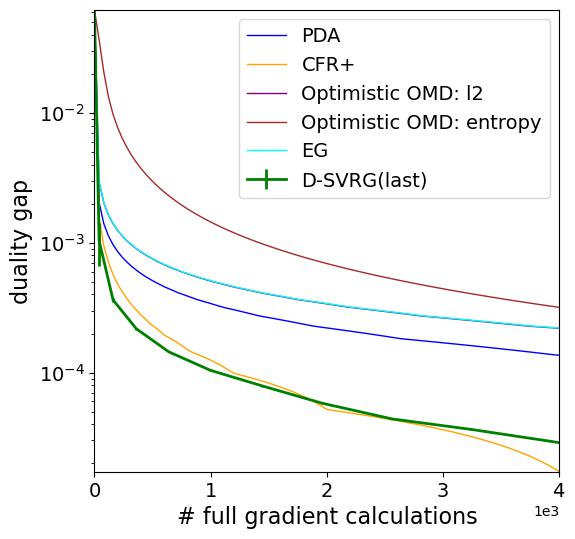}
    \includegraphics[width=0.28\textwidth]{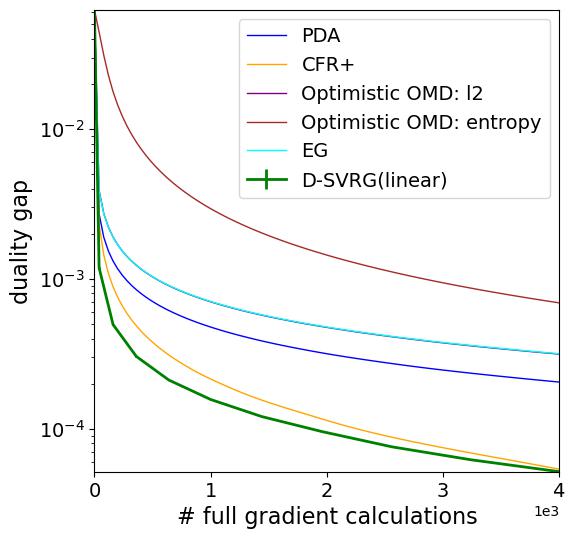}
    \caption{Numerical results on the second symmetric matrix in \citet{nemirovski2004prox} (with double-loop SVRG-EG)
    }
    \label{fig:nem2-d}
    \end{center}
    \vspace{-4pt}
\end{figure}

\begin{figure}[H]
    \begin{center}
        \includegraphics[width=0.28\textwidth]{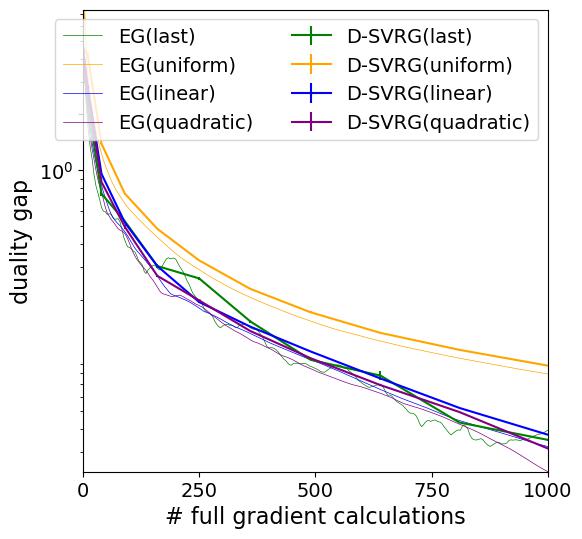}
        \includegraphics[width=0.28\textwidth]{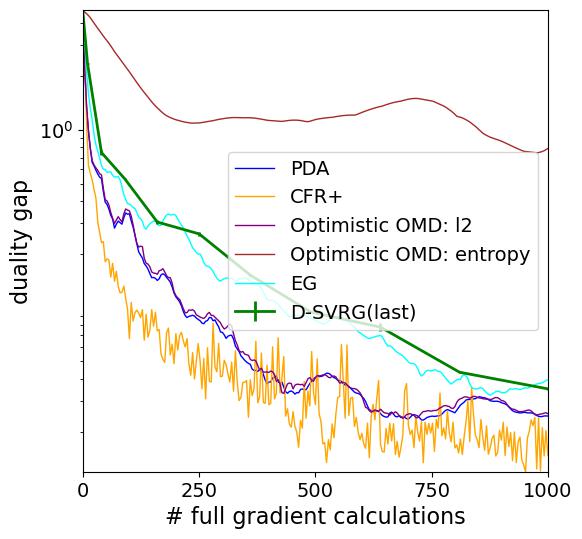}
        \includegraphics[width=0.28\textwidth]{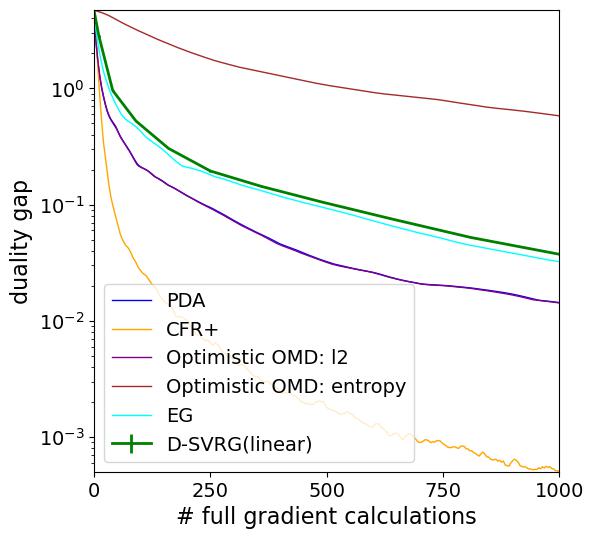}
        \caption{Numerical results on Leduc game (with double-loop SVRG-EG). } 
        \vspace{-2pt}
        \label{fig:leduc-d}
    \end{center}
    \vspace{-4pt}
\end{figure}

\begin{figure}[H]
    \begin{center}
        \includegraphics[width=0.28\textwidth]{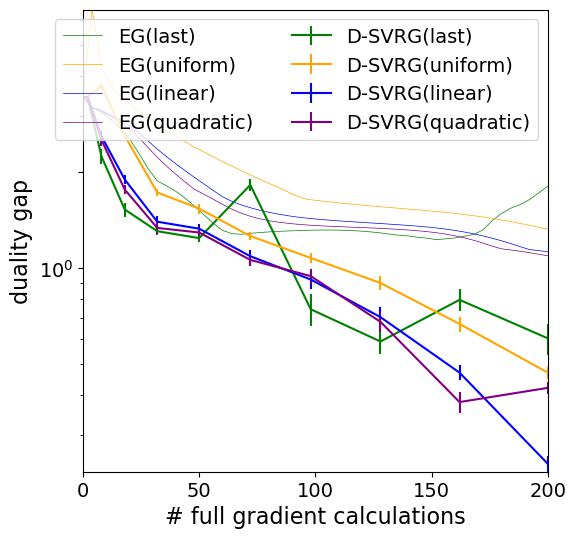}
        \includegraphics[width=0.28\textwidth]{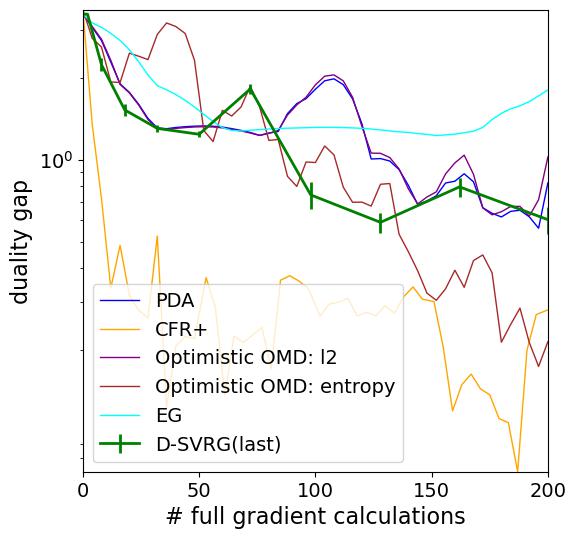}
        \includegraphics[width=0.28\textwidth]{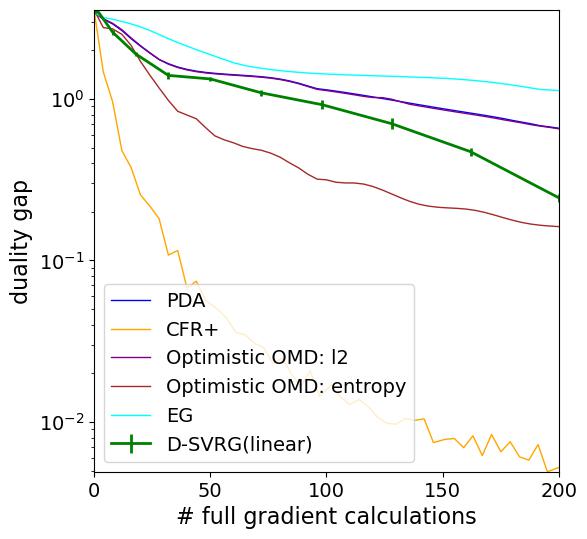}
        \caption{Numerical results on Search (zero sum) game (with double-loop SVRG-EG). } 
        \vspace{-2pt}
        \label{fig:search5-d}
    \end{center}
\end{figure}

\begin{figure}[H]
    % \vskip 0.2in
    \begin{center}
        \includegraphics[width=0.28\textwidth]{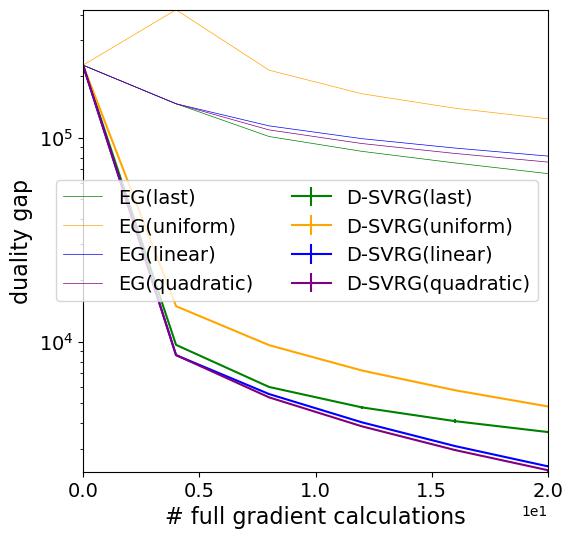}
        \includegraphics[width=0.28\textwidth]{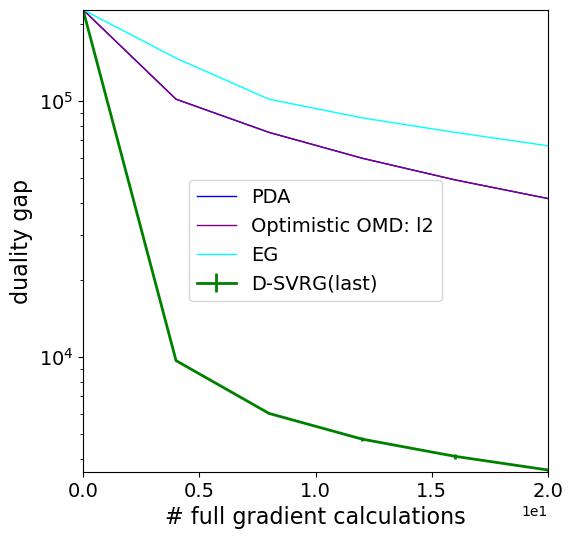}
        \includegraphics[width=0.28\textwidth]{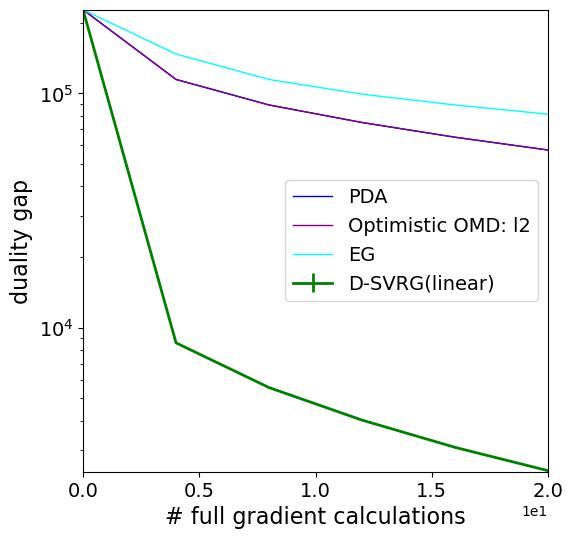}
    \caption{Numerical results on image segmentation. Left is numerical performance of double-loop SVRG-EG compared to EG.
    Middle and right are numerical performance on all applicable algorithms, with last iterate and linear averaging, respectively.
    Note that PDA and Optimistic OMD overlap.
    }
    \label{fig:image-d}
    \end{center}
    \vskip -0.2in
\end{figure}